\documentclass[12pt]{amsart}
\usepackage{amsfonts,amsmath,stmaryrd}
\usepackage{amssymb,latexsym}
\usepackage[dvips]{epsfig}
\usepackage{amscd}
\usepackage{psfrag}
\usepackage{amsthm}

\usepackage[hmargin=3cm,vmargin=3.5cm]{geometry}

\usepackage{graphicx}
\usepackage[all]{xy}
\SelectTips{cm}{}

\renewcommand{\xi}{x_i}
\newcommand{\xp}{x_{i+1}}
\newcommand{\xqq}{x_{i+2}}
\newcommand{\xm}{x_{i-1}}
\newcommand{\xj}{x_j}
\newcommand{\TenR}{\otimes}
\newcommand{\oTen}[1]{\otimes_{#1}}
\newcommand{\Teni}{\otimes_i}

\newcommand{\ot}{\otimes}
\newcommand{\teni}{\otimes}
\newcommand{\tenj}{\otimes}
\newcommand{\tenk}{\otimes}
\newcommand{\tenp}{\otimes}
\newcommand{\tenadj}{\otimes}
\newcommand{\tenfar}{\otimes}

\newcommand{\ii}{\underline{\textbf{\textit{i}}}}
\newcommand{\jj}{\underline{\textbf{\textit{j}}}}
\newcommand{\kk}{\underline{\textbf{\textit{k}}}}

\newcommand{\lra}{\longrightarrow}

\newcommand{\Hom}{{\rm Hom}}

\renewcommand{\to}{\rightarrow}
\newcommand{\maps}{\colon}

\newcommand{\define}{\stackrel{\mbox{\scriptsize{def}}}{=}}

\newcommand{\grdrk}{{\rm grk}}     

\newcommand{\Ztt}{\ensuremath{\mathbbm{Z}\left[t,t^{-1}\right]}}
\newcommand{\Zvv}{\ensuremath{\mathbbm{Z}\left[v,v^{-1}\right]}}

\newcommand{\RmolfR}{\ensuremath{R\mathrm{-molf-}R}}
\newcommand{\RmolfZR}{\ensuremath{R\mathrm{-molf}_{\mathbb{Z}}\mathrm{-}R}}

\newcommand{\mc}[1]{\mathcal{#1}}

\newcommand{\ig}[2]{\vcenter{\xy (0,0)*{\includegraphics[scale=#1]{#2}} \endxy}}
\newcommand{\igv}[2]{\vcenter{\xy (0,0)*{\reflectbox{\includegraphics[scale=#1, angle=180]{#2}}} \endxy}}
\newcommand{\igh}[2]{\vcenter{\xy (0,0)*{\reflectbox{\includegraphics[scale=#1]{#2}}} \endxy}}
\newcommand{\ighv}[2]{\vcenter{\xy (0,0)*{\includegraphics[scale=#1, angle=180]{#2}} \endxy}}
\newcommand{\igrotCW}[2]{\vcenter{\xy (0,0)*{\includegraphics[scale=#1, angle=90]{#2}} \endxy}}

\newcommand{\igc}[2]{\begin{center} \includegraphics[scale=#1]{#2} \end{center}}

\newcommand{\auptob}[2]{\xy  (0,-5)*+{#1}="1"; (0,5)*+{#2}="2"; {\ar@{|->} "1";"2"};\endxy}

\newtheorem*{prop*}{Proposition}
\newtheorem{prop}{Proposition} [section]
\newtheorem{lemma}[prop]{Lemma}
\newtheorem{thm}{Theorem}
\newtheorem{cor}[prop]{Corollary}

\newtheorem{claim}[prop]{Claim}

\theoremstyle{definition}
\newtheorem{defn}[prop]{Definition}
\newtheorem{notation}[prop]{Notation}
\newtheorem{example}[prop]{Example}

\theoremstyle{remark}
\newtheorem{remark}[prop]{Remark}

\numberwithin{equation}{section}

\let\hat=\widehat

\let\phi=\varphi
\let\epsilon=\varepsilon

\usepackage{bbm}
\def\C{{\mathbbm C}}
\def\N{{\mathbbm N}}
\def\R{{\mathbbm R}}
\def\Z{{\mathbbm Z}}

\def\1{\mathbbm{1}}%
\psfrag{Gammaprime}{\large{$\Gamma'$}}
\psfrag{Xi}{$x_i$}
\psfrag{Xp}{$x_{i+1}$}
\psfrag{Xm}{$x_{i-1}$}
\psfrag{Xmip}{$x'$}
\psfrag{Xf}{\large{$f$}}
\psfrag{Xdf}{\large{$\partial_i f$}}
\psfrag{Xpf}{\large{$P_i (f) $}}
\psfrag{p1}{$p_1$}
\psfrag{p2}{$p_2$}
\psfrag{i1}{$\alpha_1$}
\psfrag{i2}{$\alpha_2$}
\psfrag{Zero}{$0$}

\title{Diagrammatics for Soergel categories}
\author{Ben Elias and Mikhail Khovanov}
\date{May 1, 2010}
 
\begin{document} 

\maketitle
\baselineskip 14pt
 
\vspace{0.1in} 

\setcounter{tocdepth}{1}
\tableofcontents
      
\vspace{0.1in} 

\begin{abstract} The monoidal category of Soergel bimodules can be thought of as a 
categorification of the Hecke algebra of a finite Weyl group. We present this
category, when the Weyl group is the symmetric group, in the language of 
planar diagrams with local generators and local defining relations.
\end{abstract}

\newpage 

\section{Introduction} 

In the paper~\cite{Soe1} Soergel gave a combinatorial description of a certain category of Harish-Chandra
bimodules over a simple Lie algebra $\mathfrak{g}$. This category was and continues to be of primary interest
in the infinite-dimensional representation theory of simple Lie algebras. Soergel discovered a functor from
this category to a full subcategory of bimodules over a certain ring $R$, the objects of which are now commonly
called \emph{Soergel bimodules}. The category of Soergel bimodules is additive and monoidal, unlike the
original category which is abelian, but it still has sufficient information to describe the original category. Soergel constructed an
isomorphism between the Grothendieck ring of his category and the integral form of the Hecke algebra of the
Weyl group $\mc{W}$ of $\mathfrak{g}$. Hence, Soergel's construction provides a categorification of the Hecke
algebra.

Given a $k$-dimensional $\C$-vector space $V$ and a generic $q\in \C$, there are commuting actions
of the quantum group $U_q(\mathfrak{sl}_k)$ and the Hecke algebra $\mc{H}$ of the symmetric group
$S_n$ on $V^{\otimes n}$. These actions turn the quotient $U_q(\mathfrak{sl}_k)/J_1$ of the quantum
group and $\mc{H}/J_2$ of the Hecke algebra by the kernels of these action into a dual pair. A
categorical realization of the triple $(V^{\otimes n}, U_q(\mathfrak{sl}_k)/J_1, \mc{H}/J_2)$ was
given by Grojnowski and Lusztig~\cite{GL} via categories of perverse sheaves on products of flag and
partial flag varieties, also see~\cite{BLM, Gr, BFK, FKS}.
  
Many foundational ideas about categorification were put forward by Igor Frenkel in the early 90's (a
small fraction of these ideas formed a part of the paper~\cite{CF}). In particular, Frenkel
conjectured~\cite{Fr} that quantum groups and not just their finite-dimensional quotients
$U_q(\mathfrak{sl}_k)/J_1$ can be categorified. These conjectures remained open until recently, when
categorifications of quantum $\mathfrak{sl}_2$ and $\mathfrak{sl}_k$ were discovered in~\cite{AL2}
and~\cite{KL3}, with a related but different approach developed in~\cite{CR,Rou3}.  In the
categorifications~\cite{AL2, KL3} of quantum groups, 2-morphisms are given by linear combinations of
planar diagrams, modulo local relations.

The parallel objective would be to categorify the Hecke algebra in the same spirit, using planar diagrams.
Soergel had already provided a categorification, so it remains to ask whether his category can be rephrased
diagrammatically. Diagrammatics should also provide a presentation of the category by generators and relations.
A similar question was recently posed by Libedinsky~\cite{Lib3}, who essentially produced such a description
for categorifications of Hecke algebras associated to ``right-angled'' Coxeter systems.

Here we answer this question positively in the case of the Hecke algebra associated to the symmetric group.
This is, of course, the Hecke algebra that appears in the Schur-Weyl duality for $V^{\otimes n}$. For
notational convenience, we use $n+1$, not $n$, as our parameter and define a diagrammatical version $\mc{DC}$
of the category of Soergel bimodules that categorifies the Hecke algebra of the symmetric group $S_{n+1}$.

In some sense, diagrammatic categorifications are very ``low-tech,'' in that they can be described easily and do not rely on
heavy machinery. While one can prove that Soergel bimodules categorify the Hecke algebra using only elaborate commutative
algebra (see \cite{Soe4}, although it is never stated explicitly), showing that indecomposable bimodules descend to the
Kazhdan-Lusztig basis of the Hecke algebra utilizes Kazhdan-Lusztig theory~\cite{KaLu1, KaLu2}. This, in turn, is related to
fundamental developments in geometric representation theory like D-modules on flag manifolds~\cite{BeBe, BK} and perverse
sheaves~\cite{BBD}. One hopes that a diagrammatic approach will help to visualize and work with these sophisticated
constructions, in the same way that the categorifications of quantum groups \cite{AL2,KL3} have led to an improved
understanding of perverse sheaves on quiver varieties (see \cite{VV}). One also hopes that this approach can shed light on
categorifications of representations of the Hecke algebra coming from the context of category $\mc{O}$ (see \cite{MS}).

We start with an intermediate category $\mc{DC}_1$ whose objects are finite sequences $\ii=i_1\dots
i_d$ of numbers between $1$ and $n$.  An object is represented graphically by marking $d$ points in
the standard position (say, having coordinates $1, \dots, d$) on the $x$-axis and assigning labels
$i_1, \dots, i_d$ to marked points from left to right. Morphisms between $\ii$ and $\jj$ are given
by linear combinations (with coefficients in a ground field $\Bbbk$) of planar diagrams embedded in
the strip $\R\times [0,1]$.  These diagrams are decorated planar graphs, where edges may extend to
the boundary $\R\times \{0,1\}$.  Each edge carries a label between $1$ and $n$, so that the induced
labellings of the lower and upper boundaries are $\ii$ and $\jj$, respectively. In the interior of
the strip, we allow
\begin{itemize}
\item vertices of valence 1, 
\item vertices of valence 3 with all 3 edges carrying the same label, 
\item vertices of valence 4 seen as intersections of $i$ and $j$-labelled lines with $|i-j|>1$, 
\item vertices of valence 6 with the edge labelling $i,i+1,i,i+1,i,i+1$, reading clockwise 
around the vertex. 
\item boxes labelled by numbers between $1$ and $n+1$ which float in 
the regions of the graph.
\end{itemize} 
We impose a set of local relations on linear combinations of these diagrams, including 
invariance of diagrams under all isotopies.  A subset of the relations says that 
$i$ is a Frobenius object in the category $\mc{DC}_1$.

The space of morphisms in $\mc{DC}_1$ between $\ii$ and $\jj$ is naturally a graded vector
space. Allowing grading shifts and direct sums of objects, then restricting to grading-preserving
morphisms, and finally passing to the Karoubian closure of the category results in a graded
$\Bbbk$-linear additive monoidal category $\mc{DC}$. Our main result (Theorem \ref{mainthm} in
Section~\ref{sec-conseq}) is an explicit equivalence between this category and the category
$\mc{SC}$ of Soergel bimodules.

The category $\mc{DC}$ is monoidal, and can be viewed as a 2-category with a single object. It may be easier to
tackle the diagrammatics after reading an introductory reference on diagrammatics for 2-categories. Such an
introduction can be found in \cite{AL2}. This may make it easier to explore similarities with the
categorifications of quantum groups in~\cite{KL3, AL2}, where regions of diagrams are labelled by integers
in~\cite{AL2} and integral weights of $\mathfrak{sl}_n$ in~\cite{KL3}. Boxes floating in regions are
superficially analogous to floating bubbles of~\cite{KL3, AL2}. Unlike the diagrammatic categorifications
in~\cite{KL3, AL2}, our lines don't carry dots and are not oriented.

There is another way to view our diagrammatics, which is not developed in this paper. Rouquier~\cite{Rou1,
Rou2} defined an action of the Coxeter braid group associated to $\mc{W}$ on the category of complexes of
Soergel bimodules up to homotopies, which is related to a braid group action using Harish-Chandra bimodules
that had been known for some time. These complexes were later used in an alternative
construction~\cite{Ktriply} of a triply-graded link homology theory~\cite{KR} categorifying the HOMFLY-PT
polynomial~\cite{DGR, KR, JR, W}. In this approach, a product Soergel bimodule $B_{i_1}\otimes \dots \otimes
B_{i_d}$ is depicted by a planar diagram given by concatening elementary planar diagrams lying in the
$xy$-plane that consist of $n+1$ strands going up, with $i$ and $i+1$-st strands merging and splitting,
see~\cite[Figures 2,3]{Ktriply}. Morphisms between product bimodules can be realized by linear combinations of
foams -- decorated two-dimensional CW-complexes embedded in $\R^3$ with suitable boundary conditions. Foams
have been implicit throughout papers on triply-graded link homology (see~\cite{MSV2} for instance, where
various arrows between planar diagrams can be implemented by foams), and explicitly appear in the papers on
their doubly-graded cousins, see~\cite{MSV1, MN} and references therein.

Foams are 3-dimensional objects -- they are two-dimensional CW-complexes embedded in $\R^3$ that
produce cobordisms along the $z$-axis direction between planar objects corresponding to product
Soergel bimodules.  The planar diagrams of our paper are two-dimensional encodings of these foams,
essentially projections of the foams onto the $yz$-plane along the $x$-axis.

It was shown in~\cite{KhT} that the action of the braid group on the homotopy category of Soergel
bimodules extends to a (projective) action of the category of braid cobordisms.  Thus, the homotopy
category of $\mc{DC}$ produces invariants of braid cobordisms, so that our planar diagrammatics
carry information about four-dimensional objects. This informational density indicates the efficiency 
of such encodings.

{\bf Addendum} Since this paper first appeared, the diagrammatics developed here have led to several
developments which we briefly mention here. In \cite{EKr} it is shown that Rouquier's braid group action lifts
functorially to the braid cobordism category. In \cite{Vaz,MV}, a functor is given from the category
$\mc{DC}$ to categories of foams used in link homology. Together, these papers show that the encodings
mentioned above are more than simply heuristic. In \cite{E}, the Temperley-Lieb algebra is categorified as a
quotient of $\mc{DC}$. Additional statements relating this paper to either newer papers or to previous versions
of this paper are found sparsely under a similar ``Addendum" heading.

{\bf Acknowledgments.} The authors were supported by the NSF grants DMS-0706924 
and DMS-0739392. They are grateful to Geordie Williamson and Catharina Stroppel for remarks and suggestions
on an earlier version of this paper. 

\section{Preliminaries}
\label{sec-preliminaries}

Henceforth we will fix a positive integer $n$. Indices $i$, $j$, and $k$ will range over $1, \ldots, n$ if not
otherwise specified. Finite ordered sequences of such indices (allowing repetition) will be denoted
$\ii=i_1\ldots i_d$, as well as $\jj$ and $\kk$. The length of the sequence will be denoted $d=d(\ii)$. For
sequences of length $d=1$ where the single entry is $i$, we use $i$ and $\ii$ interchangeably. Occassionally
$i+1$ will also be used as an index, and whenever this occurs we make the tacit assumption that $i \le n-1$ so
that all indices used remain between 1 and $n$. The same goes for $i-1$, $i+2$ and the like. We denote the
length 0 sequence by the empty set symbol $\emptyset$.

We work over a field $\Bbbk$, usually assuming that Char $\Bbbk \ne 2$, and sometimes 
specializing it to $\C$.

Given a noetherian ring $R$, the category $\RmolfR$ is the full subcategory of $R$-bimodules
consisting of objects which are finitely generated as left $R$-modules. If $R$ is graded, the
category $\RmolfZR$ is the analogous subcategory of graded $R$-bimodules and grading-preserving 
homomorphisms. 

%
\subsection{Hecke algebra}
\label{subsec-hecke}
%

Let $(\mc{W},\mc{S})$ be a Coxeter system of a finite Weyl group $\mc{W}$, with length function $l
\colon \mc{W} \to \N=\{0,1,2,\dots \}$, and $e \in \mc{W}$ the identity.  The Hecke algebra $\mc{H}$
is an algebra over $\Zvv$ (we follow Soergel's use~\cite{Soe5} of the variable $v$; related
variables are denoted in the literature by $t=v^{-1}$ and $q=t^2$), which is free as a module with
basis $T_w,\ w\in\mc{W}$. Multiplication in this basis is given by $T_vT_w=T_{vw}$ when
$l(v)+l(w)=l(vw)$, and $T_s^2=(v^{-2}-1)T_s+v^{-2}T_e$ for $s \in \mc{S}$. $T_e$ is the identity element
in $\mc{H}$ and will often be written as $1$.

In the case we are interested in presently, $\mc{W}=S_{n+1}$, and $\mc{S}$ consists of the
transpositions $s_i=(i,i+1)$ for $i=1, \ldots, n$. The element $T_{s_i}$ will be denoted $T_i$. The
Hecke algebra has a presentation over $\Zvv$, being generated by $T_i$ subject to the relations
\begin{eqnarray*}
T_i^2 & = & (v^{-2}-1)T_i + v^{-2} \\
T_iT_j & = & T_jT_i \ \mathrm{for}\  |i-j|\ge 2 \\
T_iT_{i+1}T_i & = & T_{i+1}T_iT_{i+1}.
\end{eqnarray*}
Clearly then, $\mc{H}$ is also generated as an algebra by $b_i\define
C'_{s_i}=v(T_i+1)$, $1\le i \le n$, and the relations above transform into
\begin{eqnarray}
b_i^2 & = & (v+v^{-1})b_i \label{eqn-bisq}\\
b_ib_j & = & b_jb_i \ \mathrm{for}\  |i-j|\ge 2 \label{eqn-bibj}\\
b_ib_{i+1}b_i + b_{i+1} & = & b_{i+1}b_ib_{i+1} + b_i \label{eqn-bibpbi}.
\end{eqnarray}
We often write the monomial $b_{i_1}b_{i_2}\cdots b_{i_d}$ as $b_{\ii}$ where
$\ii=i_1\ldots i_d$. Notice that $b_{\emptyset}=1$.

Let $a\longmapsto \overline{a}$ be the involution of $\Zvv$ determined by 
$\overline{v}=v^{-1}$. It extends to an involution of $\mc{H}$ given by 
\begin{equation*} \overline{\sum a_w T_w} = \sum \overline{a_w} T_{w^{-1}}^{-1}
\end{equation*} 
In particular, $\overline{T_i} = T_i^{-1}= v^2T_i + v^2 -1.$ 

Kazhdan and Lusztig~\cite{KaLu1} defined a pair of bases $\{C_w\}_{w\in \mc{W}}$ and $\{C'_w\}_{w\in \mc{W}}$
for $\mc{H}$, which immediately proved to be of fundamental importance for representation theory and
combinatorics. The two bases are related via a suitable involution of $\mc{H}$, and the elements of the second
Kazhdan-Lusztig basis $\{C'_w\}_w$ are determined by the two properties: \begin{eqnarray*} \overline{C'_w} & =
& C'_w, \\ C'_w & = & v^{l(w)} \sum_{y \le w} P_{y,w} T_y, \end{eqnarray*} where $P_{y,w}\in \Z[v^{-2}]$ has
negative $v$-degree strictly less than $l(w)-l(y)$ for $y<w$ and $P_{w,w}=1$. There is no simple formula
expressing $C'_w$ in terms of $T_y$, but observe that $C'_e=1$ and $C'_{s_i}=b_i=v(T_i+1)$. For a good
introduction to the Kazhdan-Lusztig basis, see \cite{Soe5}.

Let $\epsilon: \mc{H}\lra \Zvv$ be the $\Zvv$-linear map given by $\epsilon(T_e)=1$ and $\epsilon(T_w)=0$ if $w
\ne e$. Thus, $\epsilon$ simply picks up the coefficient of $T_e$ in $x$. The easily checked property
$\epsilon(T_i x) = \epsilon (x T_i)$ for any $ x\in \mc{H}$ implies that $\epsilon(xy) = \epsilon(y x)$,
$\forall x,y\in \mc{H}$, so that $\epsilon$ is a trace map and turns $\mc{H}$ into a symmetric
Frobenius $\Zvv$-algebra.

Denote by $\omega$ a $v$-antilinear antiinvolution $\omega \colon \mc{H} \to \mc{H}$ defined uniquely by
$\omega(b_i)=b_i$. The antiinvolution and $v$-antilinearity conditions say that $\omega(xy)=\omega(y)\omega(x)$
and $\omega(ax) = \overline{a} \omega(x)$, for $x,y\in \mc{H}$ and $a\in \Zvv$.

Consider the pairing $(,) \colon \mc{H} \times \mc{H} \to \Zvv$ of $\Z-$modules
given by 
$$ (x,y) =\epsilon(\omega(x)y) .$$ 

It satisfies the following properties:
\begin{enumerate}
\item The pairing is \emph{semi-linear}, i.e.  $(ax,y)=\overline{a}(x,y)$ while 
$(x,ay)=a(x,y)$, for $a\in \Zvv$.
\item $b_i$ is self-adjoint, i.e. $(x,b_iy)=(b_ix,y)$ and $(x,yb_i)=(xb_i,y)$.
\item If $\ii=i_1 \dots i_d$ with $i_1<i_2<\cdots<i_d$ then $(1,b_{\ii})=v^d$. Such a monomial
  $b_{\ii}$ is called an \emph{increasing monomial}, and $\ii$ an \emph{increasing sequence}.  When
  $d=0$, the sequence $\ii$ is empty and $(1,1)=1$.
\end{enumerate}

\begin{remark} \label{unicity} It is not difficult to observe that $(,)$ is the unique form
  satisfying these three properties.  This is because the Hecke algebra has a spanning set over
  $\Ztt$ consisting of monomials $b_{\ii}$, and every monomial may be reduced, by cycling the last
  $b_i$ to the beginning and by applying the Hecke algebra relations, to an increasing
  monomial. This is a simple combinatorial argument that we leave to the reader.
\end{remark}

%
\subsection{Soergel bimodules}
\label{subsec-sbim}
%

In~\cite{Soe1} Soergel introduced a category of bimodules which categorified the Hecke algebra,
and later generalized his construction to any Coxeter group $\mc{W}$~\cite{Soe4}.  
Within the category $\RmolfZR$, for  $R$ a certain graded $\Bbbk$-algebra ($\Bbbk$ an infinite field 
of characteristic $\ne 2$), 
he identified indecomposable modules $B_w$ for $w \in \mc{W}$, such that the only indecomposable
summands of tensor products of $B_w$'s are $B_{w'}$ for $w'\in \mc{W}$.  Thus, the subcategory of
$\RmolfZR$ generated additively by the $B_w$ has a tensor product, and its Grothendieck ring is
isomorphic to $\mc{H}$, under the isomorphism sending $C'_s$ to $[B_s]$. Moreover, every 
$B_w$ shows up as a summand of some tensor product of various $B_s$ for $s \in \mc{S}$. 
While the general $B_w$ may be difficult to describe, $B_s$ has an easy description.

It is conjectured in~\cite{Soe4} that this isomorphism sends $[B_w]$ to $C'_w$ for all $w \in
\mc{W}$, and it is proven for $\Bbbk=\C$ and $\mc{W}$ a Weyl group in~\cite{Soe1}, using geometric methods.

Henceforth we specialize to the case where $\mc{W}=S_{n+1}$ and $\mc{S}=\{s_i\}$. We also make one additional
change from Soergel's conventions:

\begin{remark} \label{dumbchange} Soergel defines $R$ to be the coordinate ring of the $n$-dimensional
reflection representation $V$ of $S_{n+1}$, while we find it easier to consider $R'$, the coordinate ring of
the $n+1$-dimensional standard representation $V'$. This is akin to treating $\mathfrak{gl}_n$ instead of
$\mathfrak{sl}_n$, and a similar convention is adopted in~\cite{Ktriply}. The bimodules $B_w$ are defined in
\cite{Soe4} to be the coordinate rings of unions of ``twisted diagonals" in $V \times V$, and $B_w'$ can be
defined analogously for $V'$. Now $R' \cong R \otimes_{\Bbbk} \Bbbk[y]$ and the entire story of $B_w$
translates to $B_w'$ by base extension. Conversely, $R$ is a quotient of $R'$ by the first elementary symmetric
polynomial $e_1$, which is a central element of our category, so that the entire story of $B_w'$ translates
easily to $B_w$ under the quotient. We will interest ourselves with the $B_w'$ story below because the ring
$R'$ is slightly more intuitive, and mention briefly the changes required to deal with $R$ in Section
\ref{subsec-quotient}. Since we only use $B_w'$ and $R'$ below, we will denote them as $B_w$ and $R$ instead to
avoid an apostrophe catastrophe. \end{remark}

With these conventions, we now make the story explicit.

\begin{notation} Let $R= \Bbbk \left[ x_1, x_2,
  \ldots, x_{n+1} \right] $ be the ring of polynomials in $n+1$ variables, with the natural action
of $S_{n+1}$. The ring $R$ is graded, with $\deg(\xi)=2$. If $W$ is the subgroup of $S_{n+1}$
generated by transpositions $\{ s_{i_1},\ldots,s_{i_r} \}$, then we denote the ring of invariants
under $W$ as $R^{i_1,\ldots,i_r}$ or $R^W$. Thus $R^i$ are the invariants under the transposition
$(i,i+1)$.

Since $R$ is an $R^W$-algebra, $\oTen{R^W}$ is a bifunctor sending two $R$-bimodules to an
$R$-bimodule. Henceforth $\otimes$ with no subscript denotes tensoring over $R$, while
$\oTen{i_1,\ldots,i_r}$ denotes tensoring over the subring $R^{i_1,\ldots,i_r}$. Most commonly
we will just use $\Teni$ for various indices $i$.
\end{notation}

\begin{defn} The Soergel bimodule $B_i$ is $R\Teni R \{-1\}$, where $\{m\}$ denotes a grading shift.
\end{defn}

\begin{notation} We denote by $B_{\ii}$ the tensor product $B_{i_1}\TenR B_{i_2}\TenR \cdots \TenR B_{i_d}$.
\end{notation}

Note that $B_{\emptyset}=R$ and \begin{align} \label{useoften} B_{\ii}=& (R\oTen{i_1}R\{-1\})\TenR (R\oTen{i_2}R\{-1\})\TenR
\cdots \\ \nonumber &= R\oTen{i_1}R\oTen{i_2}\cdots\oTen{i_d}R\{-d\}.  \end{align} We reiterate this
important point: a typical element of a tensor product of $d$ generators $B_i$ can be expressed (up to linear
combination) by a choice of $d+1$ polynomials, one in each slot separated by the tensors. Multiplication by an
element of $R$ in any particular slot is an $R$-bimodule endomorphism.

For each $i$ there is a map of graded vector spaces $\partial_i \maps R \to R^i\{2\}$, called the
\emph{Demazure operator}, sending $f\mapsto \frac{f-s_i(f)}{\xi-x_{i+1}}$. This map is $R^i$-linear. Since
$\partial_i(f)$ is $s_i$-invariant, it is not hard to see that $P_i(f)=f-\xi\partial_i(f)$ is also
$s_i$-invariant. Since $f=P_i(f)+\xi\partial_i(f)$, this implies that $R$ is a free graded $R^i$-module of rank
two, with homogeneous generators 1 and $\xi$. In other words, there is an isomorphism $R \cong R^i \oplus
R^i\{2\}$ of graded $R^i$-modules, sending $f \mapsto (P_i(f),\partial_i(f))$, with inverse $(f,g) \mapsto
f+g\xi$.

From the isomorphism $R \cong R^i \oplus R^i\{2\}$ of graded $R^i$-modules just illustrated, one can deduce
other isomorphisms. For instance, $B_i = R \Teni R \{-1\} \cong R \{-1\} \oplus R\{1\}$ as graded left (or
right) $R$-modules. Repeating this, we see that $B_{\ii}$ is a free left $R$-module of rank $2^{d(\ii)}$ , and
properly belongs in $\RmolfZR$. Finally, we can deduce an isomorphism $B_i \Teni B_i \cong B_i\{1\} \oplus
B_i\{-1\}$, which unlike the previous isomorphisms is actually an isomorphism of $R$-bimodules.

\begin{remark} \label{forcingintro} Let us make this slightly more explicit. To give the isomorphism of left
$R$-modules $B_i \cong R\{-1\} \oplus R\{1\}$, note that \begin{equation} \label{force} f \teni g = f \teni
P_i(g) + f \teni \xi\partial_i(g)=P_i(g)f\teni 1 + \partial_i(g)f \teni \xi. \end{equation} Rewriting a term $1
\teni g$ as a sum of terms like above will happen often, and we refer to it as \emph{forcing} the polynomial
$g$ across the tensor. If $g$ is $s_i$-invariant then it may be slid across leaving nothing behind, while an
arbitrary $g$ when forced leaves terms with at either 1 or $\xi$ behind (alternatively, we may choose to leave
1 and $\xp$ behind, if it is more convenient). We consistently use the term ``slide" instead of ``force" when
the polynomial $g$ is invariant so it can be moved across without any ado.

Inside $B_i \ot B_i$, taking an element $f \teni g \teni h$ and forcing $g$ to the left (or right) is now an
$R$-bilinear operation, since multiplication on the left or right will only affect $f$ or $h$, not $g$. This
gives the isomorphism $B_i \ot B_i \cong B_i\{1\} \oplus B_i\{-1\}$ of $R$-bimodules, via $f \teni g \teni h
\mapsto (fP_i(g) \teni h,f\partial_i(g) \teni h)$ with inverse $(f_1 \teni h_1,f_2 \teni h_2) \mapsto f_1 \teni
1 \teni h_1 + f_2 \teni \xi \teni h_2$. These maps are $R$-bimodule morphisms, since the only polynomials which
can slide from $f_i$ to $h_i$ in both the source and the target of the map are those polynomials in $R^i$.
\end{remark}

\begin{remark} \label{spanningintro} We also remark on spanning sets for $B_{\ii}$ as $R$-bimodules. For
instance, we've seen that $B_i\TenR B_i$ has a spanning set $\{ 1\teni 1 \teni 1,\ 1\teni \xi \teni 1 \}$. The
bimodule $B_i \TenR B_j$ for $j\ne i$ has a spanning set $\{ 1\teni 1 \tenj 1\}$, since any polynomial in the
middle can be forced to the left leaving at most $\xi$ behind (or $\xp$, which we choose when $j=i-1$), and
that can be slid to the right; thus $1 \teni 1 \tenj 1$ generates it as a $R$-bimodule. The bimodule $B_i \TenR
B_j \TenR B_k \TenR B_i$ has a spanning set $\{ 1\teni 1 \tenj 1 \tenk 1 \teni 1,\ 1\teni \xi \tenj 1 \tenk 1
\teni 1 \}$. This is because all polynomials anywhere between the two $i$ tensors may be slid out, leaving
$\xi$ somewhere in-between. As an exercise, the reader may generalize this argument to an arbitrary $B_{\ii}$
and find a spanning set as a $R$-bimodule, consisting of $2^{m(\ii)}$ terms, where $m(\ii)$ is the number of
pairs $1 \le r < s \le d$ such that $i_r=i_s$ and $i_t \ne i_s$ for $t$ between $r$ and $s$. Between such a
pair, one either places a linear ``unslideable" term like $x_{i_r}$, or just $1$. Note that $m(\ii)$ is equal
to $d(\ii)$ minus the number of distinct indices in $\ii$. \end{remark}

For more information on Soergel bimodules and their applications we refer the reader to 
\cite{Soe1, Soe2, Soe3, Fie, Ktriply, Lib1, Lib2, Rou1, WW1, Wi} and references therein. 

%
\subsection{The Soergel Categorification}
\label{subsec-sbimcat}
%

Let us summarize the Soergel categorification of the Hecke algebra $\mc{H}$ of the symmetric 
group $S_{n+1}$ and the various structures on it, following~\cite{Soe1, Soe2, Soe4, Lib1, Rou1}.

Several subcategories of $\RmolfR$ and $\RmolfZR$ will play a role in what follows. Let $\mc{SC}_1$ be the full
subcategory of $\RmolfR$ whose objects consist of $B_{\ii}$ for all sequences of indices $\ii$; these are
called \emph{Bott-Samelson bimodules}. Since $R$ is a commutative ring, the Hom spaces in $\mc{SC}_1$ are in
fact enriched in $\RmolfZR$. Let $\mc{SC}_2$ be the subcategory of $\RmolfZR$ whose objects are finite direct
sums of various graded shifts of objects in $\mc{SC}_1$ and the morphisms are all grading-preserving bimodule
homomorphisms. Finally, let $\mc{SC}$ be the Karoubi envelope of $\mc{SC}_2$, a category equivalent to the full
subcategory of $\RmolfZR$ which contains all summands of objects of $\mc{SC}_2$:

\begin{equation*}
\begin{CD} 
\mc{SC}_1 @>{\mathrm{Grading \ shifts \ and \ direct \ sums \ \ }}>> \mc{SC}_2  
@>{\mathrm{Karoubi \ envelope \ \ }}>>     \mc{SC}
\end{CD}    
\end{equation*} 

In general, the Karoubi envelope is the category which formally includes all ``summands'', where a summand of
an object $M$ is identified by an idempotent $p \in\mathrm{End}(M)$ corresponding to projection to that
summand. Since $\RmolfZR$ is abelian, it is idempotent-closed, and the Karoubi envelope of $\mc{SC}_2$ can be
realized as a subcategory in $\RmolfZR$. We refer the reader to \cite{BM} for basic information about Karoubi
envelopes.

This final category $\mc{SC}$ (the category of \emph{Soergel bimodules}) is a $\Bbbk$-linear additive monoidal
category with the Krull-Schmidt property. Soergel showed that, when $\Bbbk$ is an infinite field of
characteristic other than $2$, the indecomposable bimodules in this category are enumerated by elements of the
Weyl group and grading shifts (Theorem 6.16 in~\cite{Soe4}). They are denoted by $B_w\{j\}$ for $w\in \mc{W}$
and $j\in \Z$. An indecomposable $B_w$ is determined by the condition that it appears as a direct summand of
$B_{\ii}$, where $\ii =i_1\dots i_d$ and $s_{i_1}\dots s_{i_d}$ is a reduced presentation of $w$, and does not
appear as direct summand of any $B_{\ii}$, for sequences $\ii$ of length less than $l(w)$.

The Hecke algebra $\mc{H}$ is canonically isomorphic to $K_0(\mc{SC})$, the Grothendieck group of $\mc{SC}$.
Multiplication by $v$ corresponds to the grading shift: $[M\{d\}] = v^d [M]$. Multiplication in the
Hecke algebra corresponds to the tensor product of bimodules: $$ [M] \cdot [N] \ := \ [M\otimes_R N].$$ The
isomorphism takes $[B_i]$ to $b_i$ and $[B_{\ii}]$ to $b_{\ii}$.

\begin{remark} Nothing prevents one from defining a category $\mc{SC}_{\Z}$ where the field $\Bbbk$ is replaced
with $\Z$ in the definitions of the previous section. Thus one could define the category for any ring. However,
one does not have control over the size of the Grothendieck group in this instance. When defined over a field
$\Bbbk$ of characteristic $\ne 2$, we may use Theorem 6.16 of \cite{Soe4} to classify indecomposables and get
results about the Grothendieck group. When $\Bbbk=\C$, Soergel has shown that the Kazhdan-Lusztig basis
$\{C'_w\}$ satisfies $C'_w=[B_w]$. This is unknown in general. \end{remark}

Relations (\ref{eqn-bisq})--(\ref{eqn-bibpbi}) lift to isomorphisms of graded bimodules in $\mc{SC}$

\begin{eqnarray}
B_i \TenR B_i & \cong & B_i\{ 1 \} \oplus B_i\{-1\} \label{eqn-iibimod}\\
B_i \TenR B_j & \cong & B_j \TenR B_i \ \mathrm{ for }\  |i-j|\ge 2\\
(B_i \TenR B_{i+1} \TenR B_i) \oplus B_{i+1} & \cong & (B_{i+1} \TenR B_i \TenR B_{i+1}) 
\oplus B_i . \label{eqn-ipibimod}
\end{eqnarray}

It is important to note that these isomorphisms take place in $\mc{SC}_2$, not $\mc{SC}_1$, since the latter
does not have grading shifts or direct sums of objects. However, the same information can be encapsulated in
inclusion and projection morphisms of various degrees, which do live in $\mc{SC}_1$. This will be explored in
Section \ref{subsec-functor}.

The first isomorphism has already been made explicit in Remark \ref{forcingintro}. We have chosen a specific
isomorphism; other choices were possible. The second and third isomorphisms come from the following
isomorphisms in $\RmolfZR$:

\begin{eqnarray}
B_i \TenR B_j \cong R \oTen{i,j} R\{-2\} \cong B_j \TenR B_i \ \mathrm{for}\ |i-j|\ge 2 
\label{eqn-auxbimod}\\
B_i \TenR B_{i+1} \TenR B_i \cong B_i \oplus (R \oTen{i,i+1} R\{-3\})\\
B_{i+1} \TenR B_i \TenR B_{i+1} \cong B_{i+1} \oplus (R \oTen{i,i+1} R\{-3\}).
\label{eqn-auxbimod2}
\end{eqnarray}

These isomorphisms will be made explicit in Section \ref{subsec-functor}.

The ring $\Zvv$ is canonically isomorphic to the Grothendieck group of the category $R\mathrm{-fmod}$ of
finitely-generated graded free $R$-modules. Under this isomorphism $$ K_0(R\mathrm{-fmod}) \cong \Zvv$$ $[R]$
goes to $1$ and $[R\{d\}]$ to $v^d$. In particular, given any graded free $R$-module $M$, its image in the
Grothendieck group will be its \emph{graded rank}, calculated by choosing a homogeneous $R$-basis $\{y_j\}$ of
$M$ and letting $\grdrk M \define \sum_j v^{\deg {y_j}}$. The bar involution on $\Zvv$ lifts to the
contravariant equivalence that takes $M\in R\mathrm{-fmod}$ to $\Hom_R(M,R),$ the latter naturally viewed as an
$R$-module.

In the category $\mc{SC}_1$, given any objects $M,N$, the space $\Hom_{\mc{SC}_1}(M,N)$ is a graded free
finitely-generated left $R$-module. By extension, the same is true of the module $\oplus_{m \in
\Z}\Hom_{\mc{SC}}(M\{m\},N)$. Shifting the grading of $N$ will shift the grading of this Hom space in the same
direction, while shifting $M$ will shift the Hom space in the opposite direction. Therefore, the bifunctor
$$\Hom_{\mc{SC}}(\cdot,\cdot) \colon \mc{SC}^{\rm op} \times \mc{SC} \to R\mathrm{-fmod}$$ categorifies a
semilinear form $\mc{H} \times \mc{H} \to \Zvv$ which sends $$([M],[N]) \mapsto \grdrk \big( \oplus_{m\in
\Z}\Hom_{\mc{SC}}(M\{m\},N)\big) .$$

The bimodule $B_i$ self-biadjoint, i.e. that $\Hom_{\mc{SC}_1}(M \TenR B_i,N)=\Hom_{\mc{SC}_1}(M,N \TenR B_i)$
and $\Hom_{\mc{SC}_1}(B_i \TenR M,N) = \Hom_{\mc{SC}_1}(M,B_i \TenR N)$ via some adjunction maps. This will
become explicit in Section \ref{subsec-dc1one}. In fact, every bimodule $M$ in $\mc{SC}$ has a biadjoint
bimodule $\Omega(M)$ such that tensoring with $M$ on the left (resp. right) is biadjoint to tensoring with
$\Omega(M)$ on the left (resp. right). Due to a cyclicity property (see the next section) any homomorphism
$f:M\longrightarrow N$ of bimodules dualizes to a canonical homomorphism $\Omega(f): \Omega(N) \longrightarrow
\Omega(M)$, so that $\Omega$ can be made into an antiequivalence of $\mc{SC}$, lifting the antiinvolution
$\omega$. Notice that $\Omega$ takes $B_i$ to itself and $B_{\ii}$ to $B_{\jj}$ where $\jj$ is given by reading
$\ii$ from right to left.

Unsurprisingly, the semilinear product on $\mc{H}$ above (induced by the Hom bifunctor) agrees with the one
defined in Section \ref{subsec-hecke}. To check this, following Remark~\ref{unicity} and using the
self-biadjointness of $B_i$, we only need to show the following claim.

\begin{claim} \label{increasinghom} When $\ii$ is an increasing sequence, $\Hom_{\mc{SC}_1}(R,B_{\ii})$ is a
free left $R$-module of rank 1, generated by a morphism of degree $d$, the length of $\ii$. \label{phiii}
\end{claim}

\begin{proof} We only sketch this result. An $R$-bimodule map from $R$ to $B_i$ is clearly determined by an
element of $B_i$ on which right and left multiplication by polynomials in $R$ are identical. Any element of
$B_i$ is of the form $m=f\teni 1 + g \teni x_i$, and clearly $\xj m=m \xj$ for $j \ne i,i+1$, and
$(\xi+\xp)m=m(\xi+\xp)$. Hence $m$ can be the image of 1 under a bimodule map from $R$ if and only if $\xi m=m
\xi$. $$m \xi= f \teni \xi + g \teni \xi^2 = f \teni \xi+ g(\xi+\xp) \teni \xi - g(\xi\xp) \teni 1.$$ Thus $m$
can be an image iff $g(\xi\xp)=-f\xi$ and $f+g(\xi+\xp)=g\xi$, iff $f=-g\xp$. Such elements form a left
$R$-module generated by the case $g=1$, $f=-\xp$, or in other words by $m=1 \teni \xi - \xp \teni 1$. The
element $m$ has degree 1 in $B_i$, so we deduce that $(R,B_i)=v$. Let us call $\phi_i$ the corresponding map $R
\to B_i$, $1 \mapsto m$.

One may use the same argument for the general case. Suppose $\ii$ is increasing and has length $d$. Then
$B_{\ii}$ is a free left $R$-module of degree $2^d$, with generators
$\{1\otimes_{i_i}f_1\cdots\otimes_{i_k}f_k\}$ ranging over terms where either $f_l=1$ or $f_l=x_{i_l}$. As an
exercise, the reader may find the criteria for a general element to be ad-invariant under $\xi$, and verify
that the only possible bimodule maps $R \to B_{\ii}$ are $R$-multiples of the following iterated version of
$\phi_i$: $$R \to B_{i_1} \cong B_{i_1} \TenR R \to B_{i_1} \TenR B_{i_2} \cong \cdots.$$ The first map is
$\phi_{i_1}$, the second map is $\mathrm{Id}\otimes \phi_{i_2}$, and so forth. This generator is a map of
degree $d$, so that $\Hom_{\mc{SC}_1}(R, B_{\ii})= R\{d\}$ and $([R],[B_{\ii}])=v^d$. \end{proof}

\begin{remark} We have swept the calculation under the rug, so the dependence of this claim on the fact that $\ii$ is
increasing is unclear. In general, when $\ii$ has a repeated index there will be additional maps from $R$ to $B_{\ii}$. Roughly
speaking, certain symmetry conditions are placed upon polynomials in order for them to slide across certain tensors. The
duplication of an index will yield a redundant symmetry condition that places fewer constraints on an ad-invariant element than
would be expected from the length of the sequence. We suggest the reader try to find all the maps from $R$ to $B_{\ii}$ in the
length 2 case, first when $\ii=ij$ has no repeated index and then when $\ii=ii$ has a repeated index. This should illustrate
the main idea. \end{remark}

Because it is a crucial statement which we use again and again, we restate the overall result and give a
reference.

\begin{prop} \label{homsinSC} Given two sequences $\ii$ and $\jj$, $\Hom_{\mc{SC}}(B_{\ii},B_{\jj})$ is a free
graded left (or right) $R$-module of rank $\epsilon(\omega(b_{\ii})b_{\jj})$, where $\epsilon$ is the standard
trace on $\mc{H}$ defined in Section \ref{subsec-hecke}. \end{prop}

\begin{proof} This is deduced from the above discussion. For Soergel's proof, see \cite{Soe4}, although it is
somewhat obscured. Theorem 5.15 in that paper and especially its proof together state this result, once one
unravels exactly what $h_\Delta$ means. Propositions 5.7 and 5.9 state that $h_\Delta(B_{\ii})=b_{\ii}$, since
$h_\Delta$ sends $B_i \TenR \cdot$ to $b_i \cdot$ and sends $R$ to $1$. \end{proof}

The facts below will not be used in this paper.

For a Soergel bimodule $M$ the space of bimodule homomorphisms $\Hom_{\mc{SC}_1}(R,M)$ is just the $0$-th
Hochschild cohomology $\mathrm{HH}^0(R,M)$ of $M$. Thus, unraveling the definitions, $$ \epsilon([M]) := \grdrk
\big( \mathrm{HH^0(R,M)}\big).$$ Calculations in Hochschild cohomology can be used to provide a proof of the
claim above. One could also define a trace map $\tau(x)=\overline{\epsilon(\omega(x))}$ which is the
decategorification of the functor $\mathrm{HH}_0$ of taking the 0th Hochschild homology.

Hecke algebra $\mc{H}$ has a trace more sophisticated than $\epsilon$ or $\tau$, called the Ocneanu
trace~\cite{Jo}, which describes the HOMFLY-PT polynomial. The categorification of the Ocneanu trace utilizes
all Hochschild homology groups rather than just $\mathrm{HH}_0$, see~\cite{Ktriply, WW1, WW2}.

The Rouquier complexes mentioned in the introduction are described here. Invertible elements $T_i$ that satisfy the braid
relations become~\cite{Rou1} invertible complexes $$ 0 \lra R \lra B_i \{ -1\} \lra 0 $$ in the homotopy category of the
Soergel category (with $R$ sitting in cohomological degree $-1$). This aligns with the fact that $T_i=v^{-1}b_i-1$ in the Hecke
quotient of the braid group. Their inverses $T_i^{-1}$ become inverse complexes $$ 0 \lra B_i\{1\} \lra R \lra 0 $$ with $R$ in
cohomological degree $1$, agreeing with $T_i^{-1}=vb_i-1$. The homomorphism from the braid group into the Hecke algebra is
categorified by a projective functor from the category of braid cobordisms between $(n+1)$-stranded braids to the category of
endofunctors of the homotopy category of the Soergel category~\cite{KhT}.

\begin{remark} Note that this convention for Rouquier complexes is opposite that found in~\cite{Rou1}, which is to say that we
have flipped $T_i$ with $T_i^{-1}$. Presumably this arises because we are using $v$ as a parameter for the grading shift, and
not $t=v^{-1}$. The choice of $v$ is more natural for the calculation of graded ranks of Hom spaces. \end{remark}

%
\subsection{Diagrammatic calculus for bimodule maps}
\label{subsec-diagintro}
%

We follow the standard rules for the diagrammatic calculus of bimodules, or more generally for the diagrammatic
calculus of a monoidal category. An excellent and thorough explanation of these rules can be found in
\cite{AL2}, so we will provide a quick summary. A planar diagram will represent a morphism of $R$-bimodules,
with the following conventions. A horizontal slice or line segment in this diagram will represent an object (an
$R$-bimodule). A rectangle inside the plane will represent a morphism from its bottom horizontal line segment
to its top horizontal line segment.

The $R$-bimodule $B_i$ is denoted by a point (on a horizontal line segment) labelled $i$. The tensor product of
bimodules is depicted by a sequence of labelled points on a horizontal line segment, so that tensor products
are formed ``horizontally''. A vertical line labelled $i$ denotes the identity endomorphism of $B_i$, and
similarly labelled lines placed side by side denote the identity endomorphism of the tensor product. More
general bimodule maps are represented by some symbols connecting the appropriate lines, and are composed
``vertically'', and tensored ``horizontally''. All diagrams are read from bottom to top, so that the following
diagram represents a bimodule map from $B_k$ to $B_i\TenR B_i \TenR B_j$.

\igc{.5}{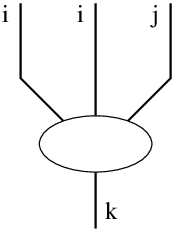}

A horizontal line segment which does not contain any marked points represents $R$ as a bimodule over itself,
the monoidal identity. The empty rectangle represents the identity endomorphism of $R$. Planar
diagrams without top and bottom endpoints (without boundary) represent more general endomorphisms of $R$.

The structure of bimodule categories (or more generally strict 2-categories) guarantees that a planar diagram
will unambiguously denote a morphism of bimodules.

We will be using so many such pictures that it will become cumbersome to continuously label each
line by an index. Generally, the calculations we do will work independently for each $i$, and can be
expressed with diagrams that use lines labelled $i$, $i+1$ and the like. In these circumstances,
when there is no ambiguity, we will fix an index $i$ and draw a line labelled $i$ with one style, a
line labelled $i+1$ with a different style, and so forth, maintaining the same conventions
throughout the paper.  We use different styles of lines because most printers are black and white, but
we recommend that you do your calculations at home in colored pen or pencils instead; 
we even refer to the labels as ``colors'' throughout this paper.

\igc{.5}{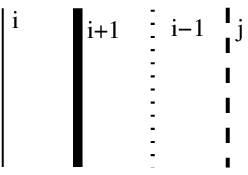}

We use the styles above when referring to indices $i$, $i+1$, $i-1$, and $j$, where $j$ will be
used unambiguously for any index which is ``far away'' from any other indices in the picture (in
other words, when drawing a picture only involving $i$-colored strands, we require $|i-j|>1$, while
for a picture involving both $i$ and $i+1$ we require $j<i-1$ or $j>i+2$).

%
\subsection{Methodology}
\label{subsec-moral}
%

\begin{prop} \label{prop-full} Suppose we choose the subset of the morphisms in $\mc{SC}_1$,
  including the identity morphism of each object, as well as the following morphisms:
\begin{enumerate}
\item The generating morphism from $R$ to $B_i$.
\item Some isomorphisms that yield the Hecke algebra relations, as well as the respective projections
  to and inclusions from each summand in (\ref{eqn-iibimod}) and (\ref{eqn-ipibimod}).
\item The unit and counit of adjunction that make $B_i$ into a self-biadjoint bimodule.
\end{enumerate}

Consider $\mc{C}$ the subcategory generated monoidally over the left action of $R$ by these
morphisms, i.e. it includes left $R$-linear combinations, compositions, and tensors of all its
morphisms. Then $\mc{C}$ is a full subcategory, and thus it is actually $\mc{SC}_1$.
\end{prop}

\begin{proof} For any objects $M,N$ in $\mc{SC}_1$, there is an inclusion $\Hom_{\mc{C}}(M,N)
  \subset \Hom_{\mc{SC}_1}(M,N)$ of graded left $R$-modules (since it is clearly an inclusion of
  $R$-modules, and all generating morphisms are homogeneous).  One can define $\grdrk$ for any
  graded left $R$-module $M$ by choosing generators of $M/(R^+M)$, where $R^+$ is the ideal of
  positively graded elements, and it is a simple argument that a submodule of a free graded
  $R$-module with the same graded rank is in fact the entire module.  So we need only show that Hom
  spaces in $\mc{C}$ have the same graded rank.

  We can define a semilinear form on the free $\Zvv$-algebra generated by $b_i$ by the formula $(b_{\ii},b_{\jj}) = \grdrk
  \Hom_{\mc{C}}(B_{\ii},B_{\jj})$. The existence of isomorphisms and projection maps will give us the direct sum decompositions
  (\ref{eqn-iibimod}) - (\ref{eqn-ipibimod}) in $\mc{C}$, with the resulting implications for Hom spaces. Therefore the Hecke
  algebra relations (\ref{eqn-bisq}) - (\ref{eqn-bibpbi}) are in the kernel of this semilinear form, so it descends to a form
  on the Hecke algebra. Each $b_i$ will be self-adjoint. When $\ii$ is increasing, $\Hom_{\mc{C}}(R,B_{\ii})$ contains the
  generator of the free rank one $R$-module $\Hom_{\mc{SC}_1}(R,B_{\ii})$, since that generator is the tensor of the generating
  morphisms from $R$ to various $B_i$ (see the proof of Claim \ref{phiii}). Hence it is in fact the entire module, so
  $\Hom_{\mc{C}}(R,B_{\ii}) \cong R\{d\}$, and $(1,b_{\ii})=v^d$.

  By unicity, this inner product agrees with our earlier inner product on the Hecke
  algebra. In particular, the graded ranks agree, and the inclusion is full.
\end{proof}

Below we will construct a category $\mc{DC}_1$ of diagrams via generators and local relations, where
the Hom spaces will be graded $R$-bimodules. We will construct a functor $\mc{F}_1$ from $\mc{DC}_1$
to $\mc{SC}_1$, showing that our diagrams give graphical
presentation of morphisms in $\mc{SC}_1$.  The morphisms in the image of $\mc{F}_1$ will
include all the morphisms enumerated in Proposition \ref{prop-full}, hence the functor will be
full. Calculating the Hom spaces in $\mc{DC}_1$ between certain objects (corresponding to $R$,
$B_{\ii}$ for $\ii$ increasing), we may use a similar argument to the above proposition to show that
they are free $R$-modules of the same graded rank as the Hom spaces in $\mc{SC}_1$.  Then
the functor $\mc{F}_1$ will be faithful, and an equivalence of categories. This describes
$\mc{SC}_1$ in terms of generators and relations.

Let $\mc{DC}_2$ be the category whose objects are finite direct sums of formal grading shifts of
objects in $\mc{DC}_1$, but whose morphisms only include degree 0 maps. Finally, let
$\mc{DC}=Kar(\mc{DC}_2)$ be the Karoubi envelope of $\mc{DC}_2$.  The functor $\mc{F}_1$ lifts to
functors $\mc{F}_2$ and $\mc{F}$, as in the picture below, with all three horizontal arrows being
equivalences of categories.

\begin{equation}\label{sixcategories} 
\begin{CD} 
 \mc{DC}_1  @>{\mc{F}_1}>>  \mc{SC}_1  \\
@VVV  @VV{\mathrm{Grading \ shifts \ and \ direct \ sums \ \ }}V  \\
\mc{DC}_2  @>{\mc{F}_2}>>  \mc{SC}_2  \\
@VVV  @VV{\mathrm{Karoubi \ envelope \ \ }}V  \\
\mc{DC}  @>{\mc{F}}>> \mc{SC}
\end{CD}    
\end{equation} 

We will define the category $\mc{DC}_1$ originally without reference to isotopy, in order to make the
definition of the functor $\mc{F}_1$ entirely straightforward, using the standard rules for diagrammatics for
bimodules. The category would be entirely unchanged if one used different pictures to represent each morphism.
However, when the ``correct'' pictures are chosen for the generators, then every morphism can actually be
viewed as a \emph{planar graph}, and moreover two embedded graphs linked by isotopy represent the same
morphism. One could very well define $\mc{DC}_1$ originally using graphs, but this would obscure the definition
of $\mc{F}_1$.

The most difficult part of the proof will be showing the faithfulness of $\mc{F}_1$, which involves a
calculation of certain Hom spaces in the diagrammatic category. This calculation will be made possible by the
planar graphs interpretation of $\mc{DC}_1$, wherein some relatively simple graph theory can be applied to
simplify pictures.

\section{Definition of $\mc{DC}$}
\label{sec-defns}

This section contains a piecemeal definition of $\mc{DC}$ and $\mc{DC}_1$. For pedagogical reasons, we prefer
to provide commentary as we go, instead of defining the category all at once (in fact, some relations do not
make sense without the commentary). We also provide some redundant relations in the first pass, because they
help make the category more intuitive. However, we repeat the definition all in one place in Section
\ref{subsec-fulldefn}, without redundant relations, where we also explicitly define the functor $\mc{F}_1$.

%
\subsection{The category $\mc{DC}_1$: zero colors and one color}
\label{subsec-dc1one}
%

This section and the next several will hold the definition of the category $\mc{DC}_1$, which will
be a $\Bbbk$-linear additive monoidal category, with $\Z$-graded Hom spaces. Shortly it will become
clear that Hom spaces are actually graded $R$-bimodules. It is generated monoidally by $n$ objects
$i=1, \ldots, n$, whose tensor products will be denoted $\ii=i_1\ldots i_d$.

Morphisms will be given by (linear combinations of) diagrams inside the strip $\R \times [0,1]$,
constructed out of lines colored by an index $i$, and certain other planar diagrams, modulo local
relations. The intersection of the diagram with $\R \times \{0\}$, called the \emph{lower boundary},
will be a sequence $\ii$ of colored endpoints, the source of the map, and the \emph{upper boundary}
$\jj$ will be the target. A vertical line colored $i$ represents the identity map from $i$ to
$i$. The monoidal structure consists of placing diagrams side by side, and composition consists of
placing diagrams one above the other, in the standard fashion for diagrammatic categories.

We present the generators and relations in an order based on the number of colors they use. The one-color
generators and relations will be sufficient to describe the category for $n=1$, the two-color ones for $n=2$,
and the three-color ones for the general case. The set of all relations is invariant under all color changes
that preserve adjacency, so we only display each generator for a single color $i$, using the conventions
described in Section \ref{subsec-diagintro}. However, the generator exists for each index $i$.

All the relations we will give are homogeneous with respect to the grading on generators stated.

The first class of generators, which use no colors, are the following endomorphisms of the monoidal identity
$\emptyset$:

\vspace{.5cm}

$\ig{.4}{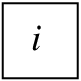}$

\vspace{.5cm}

There is one such generator for each $i=1 \ldots n+1$. It is a map of degree 2, which we call
\emph{multiplication by } $\xi$. After we apply the functor $\mc{F}_1$, this will actually correspond to the
endomorphism of $R$ given by multiplication by $\xi$. Together, these generators are called \emph{boxes}. A
morphism from $\emptyset$ to $\emptyset$ consisting of a sum of disjoint unions of boxes will be called a
\emph{polynomial}. Since the composition of multiplication by $\xi$ and multiplication by $\xj$ is
multiplication by $\xi\xj$, such a sum of products of boxes will obviously correspond under the functor to
multiplication by an element $f\in R$. As a shorthand we draw such a morphism as a box with the
corresponding element $f$ inside. As a map from $\emptyset$ to $\emptyset$, and thus a closed diagram, a
polynomial may be placed in any region of another diagram. Placing boxes in the rightmost and leftmost regions
of a diagram will define the $R$-bimodule structure on Hom spaces in $\mc{DC}_1$.

The generating morphisms which use only one color are:

$
\begin{array}{ccc}
\mathrm{Symbol} & \mathrm{Degree} & \mathrm{Name} \\
\\
\igv{.3}{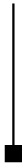} & 1 & \mathrm{EndDot}\\
\\
\ig{.3}{dot.eps} & 1 & \mathrm{StartDot}\\
\\
\igv{.3}{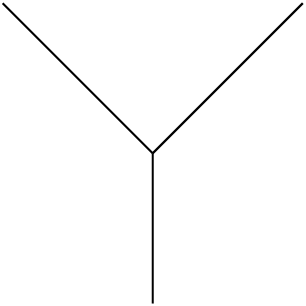} & -1 & \mathrm{Merge}\\
\\
\ig{.3}{merge.eps} & -1 & \mathrm{Split}
\end{array}
$

For the beginner, the maps are respectively: a map from $i$ to $\emptyset$, a map from $\emptyset$ to $i$, a
map from $ii$ to $i$, a map from $i$ to $ii$.

Remember, there is one such set of generators for each color $i$. We give these maps names, but the names are
temporary. Once we explore the meaning of isotopy invariance, we will stop distinguishing between Merge and
Split, and call them both \emph{trivalent vertices}. Similarly we will stop distinguishing between StartDot and
EndDot, and call them both \emph{dots}.

We also use a shorthand for the following compositions:

$
\begin{array}{ccc}
\mathrm{Symbol} & \mathrm{Degree} & \mathrm{Name}\\
\\
  \igv{.3}{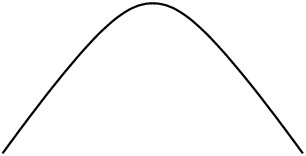} \define \igv{.3}{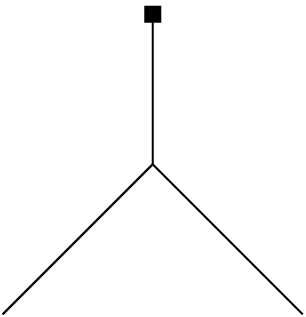} & 0 & \mathrm{Cup} \\
  \\
  \ig{.3}{cap.eps} \define \ig{.3}{dotSplit.eps} & 0 & \mathrm{Cap}
\end{array}
$

We now list a series of relations using only one color, the \emph{one-color relations}, dividing them into
several types of relations for ease of reference. The first set we refer to as the \emph{Frobenius relations},
since they imply that $i$ is a Frobenius object in $\mc{DC}_1$ (see \cite{AL1,Muger} for more on Frobenius
algebras). Once we define the functor $\mc{F}_1$, this will imply that $B_i$ is a Frobenius object in
$\mc{SC}_1$. Remember that the cups and caps appearing below can actually be rewritten in terms of the
generators.

\begin{equation}\ig{.2}{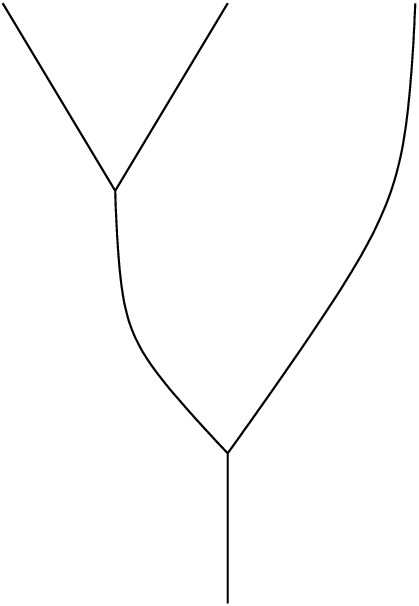} = \igh{.2}{doubleMerge.eps} \label{assoc}\end{equation}
\begin{equation}\igv{.2}{doubleMerge.eps} = \ighv{.2}{doubleMerge.eps} \label{coassoc}
\end{equation}
\begin{equation}\ig{.2}{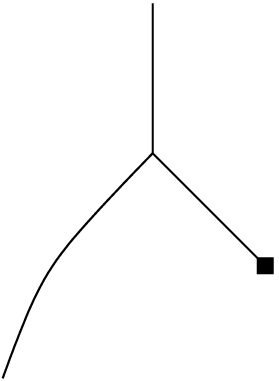} = \ig{.5}{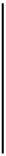} = \igh{.2}{splitDot.eps}\label{counit}
\end{equation}
\begin{equation}\igv{.2}{splitDot.eps} = \ig{.5}{line.eps} = \ighv{.2}{splitDot.eps}\label{unit}
\end{equation}
\begin{equation} \ig{.25}{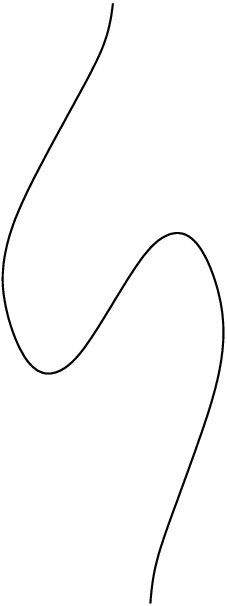} \;\; = \;\; \ig{.2}{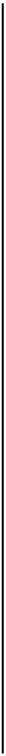} \;\; = \;\; 
 \igh{.25}{cupCap.eps}\label{biadjoint}\end{equation}
\begin{equation}\ig{.3}{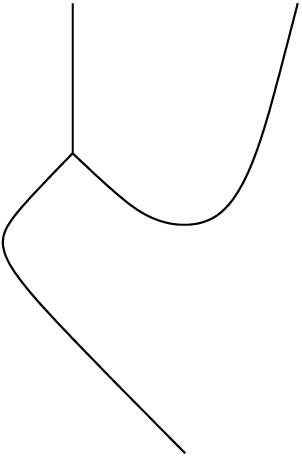} \  = \ \ig{.3}{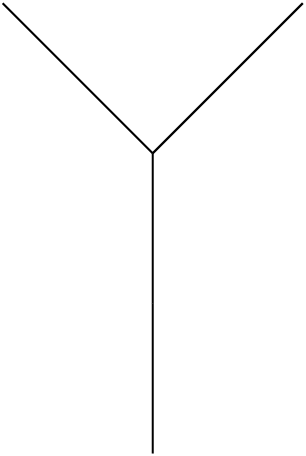} \  = \ \igh{.3}{splitCup.eps}
\label{twistMerge}\end{equation}
\begin{equation}\igv{.3}{splitCup.eps} \  =  \ \igv{.3}{tallMerge.eps}  \ =  \ \ighv{.3}{splitCup.eps}
\label{twistSplit}\end{equation}
\begin{equation}\ig{.3}{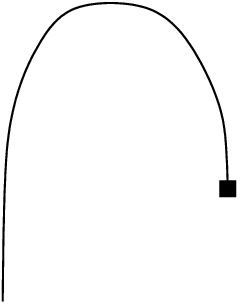} \  =  \ \igv{.4}{dot.eps}  \ =  \ \igh{.3}{capDot.eps}
\label{twistDot1}\end{equation}
\begin{equation}\igv{.3}{capDot.eps} \  = \  \ig{.4}{dot.eps} \  =  \ \ighv{.3}{capDot.eps}
\label{twistDot2}\end{equation}

For quick reference, we refer to these relations by their Frobenius algebra names. The first two are the
associativity of Merge and the coassociativity of Split. The next two are the unit and counit relations.
Relation (\ref{biadjoint}) is the biadjunction relation, and the final four are cyclicity relations.

\begin{remark} \label{cyclicity} For readers not well versed in cyclicity properties and their implications towards isotopy
invariance, let us quickly discuss the topic, using the easily visualized notion of a \emph{twist}. Given a
morphism, one can twist it by taking a line which goes to the upper boundary and adding a cap, letting the line
go to the other boundary instead. An example is given below.

\begin{center} 
$\ig{.4}{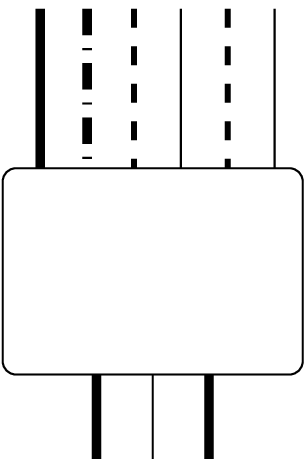} \ \ \ $ twists to $ \ \ \  \ig{.4}{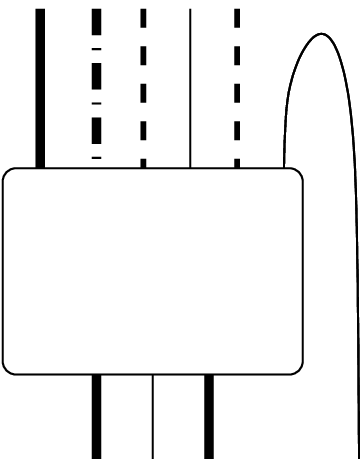}$
\end{center} 

One can also twist a downward line back up, or twist lines on the left as well. Two morphisms are twists of
each other if they are related by a series of these simple twists, using cups and caps on the right and left
side. For instance, relations (\ref{twistMerge}) and (\ref{twistSplit}) state that the Merge is a simple twist
of the Split, twisting on the left or right. If one applies the same twist to every term in a relation, one
gets a twist of that relation. For instance, relation (\ref{unit}) is actually a twist of the definition of the
cup.

Because of biadjointness (\ref{biadjoint}), twisting a line down and then back up will do nothing to the
morphism. Once biadjointness is shown, all twists of a relation are equivalent, because twisting in the reverse
direction we get the original relation back. When a morphism has a total of $m$ inputs and outputs, twisting a
single strand will often be referred to as rotation by $\frac{180}{m}$ degrees.

The above relations imply that twisting any of the above generators by 360 degrees will do nothing. A morphism
is said to be \emph{cyclic} with respect to a fixed set of adjunctions (i.e. cups and caps) if 360 rotation
does not change the morphism. Cyclicity is useful because of the following proposition.

\begin{prop} Fix adjunctions of each object, which are drawn as caps and cups. If every generating morphism in
a diagram is cyclic with respect to those adjunctions, then so is the entire diagram, and the morphism
represented by that diagram is invariant under isotopy of the diagram. \end{prop}

For more on diagrammatics of biadjointness and the cyclicity property we refer the reader to~\cite{B1, B2, CKS,
AL1, AL2}. \end{remark}

Merge and Split are 60 rotations of each other, and each is invariant under 120 degree rotation, so we may
represent them isotopy-unambiguously with pictures that satisfy the same properties. A similar statement holds
for StartDot and EndDot. We will refer to these morphisms as dots and trivalent vertices from now on, because
these terms encapsulate the picture up to rotation.

\begin{remark} \label{liberties} Henceforth, can take more liberties in our drawings. We can draw a horizontal
line colored $i$, and even though this can not be constructed using our generators, it is isotopy equivalent to
a cup or cap which can be so constructed. We can allow a diagram to have a boundary not just on the top or
bottom, but also on the side. While this does not represent a morphism in our category, the line running to the
side boundary can be twisted either up or down to represent a genuine morphism. A relation drawn using diagrams
with side boundaries does unambiguously give a relation in $\mc{DC}_1$. \end{remark}

Associativity and coassociativity are twists of each other. This relation is written in a rotation-invariant
form below, and shall be crucial in the sequel. We refer to this relation, which permits one to ``slide'' one
trivalent vertex over another, as \emph{one-color associativity}.

\begin{equation}\igrotCW{.3}{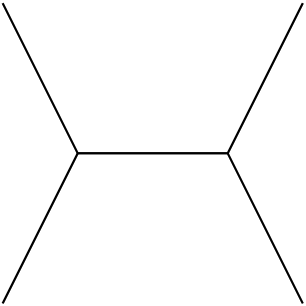} \ \ = \ \ \ig{.3}{Xdiagram.eps} \label{eqn-associativity}
\end{equation}

We refer to either picture above as an ``H''. The horizontal line in the right picture is exactly such a
liberty as in Remark \ref{liberties}.

Note that relations (\ref{assoc}), (\ref{biadjoint}), (\ref{twistMerge}), and (\ref{twistDot1}) are
sufficient to imply the other Frobenius relations, because of the remarks about twisting
made above.  Here is the proof of half of (\ref{twistSplit}) using (\ref{twistMerge}), as an
illustrative example.

\begin{equation*} \igv{.3}{tallMerge.eps} \  = \ \igv{.3}{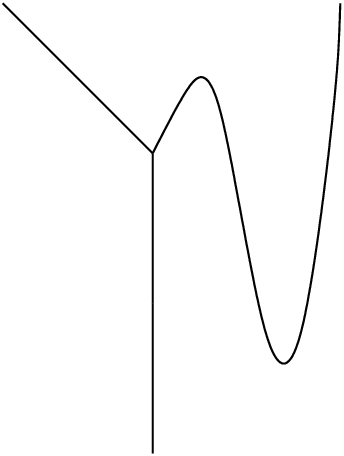} \ \ 
= \ \ \igv{.3}{splitCup.eps}\end{equation*}

The next set of relations are known as \emph{polynomial slides}, which have obvious analogies in the
definitions of the modules $B_i$.

\begin{eqnarray}
\ig{.4}{multxi.eps}\ \ig{.6}{line.eps}\ +\ \ig{.4}{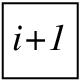}\ \ig{.6}{line.eps}\ &\ =
  \ &\ \ig{.6}{line.eps}\ \ig{.4}{multxi.eps}\ +\ \ig{.6}{line.eps}\ \ig{.4}{multxp.eps}
 \label{eq-slide1} \\    \ \  &  &  \ \  \nonumber  \\
\ig{.4}{multxi.eps}\ \ig{.4}{multxp.eps}\ \ig{.6}{line.eps}\ &\ =\ &\ \ig{.6}{line.eps}\ 
  \ig{.4}{multxi.eps}\ \ig{.4}{multxp.eps} \label{eq-slide2} \\
  \ \  &  &  \ \   \nonumber \\  
\ig{.4}{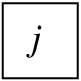}\ \ig{.6}{line.eps}\ &\ =\ &\ \ig{.6}{line.eps}\ \ig{.4}{multxj.eps}
\label{eq-slide3}
\end{eqnarray}

The $j$ appearing in the box in the last relation can be any index not equal to $i$ or $i+1$. Together, these
relations imply precisely that any polynomial which is invariant under $s_i$ can be slid across a line colored
$i$, since $R^i$ is generated by $\xi+\xp$, $\xi\xp$, and $\xj$ for $j \ne i,i+1$. Therefore, for an arbitrary
polynomial $f$, we have the following immediate consequence (see Remark \ref{forcingintro}):

\begin{prop} \label{forcediagram} We may force a polynomial to the other side of a line, leaving at most $\xi$ behind, as follows:

\begin{eqnarray}
\ig{.4}{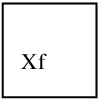}\ \ \ig{.8}{line.eps} \ & \ =
  \ &\ \ \ig{.8}{line.eps}\ \ \ig{.4}{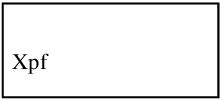}\ \ +\ \ \ig{.5}{multxi.eps} 
   \ \ \ig{.8}{line.eps}\ \  \ig{.4}{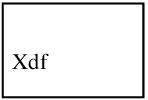}
 \label{eq-slide5} \\  
  \ \  &  &  \ \  \nonumber  \\
\ig{.8}{line.eps}\ \ \ig{.4}{boxf.eps} \ & \ = \ &\  
 \ig{.4}{boxpf.eps} \ \ \ig{.8}{line.eps}\ \ +\ \ig{.4}{boxdf.eps} \ \ \ig{.8}{line.eps}\ \ 
\ig{.5}{multxi.eps} \label{eq-slide6} 
\end{eqnarray}
\end{prop}

\begin{proof} This is proven the same way as (\ref{force}). \end{proof}

Now for the final one-color relations. First, the \emph{dot relations}:

\begin{equation}\ig{.5}{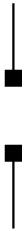} \ = \  \ig{.4}{multxi.eps}\ \ig{.75}{line.eps} \ -
   \  \ig{.75}{line.eps}\ \ig{.4}{multxp.eps} \ = \  -\ \ig{.4}{multxp.eps}\ \ig{.75}{line.eps} \ +
   \  \ig{.75}{line.eps}\ \ig{.4}{multxi.eps}\label{dotSpaceDot}\end{equation}

\begin{equation}\ig{.5}{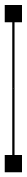} \ = \ \ig{.4}{multxi.eps} \ - \ \ig{.4}{multxp.eps} 
 \label{doubleDot}\end{equation}

The second equality in (\ref{dotSpaceDot}) is just the relation (\ref{eq-slide1}). Now for the \emph{needle
relation}:

\begin{equation}\ig{.35}{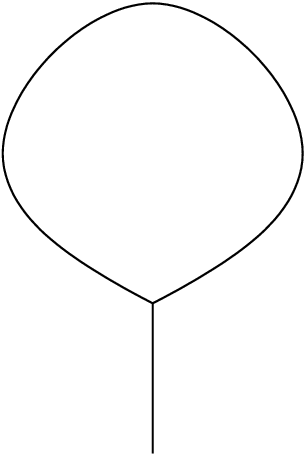} \ \ = \ \ 0 \label{needle} \end{equation}

It is important to realize that such a relation does \emph{not} apply if there is anything inside the eye of
the needle, as can be seen in the following examples.

\begin{example}
	Combining these relations, we have a number of simple but important consequences, which we leave as easy exercises to get the reader used to the diagrammatic calculus.
		\begin{equation} \ig{.35}{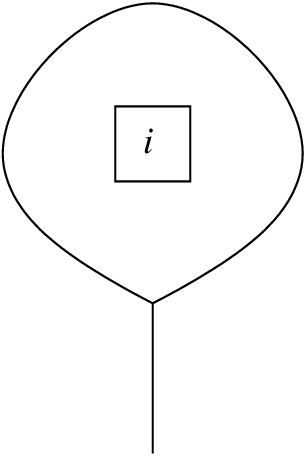} \ \ = \ \ \ \  \ig{.35}{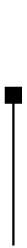} \label{needleWithEye} \end{equation}
			Use the first dot relation, then the needle relation and the unit relation.
		\begin{equation} \label{needle.eps} \ig{.3}{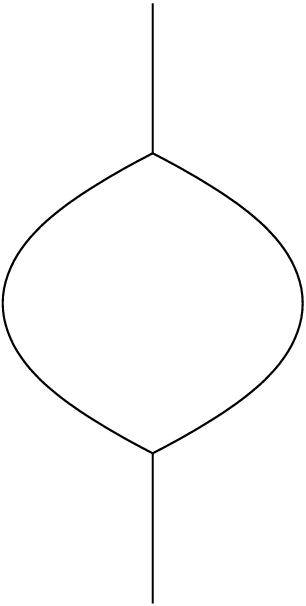} \ \ = \ 0 \end{equation}
			Use the needle relation and associativity.
		\begin{equation} \ig{.25}{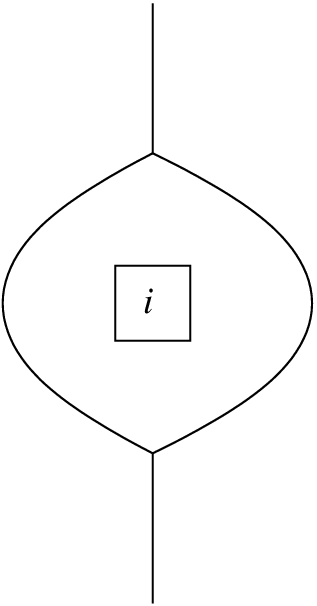} \ \ \ = \ \ \ \  \ig{.25}{tallLine.eps} \end{equation}
			As above, with the unit relation.
		\begin{align*} \ig{.35}{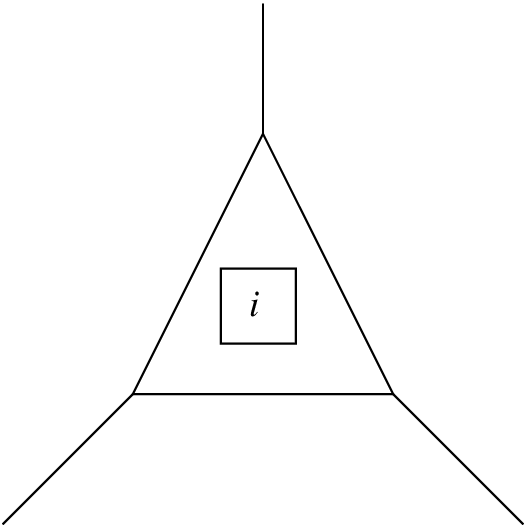}& \quad = \quad \ig{.3}{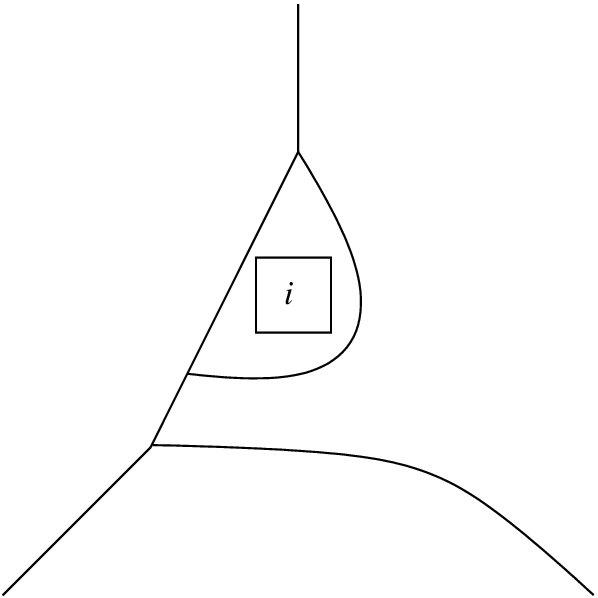} \quad = 
			\quad \ig{.3}{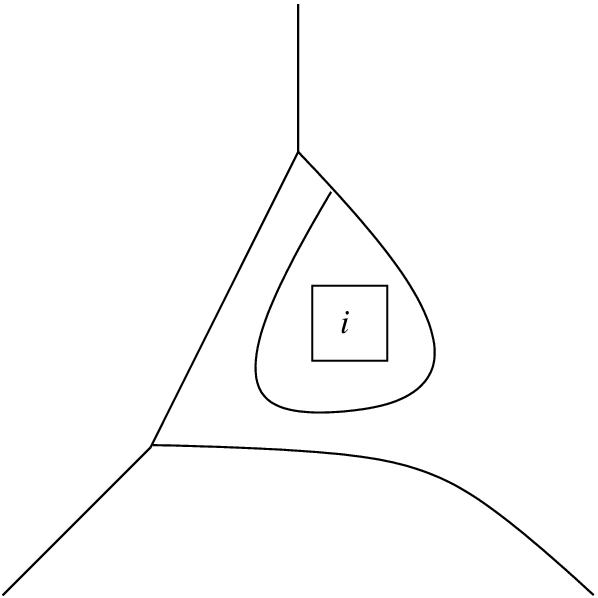}\\
			& \quad = \quad \ig{.3}{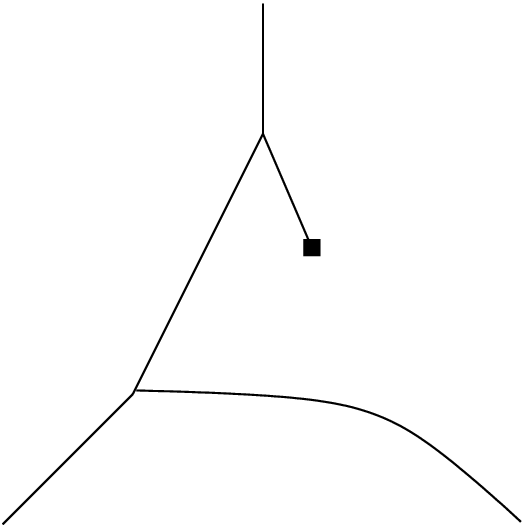} \quad = \quad \ig{.3}{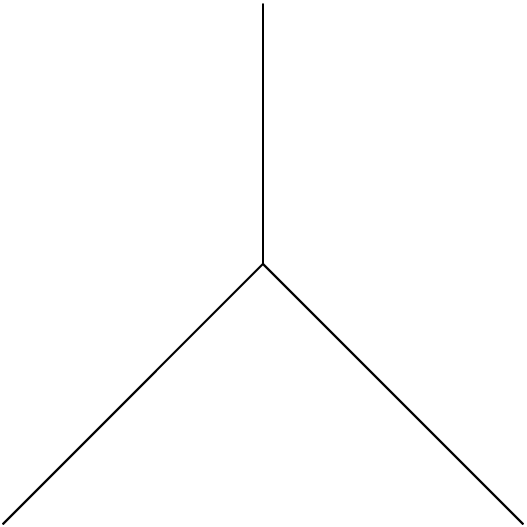} \end{align*}
			This is effectively the same example again, with more uses of associativity.
\end{example}

As the examples demonstrate, and the following proposition proves, we may remove cycles of this nice form from
a one-color graph.

\begin{prop} The following relations hold.

\begin{equation}\ig{.35}{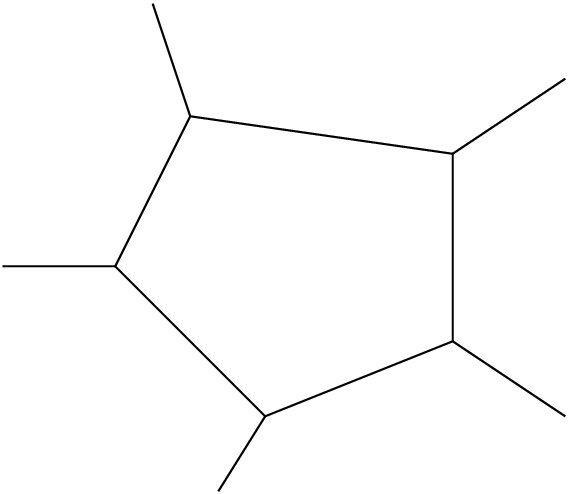} \ = \ \ \  0  \end{equation} 

\begin{equation} \ig{.35}{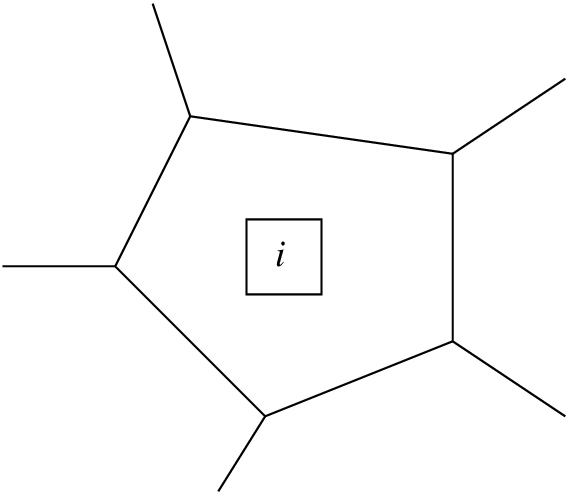} \ = \ \ \ \  \ig{.35}{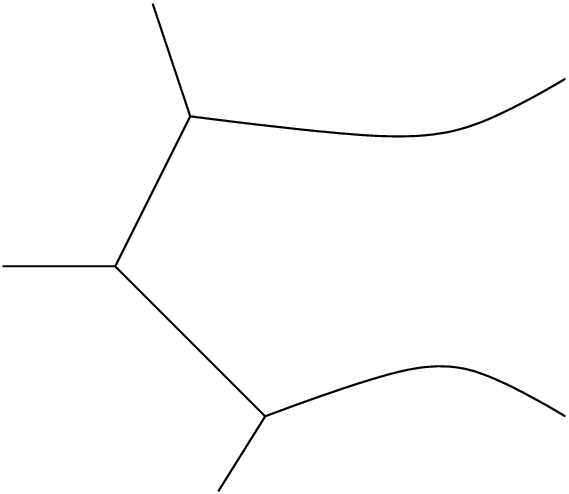} \end{equation}

More generally, for any polynomial $f$ we have 

\begin{equation} \ig{.4}{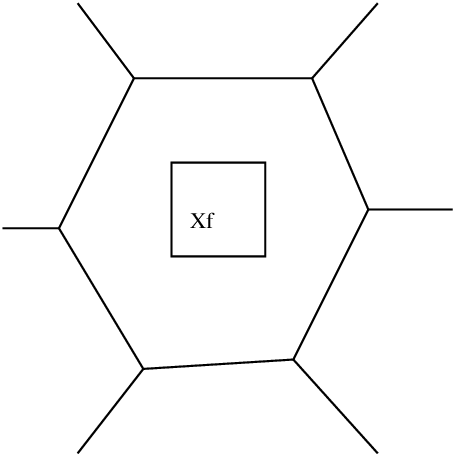} \ \ \ = \ \ \ \ig{.4}{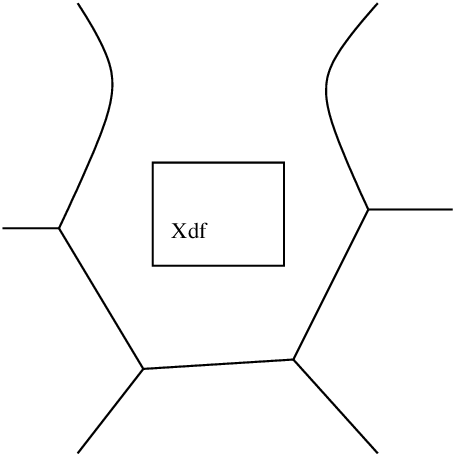} \label{holepoly} \end{equation} 
\end{prop}

\begin{proof} This is a simple consequence of (\ref{forcediagram}), along with the needle, associativity, and
unit relations. \end{proof}

There is another relation which is equivalent (given the others) to the first equality in (\ref{dotSpaceDot}).

\begin{align} \ig{.4}{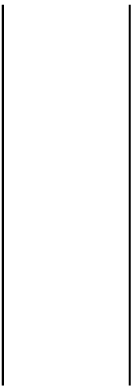}& \ \  = \ \  \ig{.4}{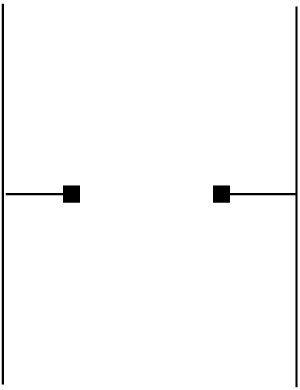} \ \  = \ 
\ig{.4}{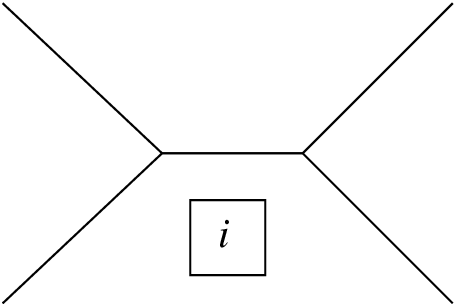} \ - \ \ig{.4}{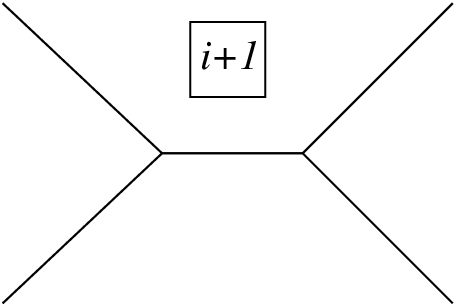} \label{twoLines}\\   \ \nonumber   \\
 &= \ \ig{.4}{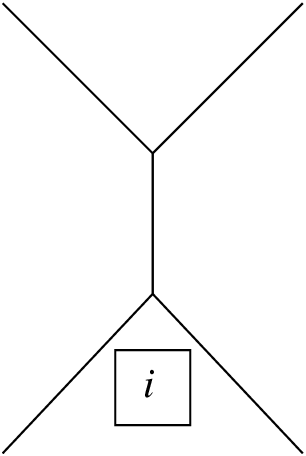} \ -\ \ig{.4}{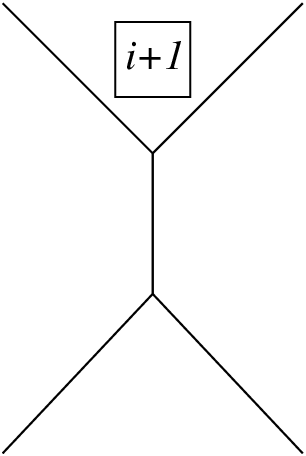} \nonumber \end{align}

This relation quickly leads to the decomposition $B_i \TenR B_i = B_i\{1\} \oplus B_i\{-1\}$, see Section
\ref{subsec-functor}.

For a single color and two variables $x_1,x_2$, the category above, modulo the relation 
$x_1=x_2$,  is equivalent to the category considered by Libedinsky~\cite{Lib3} in the case 
of a single label $r$. Morphisms given by dot, Merge, Split, and Cap correspond to 
morphisms $\hat{\epsilon_r}$, $\hat{m}_r$, $\hat{p}_r$, $\hat{j}_r$ and  
$\hat{\alpha}_r$ in~\cite[Section 2.4]{Lib3}. Planar graphical notation, of paramount 
importance to us, is implicit in~\cite{Lib3}. 
From here on, we diverge from Libedinsky's work, by generalizing to the 
case of the Weyl group $S_{n+1}$, while Libedinsky~\cite{Lib3} investigates the right-angled case.

%
\subsection{The category $\mc{DC}_1$: adjacent colors}
\label{subsec-dc1adj}

We now add some generators which mix adjacent colors, which we call \emph{6-valent vertices}. Remember that the
thick lines represent $i+1$, and the thin lines represent $i$.

$$
\begin{array}{cc}
\begin{array}{cc}
  \mathrm{Symbol} & \mathrm{Degree} \\
  \\
  \igv{.4}{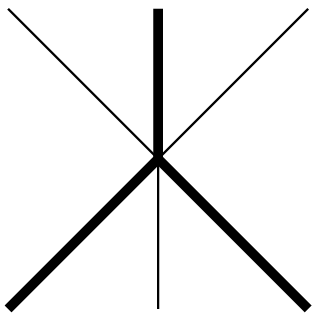} & 0\end{array} &
\begin{array}{cc}
  \mathrm{Symbol} & \mathrm{Degree} \\
  \\
  \ig{.4}{6valent.eps} & 0\end{array}
\end{array}$$

For the beginner, these maps are respectively: a map from $i(i+1)i$ to $(i+1)i(i+1)$, a map from $(i+1)i(i+1)$
to $i(i+1)i$.

Below are the relations which deal with our new generators. In addition to the relations below, we also impose
the same relations with the colors switched. The two color variants in general do not imply each other.
However, it is better to think of the two colors as being arbitrary adjacent colors, rather than one being $i$
and the other $i+1$; then one views these relations as generic for adjacent colors.

\begin{equation}
\ig{.2}{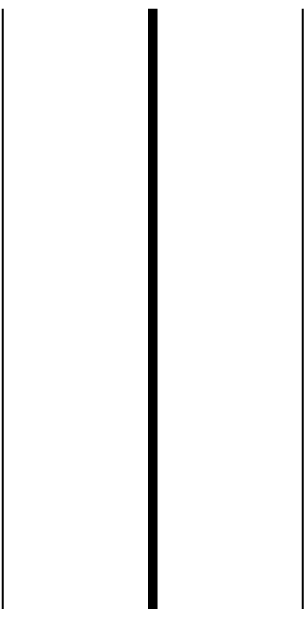}\ \ =\  \ \ig{.2}{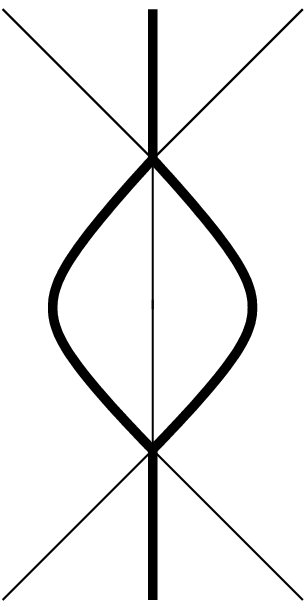} \ \ - \ \ \ig{.2}{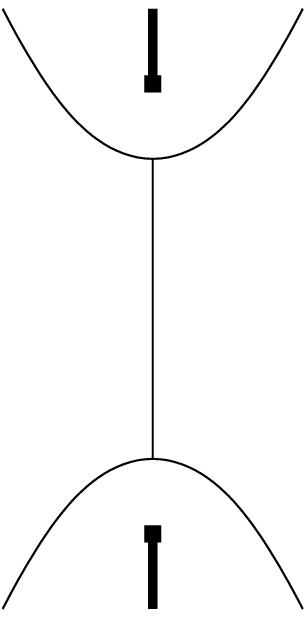} 
\label{eqn-threeLines}
\end{equation} 

\begin{equation}
\ig{.32}{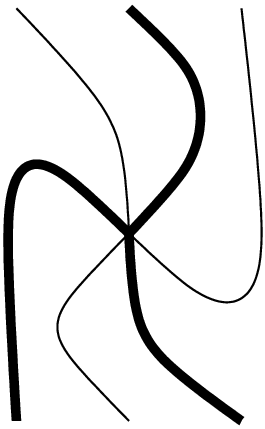} \ \ \  =\ \ \   \ig{.4}{6valent.eps} \ \ \ = 
\ \ \   \igh{.32}{ipipipRot1.eps} \label{eqn-ipipipRot}
\end{equation}

\begin{equation}
\ig{.3}{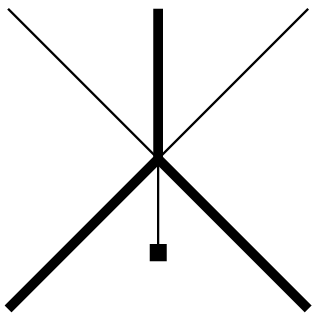} \ \ = \ \ \ \ig{.3}{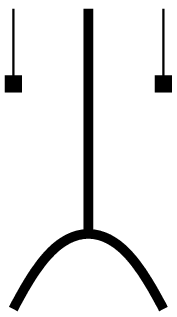} \ \ + \ \ \igh{.3}{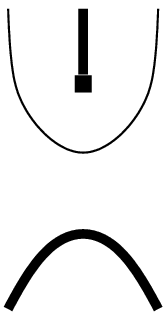} 
\label{eqn-ipipipDot}
\end{equation}

\begin{equation}
\ig{.3}{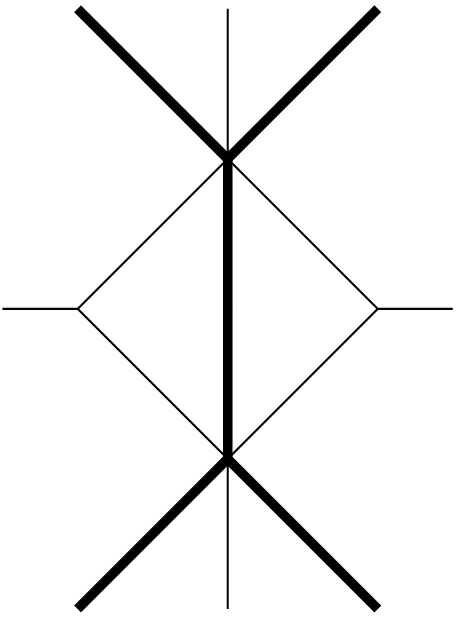} \ \ = \ \ \igrotCW{.3}{ipipipAss.eps} \ \ = 
\ \ \ig{.3}{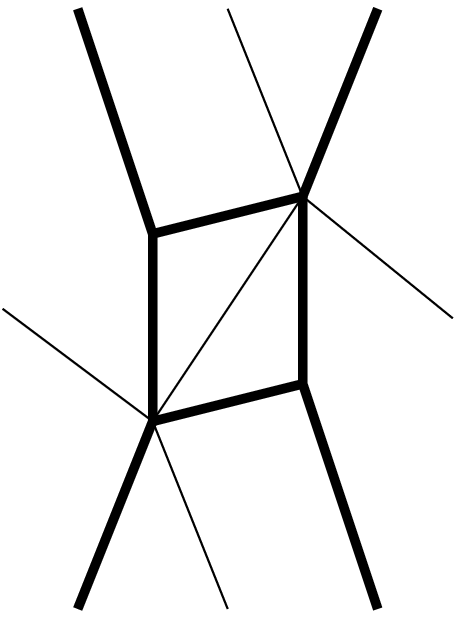} \ \ = \ \  \igh{.3}{ipipipAss2.eps} \label{eqn-ipipipAss}
\end{equation}

\begin{equation}
\ig{.3}{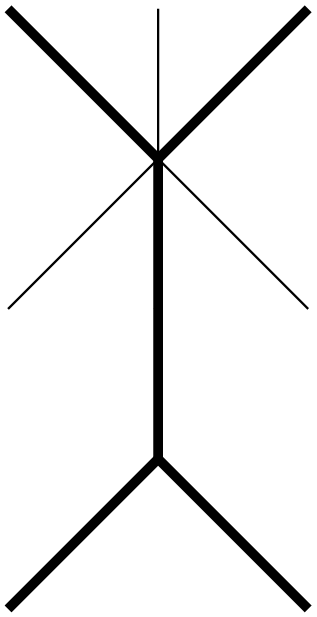} \ \ \ \ = \ \ \ \  \ig{.3}{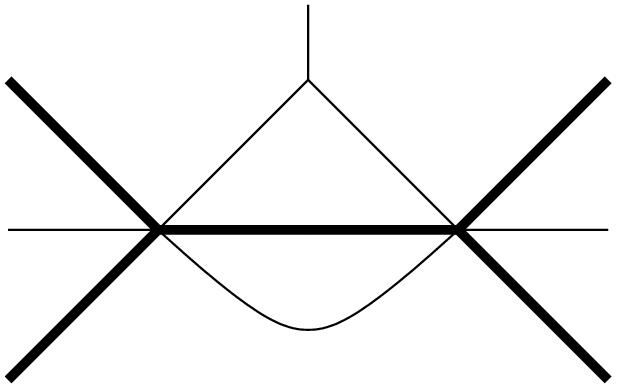} \label{eqn-ipipipAssWDot}
\end{equation}

It will be shown in Section \ref{subsec-functor} that the first relation is related to the isomorphism $(B_i
\TenR B_{i+1} \TenR B_i )\oplus B_{i+1}\cong (B_{i+1} \TenR B_i \TenR B_{i+1}) \oplus B_i$.

The second relation shows that the 6-valent vertex is cyclic, that drawing it as a 6-valent vertex is
unambiguous, and that isotopy classes of diagrams built out of our local generators will still unambiguously
designate a morphism. See Remark \ref{cyclicity} for details. Because of this, we have used the liberties of
Remark \ref{liberties} when writing the last two relations. Note that (\ref{eqn-ipipipRot}) does in fact imply
the color-switched version of that same relation, using (\ref{biadjoint}).

The relation (\ref{eqn-ipipipAss}) contains a number of equalities, and it is clear that the last equality is
merely a rotation of the color switch of the first equality. In fact, there are numerous redundancies amongst
(\ref{eqn-ipipipAss}) and (\ref{eqn-ipipipAssWDot}). It is a worthwhile exercise for the reader at this point
to check the following statement.

\begin{example} Assume the relation (\ref{eqn-ipipipDot}) and those before it. Then any pair of equalities from
(\ref{eqn-ipipipAss}) will imply both color variants of (\ref{eqn-ipipipAssWDot}) as well as the remainder of
the equalities from (\ref{eqn-ipipipAss}). Hint: adding a dot to the relation (\ref{eqn-ipipipAss}) allows one
to recover (\ref{eqn-ipipipAssWDot}), while the latter may be applied twice within the former. \end{example}

An important feature to notice is that the 6-valent vertex can be visualized as two trivalent vertices, one of
each color, that overlap. If one takes a graph constructed out of dots, trivalent vertices, and 6-valent
vertices (our generators so far), then the subgraph formed by all edges of a specific color $i$ will have only
univalent and trivalent vertices. We use the term \emph{two-color (overlap) associativity for} $i+1$ to refer
to the transformation performed by either (\ref{eqn-ipipipAssWDot}) or the first equality of
(\ref{eqn-ipipipAss}), because when viewed as an operation on the ``thick''-colored graph, these operations
mimic one-color associativity (\ref{eqn-associativity}). Note that, under the same transformations, the
``thin''-colored graph (labelled $i$) is transformed in a different way. However, the color-switched relations
will give two-color associativity for $i$ instead.

%
\subsection{The category $\mc{DC}_1$: distant colors}
\label{subsec-dc1dist}
%

Fix $j$, an index which is not adjacent to $i$. In pictures involving both $i$ and $i+1$, we also
assume $j$ is not adjacent to $i+1$. Remember that $j$ is represented by a dashed line. This new
generator is called a \emph{4-valent vertex}, or a \emph{crossing}.

$$
\begin{array}{cc}
  \mathrm{Symbol} & \mathrm{Degree}\\
  \\
  \igv{.5}{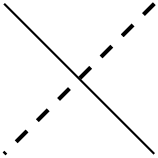} & 0
\end{array}
$$

Note that this definition also covers the same picture with the colors reversed. The colors $i$ and
$j$ can be switched freely since the only requirement was that they were distant from each other.

Now for relations involving the new generator.

\begin{equation}
\ig{.4}{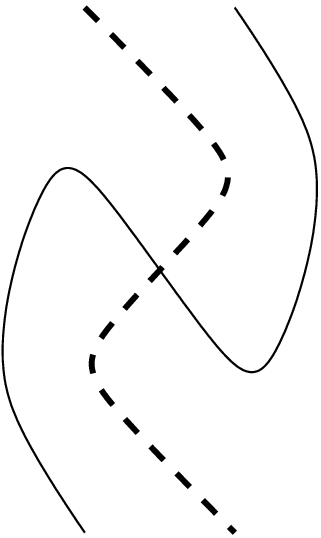}\ \ = \ \ \ig{.4}{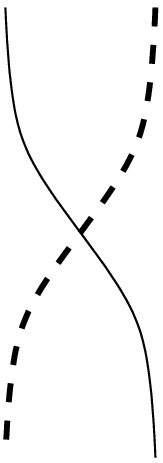}\ \ =\ \ \ig{.4}{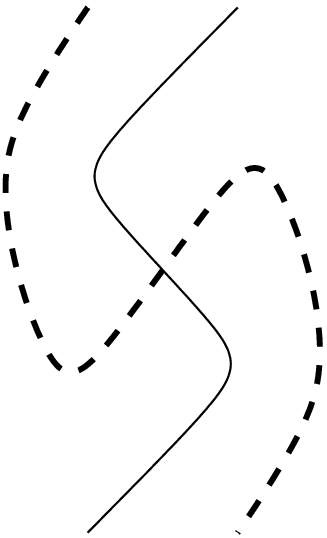} \label{eqn-ijijRot}
\end{equation}
\begin{equation}
\ig{.5}{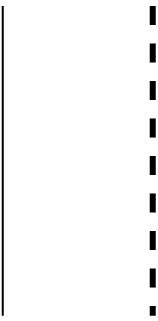}\ \ =\ \ \ig{.5}{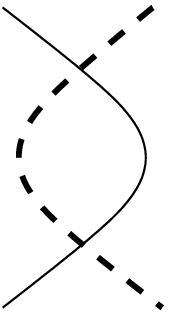} \label{eqn-R2move}  
\end{equation}
\begin{equation}
\ig{.6}{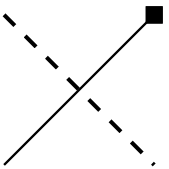} \ \ = \ \ \ig{.6}{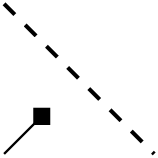} \label{eqn-ijijDot}
\end{equation}
\begin{equation}
\ig{.4}{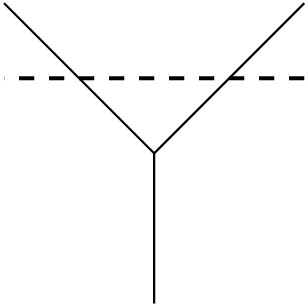} \ \ = \ \ \ig{.4}{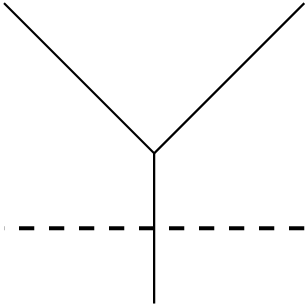} 
\label{eqn-pullFarThruTrivalent} 
\end{equation}
\begin{equation}
\ig{.4}{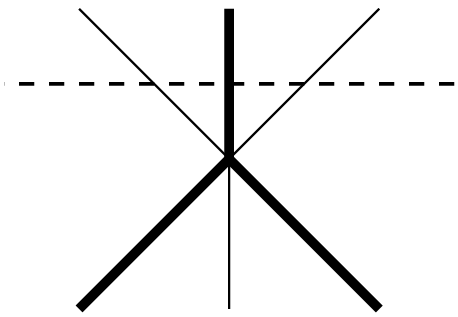} \ \ = \ \ \ig{.4}{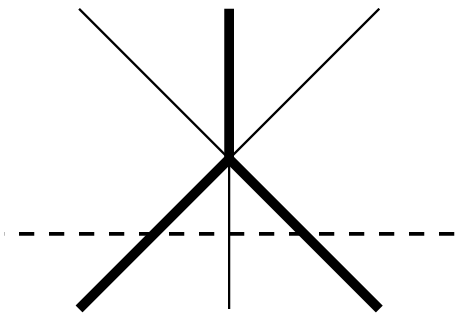}  
\label{eqn-pullFarThru6Valent} \end{equation}
\begin{equation}
\ig{.4}{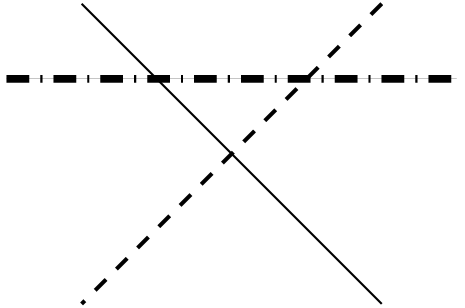} \ \ = \ \ \ig{.4}{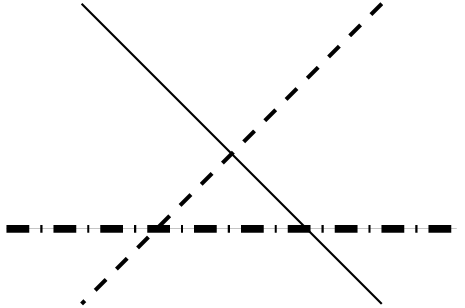}  \label{eqn-R3move}
\end{equation}

Relation (\ref{eqn-pullFarThru6Valent}) holds when you switch $i$ and $i+1$, but one color variant will follow
quickly from the other by twisting and applying the first relation. In the final relation, the new color
represents an index $k$ which is distant from both $i$ and $j$. We will refer to (\ref{eqn-R2move}) and
(\ref{eqn-R3move}) as the R2 and R3 moves respectively, because of the obvious analogy to knot theory. The R2
relation is essentially the isomorphism $B_i\TenR B_j \cong B_j \TenR B_i$, see Section~\ref{subsec-functor}.

The same statements about cyclicity and drawing diagrams with sideways boundaries apply from before (see
Remarks \ref{cyclicity} and \ref{liberties}). Once again, the 4-valent vertices are drawn so that morphisms are
isotopy invariant.

The relations (\ref{eqn-R2move})-(\ref{eqn-R3move}) imply that a $j$-colored strand can just be pulled
underneath any morphism only using colors distant from $j$, since it can be pulled under any generating
morphism, whether it be a line, a dot, a trivalent vertex, or a 6-valent vertex. In fact, thanks to
(\ref{unit}), the R2 move follows from (\ref{eqn-ijijDot}) and (\ref{eqn-pullFarThruTrivalent}).

We have now listed all the generators of our subcategory: trivalent, 4-valent, and 6-valent
vertices, and dots. There is one final relation, coming from the fact that $i+1$ and $i-1$ may
not interact, but they do jointly interact with $i$.  The final relation will be called \emph{three-color (overlap) associativity for} $i \pm 1$:

\begin{equation} \ig{.4}{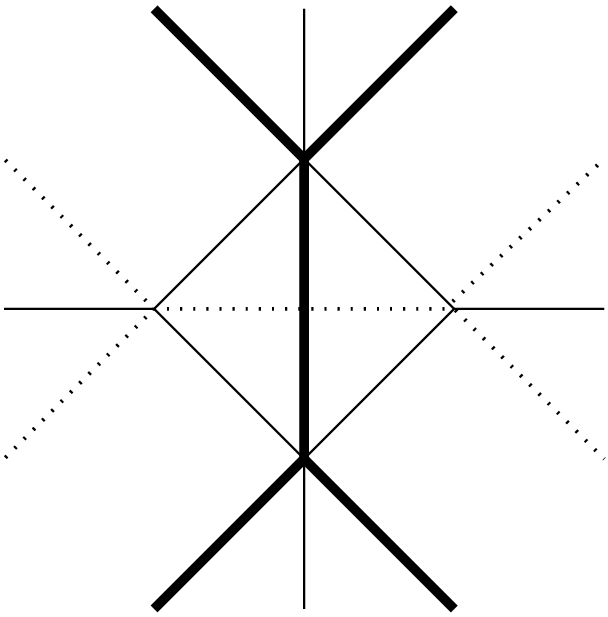} =
  \ig{.4}{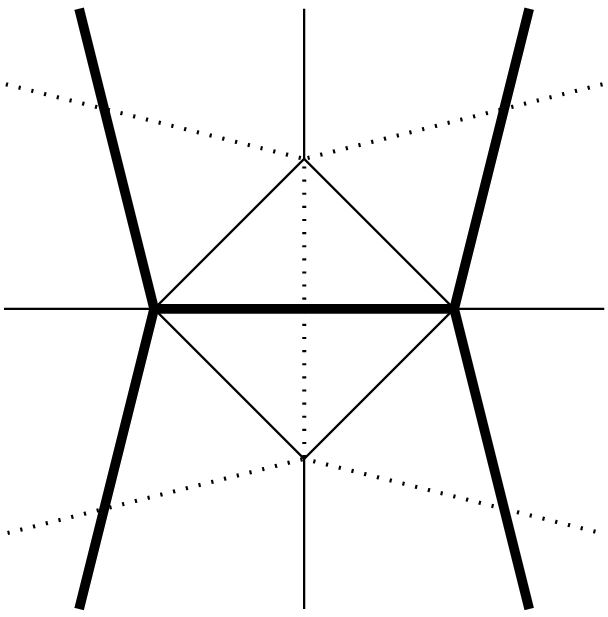} \label{eqn-threeColorAss} \end{equation}

In the above diagram, dotted lines carry label $i-1$, thick lines $i+1$ and thin solid lines $i$. Rotating this
relation by 90 degrees, we get the same relation except with $i+1$ and $i-1$ switched, so that only color
variant is needed to imply both. The ``thick''-colored graph undergoes the associativity transformation. The
same is true (symmetrically) with the ``dotted''-colored graph.

This concludes the definition of the category $\mc{DC}_1$.

%
\subsection{The complete definition and the functor $\mc{F}_1$}
\label{subsec-fulldefn}
%

In order to put everything in one place with no redundancy, let us define the category again.

\begin{defn} The category $\mc{DC}_1$ has objects given by sequences $\ii$ of indices in $\{1,\ldots,n\}$, with
a monoidal structure given by concatenation. Fix two sequences $\ii$ and $\jj$. Consider the set of all
diagrams in $\R \times [0,1]$, constructed out of vertical lines colored by indices, and out of the generating
pictures below, such that the intersection of the diagram with $\R \times \{0\}$ is the sequence of points
$\ii$, and the intersection with $\R \times \{1\}$ is $\jj$. This set is graded, where the generators have the
degree indicated. Then the space $\Hom_{\mc{DC}_1}(\ii,\jj)$ is defined to be the $\Bbbk$-linear span of this
set of diagrams, modulo the homogeneous local relations below.

{\bf Generators}

\begin{itemize}
\item For each color, pictures of degree $1$: $\igv{.3}{dot.eps} \ \ \ \ \ \ig{.3}{dot.eps}$
\item For each color, pictures of degree $-1$: $\igv{.3}{merge.eps} \ \ \ \ \ \ig{.3}{merge.eps}$
\item For each pair of distant colors, a picture of degree $0$: $\ig{.5}{4valent.eps}$
\item For each pair of adjacent colors, a picture of degree $0$: $\ig{.4}{6valent.eps}$
\item For each $i \in \{1, \ldots, n+1\}$, a picture of degree $2$: $\ig{.4}{multxi.eps}$
\end{itemize}

{\bf Relations} 

Some relations are drawn using the liberties of Remark \ref{liberties}. We also use the
definition of the cap and cup:

\[ \igv{.3}{cap.eps} \define \igv{.3}{dotSplit.eps}\quad\quad \ig{.3}{cap.eps} \define \ig{.3}{dotSplit.eps} \]

For each color:

\[\ig{.25}{cupCap.eps} \;\; = \;\; \ig{.2}{tallLine.eps} \;\; = \;\;  \igh{.25}{cupCap.eps} \quad \quad \quad
\ig{.3}{splitCup.eps} \  = \ \ig{.3}{tallMerge.eps} \  = \ \igh{.3}{splitCup.eps} \]

\[\ig{.2}{doubleMerge.eps} = \igh{.2}{doubleMerge.eps} \quad \quad
\ig{.3}{capDot.eps} \  =  \ \igv{.4}{dot.eps}  \ =  \ \igh{.3}{capDot.eps} \quad \quad
\ig{.25}{needle2.eps} \ \ = \ \ 0\]

Here, $i$ is the color of the line, and $j$ is a color $\ne i,i+1$.

\[\ig{.4}{multxi.eps}\ \ig{.6}{line.eps}\ +\ \ig{.4}{multxp.eps}\ \ig{.6}{line.eps}\ \ =
  \ \ \ig{.6}{line.eps}\ \ig{.4}{multxi.eps}\ +\ \ig{.6}{line.eps}\ \ig{.4}{multxp.eps}\]

\[\ig{.4}{multxi.eps}\ \ig{.4}{multxp.eps}\ \ig{.6}{line.eps}\ \ =\ \ \ig{.6}{line.eps}\ 
  \ig{.4}{multxi.eps}\ \ig{.4}{multxp.eps}\]

\[\ig{.4}{multxj.eps}\ \ig{.6}{line.eps}\ \ =\ \ \ig{.6}{line.eps}\ \ig{.4}{multxj.eps}\]

\[\ig{.5}{dotSpaceDot.eps} \ = \  \ig{.4}{multxi.eps}\ \ig{.75}{line.eps} \ -
   \  \ig{.75}{line.eps}\ \ig{.4}{multxp.eps} \ = \  -\ \ig{.4}{multxp.eps}\ \ig{.75}{line.eps} \ +
   \  \ig{.75}{line.eps}\ \ig{.4}{multxi.eps}\]

\[\ig{.5}{doubleDot.eps} \ = \ \ig{.4}{multxi.eps} \ - \ \ig{.4}{multxp.eps}\]

For any two adjacent colors:

\[\ig{.3}{threeLines1.eps}\ \ =\  \ \ig{.3}{threeLines4.eps} \ \ - \ \ \ig{.3}{threeLines3.eps}\]

\[\ig{.32}{ipipipRot1.eps} \ \ \  =\ \ \   \ig{.4}{6valent.eps} \ \ \ = \ \ \   \igh{.32}{ipipipRot1.eps}\]

\[\ig{.3}{ipipipDot1.eps} \ \ = \ \ \ \ig{.3}{ipipipDot2.eps} \ \ + \ \ \igh{.3}{ipipipDot3.eps}\]

\[\ig{.3}{ipipipAss.eps} \ \ = \ \ \igrotCW{.3}{ipipipAss.eps}\]

For any two distant colors:

\[\ig{.4}{ijijRot1.eps}\ \ = \ \ \ig{.4}{4valent2.eps}\ \ =\ \ \ig{.4}{ijijRot2.eps}\]

\[\ig{.6}{ijijDot1.eps} \ \ = \ \ \ig{.6}{ijijDot2.eps}\]

\[\ig{.4}{pullFarThruTrivalent1.eps} \ \ = \ \ \ig{.4}{pullFarThruTrivalent2.eps}\]

For two adjacent colors and a third, distant to both:

\[\ig{.4}{pullFarThru6Valent1.eps} \ \ = \ \ \ig{.4}{pullFarThru6Valent2.eps}\]

For three mutually distant colors:

\[\ig{.4}{R3move.eps} \ \ = \ \ \ig{.4}{R3move2.eps}\]

For three colors with the same adjacency as $\{1,2,3\}$:

\[\ig{.4}{threeColorAss1.eps} = \ig{.4}{threeColorAss2.eps}\]

\end{defn}

\begin{defn} Let $\mc{F}_1$ be the functor from $\mc{DC}_1$ to $\mc{SC}_1$ specified as follows. On objects,
$\mc{F}_1(\ii) = B_{\ii}$. We define the functor on generating morphisms and extend it monoidally to all
morphisms.

In doing so, we \emph{always} use the isomorphism (\ref{useoften}) to identify $B_{\ii}$ with the $R$-bimodule
spanned by a choice of $d(\ii)+1$ polynomials. If one thinks of $B_{\ii}$ diagrammatically as $d$ vertical
lines, then a spanning element of $B_{\ii}$ is a choice of polynomial for each empty region delineated by the
lines (and polynomials with the appropriate symmetry may slide across the lines). We write the map explicitly
for a general element when it is easy enough to do so, or we write it for a spanning set as an $R$-bimodule
(see Remark \ref{spanningintro}).

For a line colored $i$:

\[
\begin{array}{cc}
\mathrm{Symbol} & \mc{F}_1 \\
\\
\igv{.3}{dot.eps} & \auptob{f \teni g}{fg}\\
\\
\ig{.3}{dot.eps} & \auptob{1}{\xi \teni 1 - 1 \teni \xp} \\
\\
\igv{.3}{merge.eps} & \begin{array}{ccc} \auptob{f \teni 1 \teni g}{0} & 
\quad & \auptob{f \teni \xi \teni g}{f \teni g} \end{array} \\
\\
\ig{.3}{merge.eps} & \auptob{f \teni g}{f \teni 1 \teni g}
\end{array}
\]

For lines colored $i$ and $j$ distant:

\[\begin{array}{cc} \igv{.5}{4valent.eps} & \auptob{f \teni 1 \tenj g}{f \tenj 1 \teni g} \end{array}\]

For a thin line colored $i$ and a thick line colored $i+1$:

\[
\begin{array}{cc}
  \igv{.4}{6valent.eps} & \begin{array}{cc} \auptob{1 \teni 1 \tenp 1 \teni 1}
 {1 \tenp 1 \teni 1 \tenp 1} & \auptob{1 \teni \xi \tenp 1 \teni 1}{(\xi + \xp) \tenp 1 
 \teni 1 \tenp 1 - 1 \tenp 1 \teni 1 \tenp \xqq} \end{array} \\
  \\
  \ig{.4}{6valent.eps} & \begin{array}{cc} \auptob{1 \tenp 1 \teni 1 \tenp 1}
 {1 \teni 1 \tenp 1 \teni 1} & \auptob{1 \tenp \xqq \teni 1 \tenp 1}
 {1 \teni 1 \tenp 1 \teni (\xp + \xqq) - \xi \teni 1 \tenp 1 \teni 1} \end{array}
\end{array}
\]

For any $1 \le i \le n+1$:

\[\begin{array}{cc} \ig{.4}{multxi.eps} & \auptob{1}{\xi} \end{array}\]
\end{defn}

\begin{claim} The above maps are $R$-bimodule maps. \end{claim}

\begin{proof} This is obviously true for EndDot, since the resulting map is no more than multiplication.
StartDot is sent precisely to the generator $\phi_i$ of $\Hom(R,B_i)$ discussed in
Section~\ref{subsec-sbimcat}. Split and Merge have already been seen as inclusion and projection maps in the
isomorphism $B_i \TenR B_i \cong B_i\{1\} \oplus B_i\{-1\}$, see Remark \ref{forcingintro}. For the 4-valent
vertex, the only polynomials which slide all the way across $B_i \TenR B_j$ or $B_j \TenR B_i$ are in
$R^{i,j}$, so that the map $f \teni 1 \tenj g \to f \tenj 1 \teni g$ is a bimodule map (that $f \teni 1 \tenj
g$ spans it was observed in Remark \ref{spanningintro}).

Only the 6-valent vertices remain to be checked. Consider the first of the two variants. The generating set
$\{1 \teni 1 \tenp 1 \teni 1, 1 \teni \xi \tenp 1 \teni 1\}$ as an $R$-bimodule was chosen because $\xi$ can be
slid freely between the second and third slots. We have defined the $R$-bimodule map on generators before
showing that the map is an $R$-bimodule map at all, which is akin to putting the cart before the horse. Let us
explicitly define the map on a $\Bbbk$ spanning set by the following algorithm: given $f \teni g \tenp h \teni
k \in B_i \TenR B_{i+1} \TenR B_i$, first we force $h$ to the right and slide the ``remainder" to the left,
that is $f \teni g \tenp h \teni k = f \teni g \tenp 1 \teni kP_i(h) + f \teni g\xi \tenp 1 \teni
k\partial_i(h)$; then we force the terms in the second slot to the left, yielding $fP_i(g) \teni 1 \tenp 1
\teni kP_i(h) + f \partial_i(g) \teni \xi \tenp 1 \teni kP_i(h) + fP_i(g\xi) \teni 1 \tenp 1 \teni
k\partial_i(h) + f \partial_i(g\xi) \teni \xi \tenp 1 \teni k\partial_i(h)$. Finally, each term can be
evaluated using the given definition of $\mc{F}_1$ on generators. This gives an explicit formula for the image
of $f \teni g \tenp h \teni k$, which we only need check is invariant under: sliding an element of $R^i$ from
$f$ to $g$, or from $h$ to $k$; sliding an element of $R^{i+1}$ from $g$ to $h$. Sliding elements of $R^i$ does
not pose a problem, since we defined the map by forcing $h$ to $k$ and $g$ to $f$, which fully respects such
slides. Checking invariance under slides from $g$ to $h$ is non-trivial. However, the bulk of the work is
encapsulated in the following discussion, which is useful for calculations in general.

By adding and subtracting $\xqq$, the image of $1 \teni \xi \tenp 1 \teni 1$ under the first 6-valent vertex (see above) can be
written more symmetrically as $(\xi+\xp+\xqq)(1 \tenp 1 \teni 1 \tenp 1) - \xqq \tenp 1 \teni 1 \tenp 1 - 1 \tenp 1 \teni 1
\tenp \xqq$. The first term is a polynomial symmetric in all the relevant variables and thus can be slid anywhere. In the other
two terms, $\xqq$ can not be slid freely under a line labelled $i+1$, so it is stuck in its respective position. In contrast,
$1 \teni \xp \tenp 1 \teni 1$ and $1 \teni 1 \tenp \xp \teni 1$ are not equal, since $\xp$ can not be slid over a line labelled
$i+1$, but the images of both these elements are easier to remember, and are shown below.

\begin{eqnarray*}
  \igv{.4}{6valent.eps} & \begin{array}{cc} \auptob{1 \teni \xp \tenp 1 \teni 1}{ 1 \tenp 1 \teni 1 \tenp \xqq} & \auptob{1 \teni 1 \tenp \xp \teni 1}{\xqq \tenp 1 \teni 1 \tenp 1} \end{array} \\
  \\
  \ig{.4}{6valent.eps} & \begin{array}{cc} \auptob{1 \tenp \xp \teni 1 \tenp 1}{ 1 \teni 1 \tenp 1 \teni \xi} & \auptob{1 \tenp 1 \teni \xp \tenp 1}{\xi \teni 1 \tenp 1 \teni 1} \end{array}
\end{eqnarray*}

The way to remember these formulae is that the variable which can't be slid is sent to the variable which can't
be slid, from the middle on one side to the exterior on the other. It is easy to see that these calculations
were done according to the algorithm above, forcing $\xp$ to the outside first and then evaluating on the
leftover $\xi$.

Now we do the consistency check for the simplest cases. We wish to show that $1 \teni (\xp + \xqq) \tenp 1
\teni 1$ and $1 \teni 1 \tenp (\xp + \xqq) \teni 1$ are sent to the same element by the algorithm. However,
this is rather easy, for in both cases, the $\xqq$ term slides immediately to the exterior, and the $\xp$ term
is evaluated as above, so both are sent to $\xqq \tenp 1 \teni 1 \tenp 1 + 1 \tenp 1 \teni 1 \tenp \xqq$.
Similarly, both $1 \teni \xp\xqq \tenp 1 \teni 1$ and $1 \teni 1 \tenp \xp\xqq \teni 1$ are sent to $\xqq \tenp
1 \teni 1 \tenp \xqq$. The general case is not significantly more difficult than this; we leave the details to
the reader.
\end{proof}

\begin{prop} \label{f1isafunctor} The functor $\mc{F}_1$ is well-defined. That is, the relations of $\mc{DC}_1$
hold between morphisms of $R$-bimodules in $\mc{SC}_1$. \end{prop}

Checking that the relations hold is a series of simple but tedious calculations that is postponed until
Section~\ref{subsec-f1}. We assume this result henceforth. In addition, we note once and for all that (as one
can easily check) all relations in the definition of $\mc{DC}_1$ are homogeneous, and $\mc{F}_1$ preserves the
degree of the generators.

\begin{remark} {\bf Addendum} It may strike the reader as unusual that the definition of the functor seems
lopsided, while the definition of $\mc{DC}_1$ is invariant under right-left reflection, or under reversing the
order of the colors $n,n-1,\ldots,1$. For instance, $\mc{F}_1$ applied to StartDot yields the element $\xi
\teni 1 - 1 \teni \xp$, which is actually invariant under right-left reflection but not immediately so. Had
this element been rewritten as $\frac{\xi - \xp}{2} \teni 1 + 1 \teni \frac{\xi - \xp}{2}$, perhaps the
calculations would be more natural despite having more fractions. A worse offender is the forcing rule $1 \teni
g = (g - \partial_i(g)\xi) \teni 1 + \partial_i(g) \teni \xi$, which should be rewritten $1 \teni g = \frac{g +
s_i(g)}{2} \teni 1 + \partial_i(g) \teni \frac{\xi - \xp}{2}$. In general, the elements $\xi - \xp$ are more
natural than $\xi$ or $\xp$, coming from the reflection representation rather than the standard representation
of $S_{n+1}$ (see Remark \ref{dumbchange} and Section \ref{subsec-quotient}).

Previous versions of this paper, however, used the style above, and so we feel compelled to stick with it to
maintain consistency. Also, checking that $\mc{F}_1$ is a functor may be easier with the current notation.
\end{remark}

\section{Consequences}
\label{sec-conseq}

%
\subsection{Terminology}
\label{subsec-terminology}
%

We will spend the next few sections classifying the homomorphisms in $\mc{DC}_1$. For many of the results,
proofs will be postponed until Section \ref{subsec-graphproof}.

We will we using the fact, extensively discussed in the previous sections, that a morphism can be viewed
unambiguously as an isotopy class of graphs with polynomials in the regions (or rather, a linear combination of
these). Henceforth, the term \emph{graph} only refers to colored finite graphs with boundary (embedded in the
planar strip) which can be constructed out of univalent, trivalent, 4-valent, and 6-valent vertices as above.
Remember that these graphs do have edges which run to the boundary, which we call \emph{boundary lines}, and
may have edges which meet neither the boundary nor any vertex, and thus must necessarily form a circle. We say
a graph \emph{has a boundary} if it has at least one boundary line. A graph divides the planar strip into
regions, and there are two distinguished regions: the \emph{lefthand} and \emph{righthand} regions, which
contain $-\infty$ and $\infty$ respectively.

We call a \emph{boundary dot} any connected component of a graph which consists entirely of an edge starting at
the boundary and ending in a dot. We call a \emph{double dot} any connected component of a graph which consists
entirely of an edge with a dot on both ends. Cutting an edge in a diagram and replacing it with two dots we
call \emph{breaking} the edge (see, for instance, relation (\ref{dotSpaceDot})).

Given a set $S$ of graphs and a morphism $\phi$ in $\mc{DC}_1$, we say that $S$ \emph{underlies} $\phi$ if
$\phi$ can be written as a linear combination of morphisms, each of which is given by a graph $\Gamma \in S$
with polynomials in regions. We say that a graph $\Gamma$ \emph{reduces} to $S$ if $S$ underlies every morphism
that $\Gamma$ underlies. Clearly reduction is transitive, in that if $\Gamma$ reduces to $S$, and every graph
in $S$ reduces to $S^\prime$, then $\Gamma$ reduces to $S^\prime$. Our goal will be to find a nice set of
graphs to which all other graphs reduce. We will do this by finding \emph{reduction moves}, which are local
moves on graphs, sending a graph to a set of graphs to which it reduces.

\begin{example} The relation (\ref{twoLines}) implies that $\ig{.2}{twoLines.eps}$ reduces to
$\ig{.2}{Xdiagram.eps}$. In other words, $\ig{.2}{Xdiagram.eps}$ underlies both the terms on the right side of
(\ref{twoLines}). This can be applied as a local reduction move within any graph. \end{example}

Let $T$ be a subset of $\{1,\ldots,n\}$. The $T$-\emph{graph} of a graph will be the subgraph consisting of all
edges colored $i$ for $i\in T$. Some 6-valent vertices in the original graph may become trivalent vertices in
the $T$-graph. Similarly, some 4-valent vertices in the original graph may become 2-valent vertices in the
$T$-graph, which we ignore, connecting the incoming edges into a single edge. The $T$-graph is itself a graph
by our above definition. Most often we will just consider the $i$-graph for a single color (i.e. $T=\{i\}$).
Typically, our reduction moves will be designed to simplify the $i$-graph for a particular $i$, allowing us to
simplify the graph one color at a time.

\begin{remark} {\bf Addendum} The rest of this paper will have numerous calculations, but they will mostly be
calculations with the underlying graphs, not keeping track of polynomials, so they do not reflect how morphisms
actually behave in $\mc{DC}_1$. For lots of examples of computations in the graphical calculus, see \cite{EKr}.
\end{remark}

%
\subsection{One color reductions}
\label{subsec-onecolorred}
%

In this section, we assume all graphs consist of a single color $i$.

\begin{defn} \label{basicmoves} Consider the following ``moves", or transformations. They take a subdiagram
looking like Start, and replace that subdiagram with Finish. We call these the \emph{basic moves}.

$
\begin{array}{cc}
\begin{array}{ccc}
\mathrm{Move} & \mathrm{Start} & \mathrm{Finish}\\
\\
\mathrm{Associativity} & \igrotCW{.3}{Xdiagram.eps} & \ig{.3}{Xdiagram.eps}\\
\\
\mathrm{Dot Contraction} & \ig{.3}{dotSplit.eps} & \ig{.3}{cap.eps}\\
\\
\mathrm{Double Dot Removal} & \ig{.3}{doubleDot.eps} & \quad
\end{array} &
\begin{array}{ccc}
\mathrm{Move} & \mathrm{Start} & \mathrm{Finish}\\
\\
\mathrm{Needle} & \ig{.3}{needle2.eps} & \igv{.4}{dot.eps}\\
\\
\mathrm{Dot Extension}  & \ig{.3}{cap.eps} & \ig{.3}{dotSplit.eps}
\\
\mathrm{Connecting} & \ig{.3}{twoLines.eps} & \ig{.3}{Xdiagram.eps}\\
\end{array}
\end{array}
$
\end{defn}

Remember, these are moves on graphs, not graphs with polynomials. Note that the needle move, by adding a dot on
the bottom, yields a reduction from the circle to a double dot. The only moves which change the connectivity of
a graph are double dot removal, which deletes a connected component, and the connecting move, which has the
potential to link two components into one.

\begin{claim} All of these moves are reduction moves in $\mc{DC}_1$. \end{claim}

\begin{proof} The associativity move follows from (\ref{eqn-associativity}). That is, even if there are
polynomials in the regions of the graph, the relation (\ref{eqn-associativity}) can still be applied. These
polynomials, being in external regions, do not interfere with the application of relations. Similarly, dot
contraction/extension follow from (\ref{unit}), dot removal follows from (\ref{doubleDot}), and the
connecting move follows from (\ref{twoLines}).

The needle move remains. Suppose we have an arbitrary polynomial $f$ in the eye of the needle. We may use
(\ref{holepoly}), generalizing (\ref{needleWithEye}), to replace the diagram with a dot accompanied by
$\partial_i(f)$. \end{proof}

The following example of reduction should be familiar.

\begin{example} 
\begin{align*} \ig{.4}{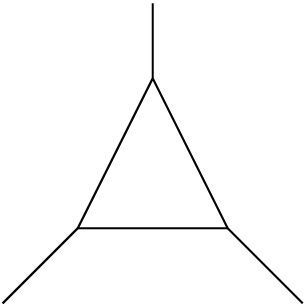}& \rightarrow \ig{.4}{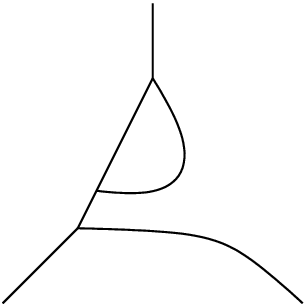} \rightarrow 
\ig{.4}{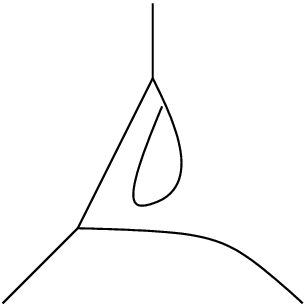}\\
& \rightarrow \ig{.4}{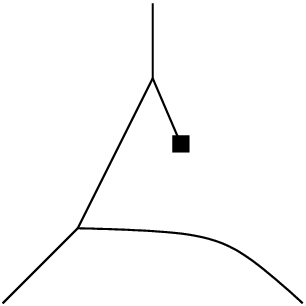} \rightarrow \ig{.4}{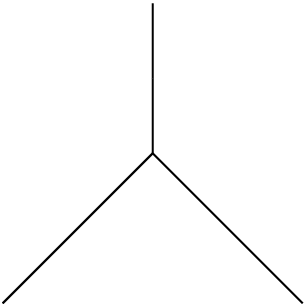} \end{align*}
In order, this is done with associativity, associativity, needle, and dot contraction moves. \label{reducecycleex}
\end{example}

\begin{example}  \begin{align*}\ig{.35}{polygonI3.eps} \rightarrow
\ig{.35}{polygonI2.eps} \end{align*} This is meant to indicate an arbitrary length cycle of this form, and
reduction is done with associativity, needle, and dot contraction moves. \label{reducecycleexgen} \end{example}

\begin{defn} A \emph{simple tree} $T$ with $m$ boundary lines is a connected one-color graph with boundary,
whose form depends on $m$: \begin{enumerate} \item If $m\ge2$ then $T$ is a trivalent tree with $m-2$ vertices
connecting all the boundary lines. Note that any two such trees are equivalent under the associativity move.
\item If $m=1$ then $T$ is a single boundary dot. \item If $m=0$ then $T$ is the empty graph. \end{enumerate}
\end{defn}

\begin{defn} A \emph{cycle} in a one-color graph is either a circle or a path from a vertex to itself which
does not repeat any edges. Any cycle splits the plane into two parts, the inside and outside of the cycle. A
cycle is \emph{minimal} if the inside of the cycle consists of a single region. \end{defn}

By counting vertices, it is clear that any connected purely trivalent graph with no cycles is a simple tree.
Any graph with a cycle has a minimal cycle.

We shall now give a precise inductive algorithm to reduce a graph to a disjoint union of simple trees, by
reducing minimal cycles.

\begin{prop} \label{reducecycle} Consider a minimal cycle in a one-color graph $\Gamma$. Using the
associativity, needle, dot removal, and dot contraction/extension moves, we may reduce a neighborhood of the
cycle (including the inside region) to a simple tree (See (\ref{figure})). \end{prop}

\begin{equation} \label{figure} \ig{.4}{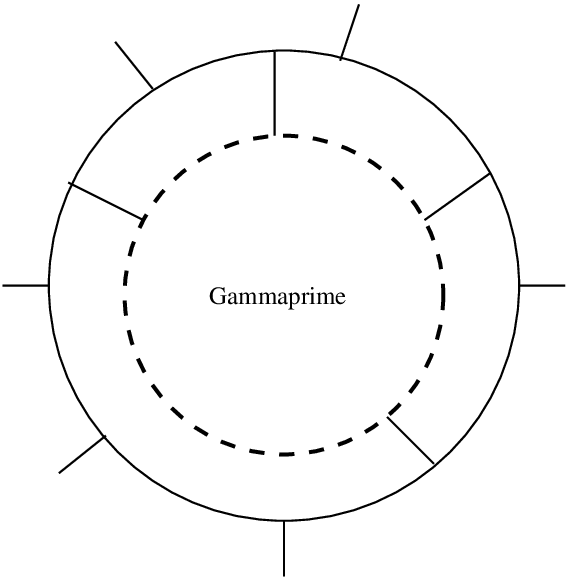} \quad \quad \to \quad \quad \ig{.4}{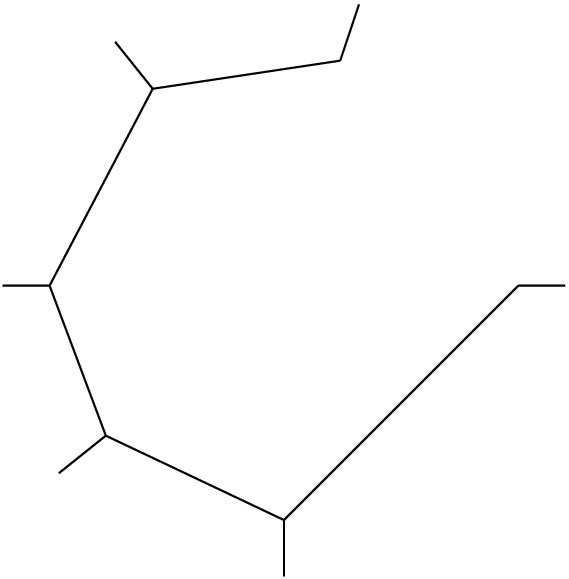}
\end{equation}

\begin{prop} \label{prop-onecolor} Using the associativity, needle, dot removal and dot contraction/extension
moves, one can reduce any one-color graph $\Gamma$ to a disjoint union of simple trees. During this process,
each component with no boundary lines will be replaced with the empty graph, and after this is done, no further
connected components are created or destroyed or merged. \end{prop}

We prove both propositions together, in several steps.

\begin{proof}[Proof of Proposition \ref{prop-onecolor} for a graph with no cycles] Suppose there are no cycles
in $\Gamma$. If there is a dot, the edge coming from that dot must connect to either the boundary, another dot,
or a trivalent vertex. Our simplification algorithm is: \begin{itemize} \item Step 1: Remove any dot connected
to a trivalent vertex, using dot contraction. Repeat Step 1 until no such dots remain. This does not alter the
connected components. \item Step 2: Replace any double dot with the empty set, using double dot removal.
Because Step 1 did not alter the connected components, this could only be applied to double dots which arose
from components which had no boundary lines. \end{itemize} Boundary dots are in their own connected component.
Any other connected component is purely trivalent and has no cycles, so it must be a simple tree. \end{proof}

\begin{proof}[Proof of Proposition \ref{reducecycle}, assuming Proposition \ref{prop-onecolor} for graphs with
no cycles] Let $\gamma$ denote a minimal cycle of $\Gamma$. Consider a neighborhood of $\gamma$ in $\Gamma$,
and let $\Gamma'$ be the subgraph consisting of the interior of the cycle, as in (\ref{figure}). Let us call the
boundary lines of $\Gamma'$ \emph{spokes}, since they run into $\gamma$ from the inside, like spokes hitting
the wheel of a bicycle. Now $\Gamma'$ may not have any cycles, or else $\gamma$ would not be a minimal cycle.
Moreover, the spokes of $\Gamma'$ must be in distinct connected components of $\Gamma'$, or else they would
create additional regions and $\gamma$ would not be a minimal cycle. Our simplification algorithm is:
\begin{itemize} \item Step 1: Apply Proposition \ref{prop-onecolor} to $\Gamma'$, replacing $\Gamma'$ with a
disjoint union of boundary dots, one for each spoke. \item Step 2: Use dot contraction to remove all the
spokes. We are now in the situation of Example \ref{reducecycleexgen}. \item Step 3: Apply associativity as in
Example \ref{reducecycleex}, reducing the length of $\gamma$ by one. Repeat until $\gamma$ is a needle or a
circle. \item Step 4: If $\gamma$ is a needle, apply the needle move to replace it with a dot. If associativity
moves were performed in Step 3, use dot contraction to contract this dot into one of the trivalent vertices, as
in Example \ref{reducecycleex}. \item Step 5: If $\gamma$ is a circle, apply dot extension to replace it with a
needle attached to a dot. Then apply needle reduction to obtain a double dot, and double dot removal to obtain
the empty graph. \end{itemize} It is a simple observation that the result of this procedure is a simple tree,
and that the only alteration of connected components which occurred was the removal of components which had no
boundary lines to begin with. \end{proof}

\begin{proof}[Proof of Proposition \ref{prop-onecolor} in the general case] Suppose we have an arbitrary graph
$\Gamma$. Our simplification algorithm is: \begin{itemize} \item Step 1: If $\Gamma$ has a cycle, apply
Proposition \ref{reducecycle} to replace its neighboorhood with a simple tree. Repeat this process until
$\Gamma$ has no cycles. \item Step 2: Apply the procedure for the case of no cycles above. \end{itemize} Note
that Step 1 will terminate, which can be shown by induction on the number of internal regions (regions which do
not meet the boundary of the planar strip). Each application of Proposition \ref{reducecycle} reduces the
number of internal regions by 1. \end{proof}

\begin{cor} Using the connecting move in addition, we can reduce the graph to a single simple tree. \end{cor}

\begin{proof} It is an easy observation that when one uses the connecting move on a simple tree with $m$
boundary lines and a simple tree with $m'$ boundary lines, one gets a simple tree with $m+m'$ boundary lines,
after possibly removing extraneous dots if either $m$ or $m'$ equals $1$. \end{proof}

\begin{remark} There are two useful sets of one-color graphs with $m$ boundary lines, to which all others
reduce. The first set just contains the simple tree with $m$ boundary lines, and the latter is the collection
of all disjoint unions of simple trees whose number of boundary lines add up to $m$. The former is useful
because there is a single graph, so we have fewer cases to deal with. The latter is useful because it doesn't
require the connecting move.

More importantly, these sets behave differently when we introduce polynomials into the equation. Let us assume
that all $m$ boundary lines are on the top boundary, so that we are looking at a morphism in
$\Hom(\emptyset,\ii)$ where $\ii$ is $iiiiiiiii$ ($m$ times). The following statements will not be used in this
paper, and can be more easily proven after the calculation of Hom spaces.

\begin{claim} Consider diagrams which are a simple tree, with an arbitrary monomial in the lefthand region, and
either $1$ or $\xi$ in each other region. This is a basis for $\Hom(\emptyset,\ii)$ over $\Bbbk$. \end{claim}

\begin{claim} Consider diagrams which are disjoint unions of simple trees, with an arbitrary monomial in the
lefthand region, and no other polynomials. These are a spanning set for $\Hom(\emptyset,\ii)$ over $\Bbbk$.
\end{claim}

The second claim is easy to see, given Proposition \ref{prop-onecolor}. Given a disjoint union of simple trees,
with arbitrary polynomials, we may use relation (\ref{dotSpaceDot}) and the polynomial slides to force all the
polynomials to the left, at the cost of potentially breaking some lines. This breakage is not a problem, since
one can reduce again to a simple tree without adding more polynomials, using (\ref{eqn-associativity}) and
(\ref{unit}). We do not get a basis this way: consider the three different ways to break a line in a
trivalent vertex diagram; there is a linear dependence relation between these diagrams and the trivalent vertex
with a polynomial in the lefthand region.

Relations (\ref{twoLines}) and (\ref{eq-slide1}) essentially allow us to get from the second spanning set to
the first, showing that the first is at least a spanning set. That it is a basis is immediate from counting the
graded dimension of the Hom space, one we prove that the dimension of Hom spaces in $\mc{DC}_1$ conforms with a
certain semilinear form. \end{remark}

The connecting move is less important than the others in the proofs, and was introduced primarily to make these
remarks. It can generally be ignored below.

%
\subsection{Broken one color reductions}
\label{subsec-broken}
%

The reductions of the previous section do apply, as stated, to any one color graph. However, we would like to
apply these moves to the $i$-colored graph of a multicolor graph, where the moves above do not extend trivially
to reduction moves in $\mc{DC}_1$. In this section, we quickly generalize the results of the previous section
to a weaker set of moves.

\begin{defn} \label{weakbasicmoves} Consider the following reductions for one-color graphs, which take the
graph on the left and replace it with the set of graphs on the right.

$
\begin{array}{ccc}
\mathrm{Move} & \mathrm{Start} & \mathrm{Finish}\\
\\
\mathrm{Weak Associativity} & \igrotCW{.3}{Xdiagram.eps} \quad \quad & \begin{array}{ccccc} \ig{.3}{Xdiagram.eps} & \ig{.2}{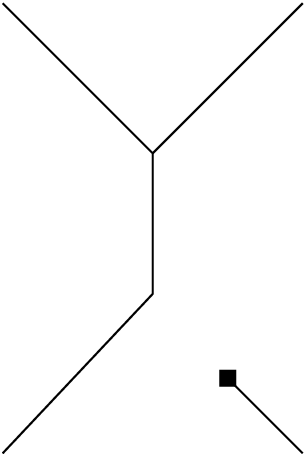} & \igv{.2}{twoLinesAlt2.eps} & \igh{.2}{twoLinesAlt2.eps} & \ighv{.2}{twoLinesAlt2.eps} \end{array}\\
\\
\mathrm{Weak Dot Contraction} & \ig{.3}{dotSplit.eps} & \begin{array}{ccc}\ig{.3}{cap.eps} & \ & \igv{.2}{dot.eps} \quad \igv{.2}{dot.eps} \end{array}
\end{array}
$

We call the moves above (together with the basic moves that have no weak analog) the \emph{broken} or
\emph{weak basic moves}. \end{defn}

These moves behave like the basic moves of the same name, except that they may also replace the original
diagram with a broken version of itself, that is, a version with some edges broken. To distinguish the original
basic moves, we may call them the \emph{strict basic moves}.

What is important is that we have analogs of Propositions \ref{reducecycle} and \ref{prop-onecolor}, despite
only being able to use weak moves.

\begin{prop} \label{weakreducecycle} Consider a minimal cycle in a one-color graph $\Gamma$. Using the weak
associativity, needle, dot removal, weak dot contraction, and dot extension moves, we may reduce a
neighborhood of the cycle (including the inside region) to a disjoint union of simple trees. \end{prop}

\begin{prop} \label{prop-weakonecolor} Using the weak associativity, needle, dot removal, weak dot
contraction, and dot extension moves, one can reduce any one-color graph $\Gamma$ to a disjoint union of simple
trees. \end{prop}

Of course, two distinct simple trees are no longer equivalent under the weak associativity move, but this is
really irrelevant for us.

\begin{proof} Breaking a line will never create a cycle or increase the number of trivalent vertices. Because
of this, the proofs of the previous section go through almost verbatim (ignoring any statements about connected
component, and occassionally replacing ``a simple tree" with ``a disjoint union of simple trees"). The only
significant alterations that need to be made come in the proof of Proposition \ref{reducecycle}. In Step 4 or
Step 5 one may need to remove an additional double dot. In Step 2 or Step 3, weak dot contraction and weak
associativity have multiple outcomes, but each outcome that does not agree with strict dot contraction or
strict associativity will have broken the cycle already, allowing us to complete the proof using the no-cycle
algorithm of Proposition \ref{prop-onecolor}.

Alternatively, one could also prove these statements by induction on the number of trivalent vertices. Each
weak move is equivalent to a strong move modulo diagrams with fewer trivalent vertices. The only part of the
proof that ever created additional trivalent vertices was the single use of dot extension in Step 5 of the
proof of Proposition \ref{reducecycle}. It is easy to see how Step 5 does not actually cause a problem,
however, since after dot extension is applied, the needle move and double dot removal will do the trick in
the same way regardless. \end{proof}

{\bf Addendum} The overall proof using weak one-color moves is slightly different than the treatment in
previous versions of this paper, but it is cleaner and more straight-forward.

%
\subsection{$i$-colored moves}
\label{subsec-coloredmoves}
%

Now we list the moves which allow us to simplify multicolor graphs.

\begin{defn} \label{icoloredmoves} Consider a graph $\Gamma$ whose $i$-graph looks like one of the pictures in
the start column of Definition \ref{basicmoves}. Let $S$ be the set of all graphs whose $i$-graph looks like
the corresponding picture in the finish column. Let $W$ be the set of all graphs whose $i$-graph looks like any
of the corresponding pictures in the finish column of definition \ref{weakbasicmoves}. The \emph{strict}
$i$-\emph{colored move} replaces $\Gamma$ with the set $S$. The \emph{weak} $i$-\emph{colored move} replaces
$\Gamma$ with the set $W$. \end{defn}

For instance, strict $i$-colored associativity will replace any graph $\Gamma$ whose $i$-graph is
$\igrotCW{.2}{Xdiagram.eps}$ with the set $S$ of graphs whose $i$-graph is $\ig{.2}{Xdiagram.eps}$. This set
$S$ is enormous, for other colors can interfere, and the $j$-graph for some other $j$ can be arbitrarily
complicated. The $i$-colored vertices, seemingly trivalent, could come from 6-valent vertices in $\Gamma$. In
general, an $i$-colored move will behave nicely on the $i$-graph, but may significantly complicate the full
graph.

\begin{prop} \label{prop-imoves} The weak $i$-colored basic moves are reduction moves in $\mc{DC}_1$, so long
as they are applied to graphs which do not contain either the color $i-1$ or $i+1$. \end{prop}

A color $i$ will be called \emph{extremal} for a graph $\Gamma$ if it appears in $\Gamma$ but either $i-1$ or
$i+1$ does not appear. Clearly, any non-empty graph will have an extremal color, such as the minimal color
present.

The proof of this proposition is found in Section~\ref{subsec-graphproof}, as well as more precise details on
what can be done. The power of the proposition can be seen immediately:

\begin{cor} \label{cor-empty} Any graph $\Gamma$ without boundary lines can be reduced to the empty graph.
\end{cor}

\begin{proof} We induct on the set of colors present in the graph $\Gamma$. If no colors are present, then
$\Gamma$ is the empty graph and we are done. Else, choose an extremal color $i$. By Proposition
\ref{prop-imoves} we may apply the weak $i$-colored basic moves and use Proposition \ref{prop-weakonecolor} to
replace every connected component of the $i$-graph with a disjoint union of simple trees with no boundary
lines. Since a simple tree with no boundary lines is the empty set, $\Gamma$ reduces to a set of graphs which
do not include the color $i$. By induction, $\Gamma$ now reduces to the empty graph. \end{proof}

One can apply a similar procedure to a graph $\Gamma$ with boundary lines. Choose an extremal color $i$, and
reduce the $i$-graph to a disjoint union of simple trees. Then, within each region delimited by the $i$-graph,
the colors $i+1$ and $i-1$ are now extremal, and we can reduce those. One can repeat this procedure, however,
it will not produce a very simple graph in all cases. If the color $i$ has at least 3 boundary lines, the
$i$-graph may have trivalent vertices, and the graph $\Gamma$ itself may have 6-valent vertices in their place.
These 6-valent vertices will produce more $i+1$ or $i-1$ colored boundary lines inside the regions delimited by
the $i$-graph.  Nonetheless, we have the following simple case.

\begin{cor} \label{cor-oneofeach} Any graph whose boundary has at most one line of each color can be reduced to
a disjoint union of boundary dots. \end{cor}

\begin{proof} We know we can reduce the $i$-graph, for $i$ an extremal color, to a disjoint union of simple
trees. A simple tree with at most one boundary line is either the empty set or a boundary dot. Therefore the
$i$-graph is now either the empty set or a boundary dot, depending on whether or not $i$ appears in the
boundary. The dot need not be a boundary dot in the entire graph $\Gamma$, but it can encounter only 4-valent
vertices en route to the boundary. Since a dot can be slid under a 4-valent vertex by relation
(\ref{eqn-ijijDot}), we may turn the dot into a boundary dot (its own connected component). The remaining
connected components form a subgraph (also viewable in the planar strip) without the color $i$. Induction now
concludes the proof.\end{proof}

%
\subsection{$\mc{F}_1$ is fully faithful}
\label{subsec-functor}
%

In this section, modulo the proofs of previous sections which were delayed until Section \ref{sec-proofs}, we
prove our main theorem.

\begin{thm} \label{mainthm} The functor $\mc{F}_1$ from $\mc{DC}_1$ to $\mc{SC}_1$ is 
an equivalence of $\Bbbk$-linear monoidal categories with Hom spaces enriched in $\RmolfZR$. \end{thm}

We know $\mc{F}_1$ is a functor by Proposition \ref{f1isafunctor}, and inspection of the objects in both
categories shows immediately that it is essentially surjective. To show $\mc{F}_1$ is full, we use Proposition
\ref{prop-full}, which motivates the next few statements.

\begin{cor} \label{cor-biadjoint} For any index $i$, the object $i$ in $\mc{DC}_1$ is self-biadjoint. This
means that for any sequences $\jj$ and $\kk$, there are natural isomorphisms $\Hom_{\mc{DC}_1}(\kk,\jj i) \to
\Hom_{\mc{DC}_1}(\kk i,\jj)$ and $\Hom_{\mc{DC}_1}(\kk,i \jj) \to \Hom_{\mc{DC}_1}(i \kk,\jj)$.\end{cor}

\begin{proof} The first isomorphism and its inverse are shown below. That these maps compose to be the identity
is exactly the relation (\ref{biadjoint}).

\begin{equation*}
\ig{.4}{Biadj1.eps} \ \ \ \mapsto \ \ \ \ig{.4}{Biadj2.eps} \quad \quad \quad \quad 
\ig{.4}{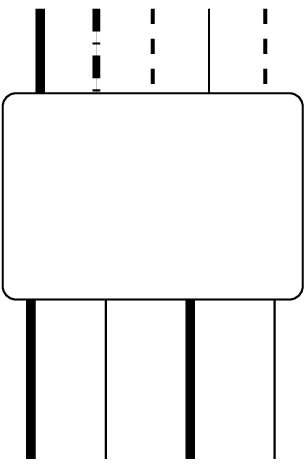}\ \ \  \mapsto \ \ \ \ig{.4}{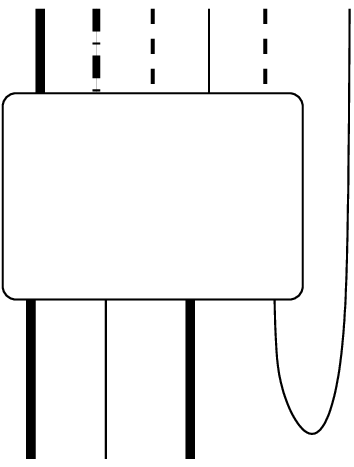}
\end{equation*} 

The second isomorphism and its inverse are the left-right mirror of the maps above. \end{proof}

Note that when $\kk=\emptyset$ in the corollary above, these isomorphisms combine to yield an isomorphism
$\Hom_{\mc{DC}_1}(\emptyset,\jj i) \to \Hom_{\mc{DC}_1}(\emptyset,i \jj)$, drawn as below.

\begin{equation*} \ig{.4}{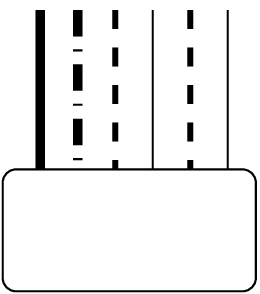} \mapsto \ig{.4}{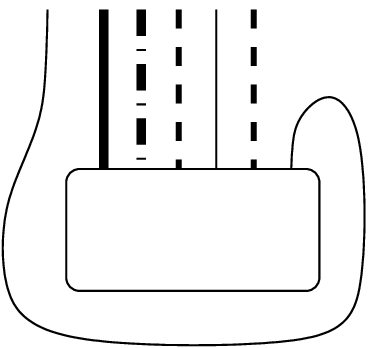} \end{equation*}

At this point, one could construct a semilinear product on the free algebra generated by $b_i$, $i=1, \ldots ,
n$, via $(b_{\ii},b_{\jj})=\grdrk \Hom_{\mc{DC}_1}(\ii,\jj)$, and $b_i$ would be self-adjoint. If we show that
the Hecke algebra relations are in the kernel of this semilinear product then it will descend to the Hecke
algebra. We have several methods by which we could do this.

\begin{itemize} \item Look in $\mc{DC}_2$, where we have direct sums and grading shifts, and prove the
isomorphisms (\ref{eqn-iibimod})-(\ref{eqn-ipibimod}). \item Look in the Karoubi envelope $\mc{DC}$, find
idempotents corresponding to the auxilliary modules in (\ref{eqn-auxbimod}) and friends, and show those
isomorphisms. \item Work entirely within $\mc{DC}_1$ and show the isomorphisms (\ref{eqn-ipibimod}) only after
applying the Hom functor. For instance, showing that $\Hom(ii,\jj) \cong \Hom(i,\jj)\{1\} \oplus
\Hom(i,\jj)\{-1\}$ will be sufficient.\end{itemize}

All these tactics are primarily the same. We illustrate the third method, although we do explore the auxilliary
modules of the second method.

The relation (\ref{twoLines}) precisely descends to $B_i \TenR B_i \cong B_i\{1\} \oplus B_i\{-1\}$. We
decompose the identity of $ii$ into the sum of two idempotents, and obtain orthogonal projections from $ii$ to
$i$ of degrees 1 and -1 respectively.

\begin{equation} \ig{.4}{twoLines.eps} \ \ = \ \ \ig{.4}{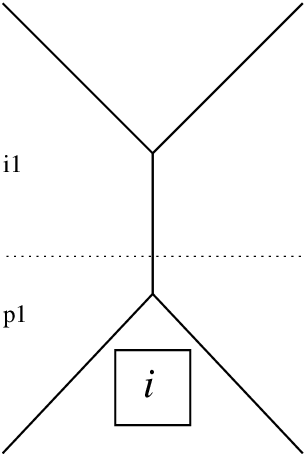} \ -\
\ig{.4}{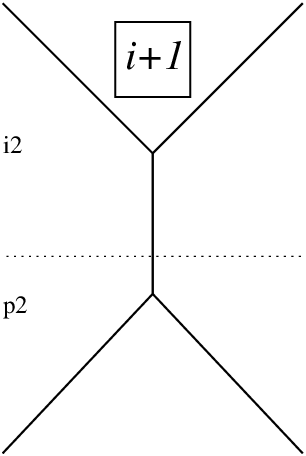} \end{equation}

Stated very explicitly, we have two projection maps $p_1,p_2 \colon ii \to i$, and two inclusion maps
$\alpha_1,\alpha_2 \colon i \to ii$, as indicated in the diagram above, which is just relation (\ref{twoLines})
divided in half. Include the minus sign on the right picture into the map $\alpha_2$. Then one can quickly
check $p_1\alpha_1=1_i$, $p_2\alpha_2=1_i$, $p_1\alpha_2=0$, $p_2\alpha_1=0$, and $1_{ii}= \alpha_1p_1 +
\alpha_2p_2$. So these must be projections to and inclusions of summands, and they are maps of the correct
degree.

Now we look at $i$ and $j$ which are distant. We have the isomorphism $B_i \TenR B_j \cong R \oTen{R^{i,j}}
R\{-2\} \cong B_j \TenR B_i$. It is clear that we may construct for any $\kk$ isomorphisms $\Hom(ij,\kk) \to
\Hom(ji,\kk)$ or vice versa merely by precomposing with the appropriate 4-valent vertex, and (\ref{eqn-R2move})
shows that these maps are inverses.

Because we can, let us discuss how we might extend our diagrammatic calculus to include additional modules from
$\RmolfR$, like $R \oTen{R^{i,j}} R\{-2\}$. Let us (temporarily) allow a new color of line in our pictures,
call it $w$, and extend the functor $\mc{F}_1$ so that it sends $w$ to $R \oTen{R^{i,j}} R\{-2\}$. We add new generators to our diagrammatics, and specify the image under the extended functor, as below.

\begin{eqnarray*}
  \mathrm{Symbol} & \mathrm{Degree} & \mathrm{Definition} \\
  \\
  \igv{.5}{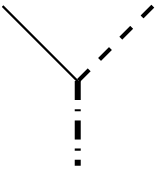} & 0 & \auptob{1 \teni 1 \tenj 1}{1 \tenfar 1} \\
  \\
  \ig{.5}{ijdown.eps} & 0 & \auptob{1 \tenfar 1}{1 \teni 1 \tenj 1}
\end{eqnarray*}

\begin{equation*} \ig{.5}{4valent.eps} \define \ig{.4}{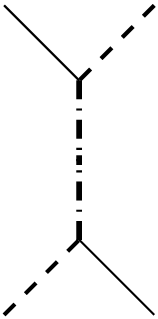} \end{equation*}

The definition of these bimodule maps, and the proof that they are in fact bimodule maps, is exactly akin to
the discussion in Section \ref{subsec-fulldefn}. It is clear that composing these morphisms to get an
endomorphism of $ij$ will yield the identity map, and composing them to get an endomorphism of $w$ will also
yield the identity map (these would be relations in the extended diagrammatic calculus). Thus we get
isomorphisms $ij \cong w \cong ji$ in $\mc{DC}$.

We now combine these techniques to deal with adjacent colors. We have isomorphisms $B_i \TenR B_{i+1}
\TenR B_i \cong B_i \oplus (R \oTen{R^{i,i+1}} R\{-3\})$ and $B_{i+1} \TenR B_i \TenR B_{i+1} \cong
B_{i+1} \oplus (R \oTen{R^{i,i+1}} R\{-3\})$. If we allow the new color of line, again called $w$,
to represent the bimodule $R \oTen{R^{i,i+1}} R\{-3\}$, then we may define the following maps:

\begin{eqnarray*}
  \mathrm{Symbol} & \mathrm{Degree} & \mathrm{Definition} \\
  \\
  \igv{.4}{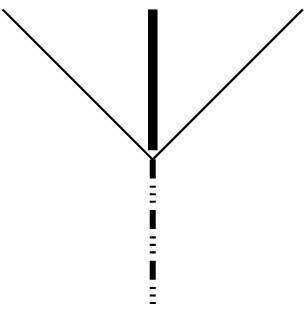} & 0 & \begin{array}{cc} \auptob{1 \teni 1 \tenp 1 \teni 1}{1 \tenadj 1} & 
\auptob{1 \teni \xi \tenp 1 \teni 1}{(\xi + \xp) \tenadj 1 - 1 \tenadj \xqq} \end{array} \\
  \\
  \ig{.4}{ipidown.eps} & 0 & \auptob{1 \tenadj 1}{1 \teni 1 \tenp 1 \teni 1} \\
  \\
  \igv{.4}{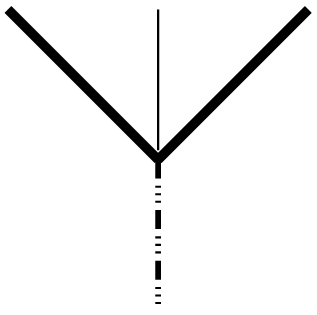} & 0 & \begin{array}{cc} \auptob{1 \tenp 1 \teni 1 \tenp 1}{1 \tenadj 1} & 
\auptob{1 \tenp \xqq \teni 1 \tenp 1}{1 \tenadj (\xp + \xqq) - \xi \tenadj 1} \end{array} \\
  \\
  \ig{.4}{pipdown.eps} & 0 & \auptob{1 \tenadj 1}{1 \tenp 1 \teni 1 \tenp 1}
\end{eqnarray*}

\begin{equation*} \ig{.4}{6valent.eps} \define \ig{.3}{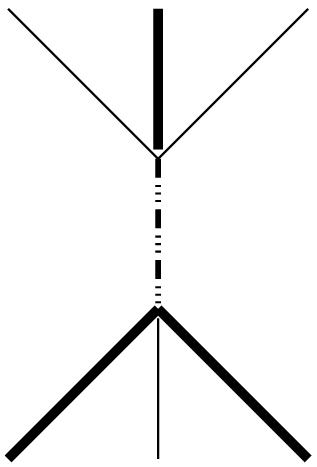} \end{equation*}
\begin{equation*} \igv{.4}{6valent.eps} \define \igv{.3}{6valentDef.eps} \end{equation*}

Again, checking that we have well-defined bimodule maps is akin to Section \ref{subsec-fulldefn}. Composing the
two maps that go through $i(i+1)i$ and $w$ to get an endomorphism of $w$ will yield the identity map of $w$,
and the same is true respectively of $(i+1)i(i+1)$.

Then the equation (\ref{eqn-threeLines}), which is a decomposition of the identity of $i(i+1)i$, actually
follows from this relation in $\mc{DC}$:

\begin{equation} \ig{.3}{threeLines1.eps}=\ig{.3}{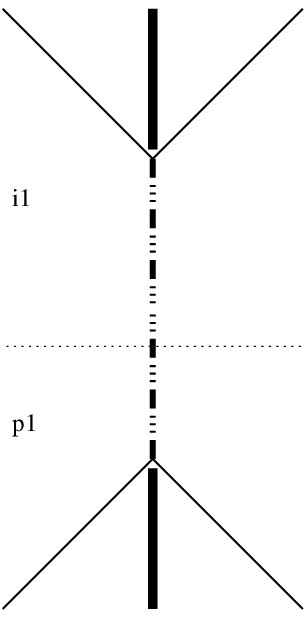} -
  \ig{.3}{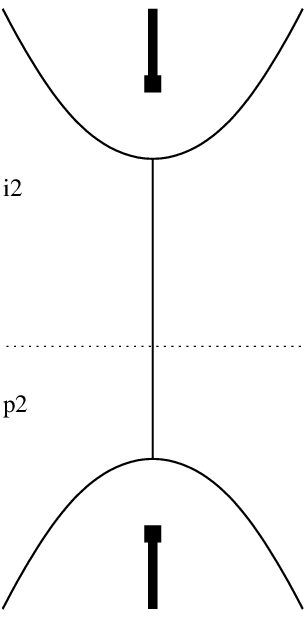} \label{eqn-threeLinesBad}\end{equation}

There is a more explicit statement to be derived from this relation, completely analogous to the two
line case, with projection and inclusion maps $p_1,\alpha_1,p_2,\alpha_2$ (include the minus sign on
the right picture in $\alpha_2$). Similarly we get a decomposition of $(i+1)i(i+1)$ via the same
relation with the colors switched.  The auxilliary module used here is in fact the indecomposable
Soergel bimodule $B_w$ where $w=s_is_{i+1}s_i=s_{i+1}s_is_{i+1}$.

To state the result without extending the calculus to the Karoubi envelope, one may merely observe that for any
$\kk$ there is a map $\Hom(i(i+1)i,\kk) \to \Hom((i+1)i(i+1),\kk)$ and vice versa given by precomposing with
the appropriate 6-valent vertex; there is another map $\Hom(i(i+1)i,\kk) \to \Hom(i,\kk)$ given by precomposing
with $\alpha_2$, and a map backwards given by precomposing with $p_2$. Then (\ref{eqn-threeLines}) exactly yields
that $\Hom(i(i+1)i,\kk) \oplus \Hom((i+1),\kk) \cong \Hom((i+1)i(i+1),\kk) \oplus \Hom(i,\kk)$.

Thus we have shown the following claim:

\begin{claim} The isomorphisms (\ref{eqn-iibimod}) through (\ref{eqn-ipibimod}) hold in $\mc{DC}_1$ after
applying any Hom functor. \end{claim}

We have now satisfied the requirements of Proposition \ref{prop-full} and can deduce that

\begin{claim} The functor $\mc{F}_1$ is full. \end{claim}

Finally, our graphical reductions give us a classification of certain Hom spaces.

\begin{cor} \label{cor-grrnk} The space $\Hom_{\mc{DC}_1}(\emptyset,\ii)$, where $\ii$ is a length $d$
increasing sequence, is a free left (or right) $R$-module of rank 1, generated by a homogeneous morphism of
degree $d$. \end{cor}

\begin{proof} Let $\phi$ be the morphism consisting entirely of boundary dots. This corresponds, under the
functor, to the generator of degree $d$ discussed in Claim \ref{increasinghom}. By Corollary
\ref{cor-oneofeach}, $\phi$ viewed as a graph underlies every morphism in $\mc{DC}_1$. This graph has a single
region, so any morphism that it underlies will be generated by $\phi$ under the action of $R$ on the right or
left, which puts a polynomial into that single region. This gives a surjective map from $R$ to the Hom space.
After applying $\mc{F}_1$, we know that the morphisms must surject onto $\Hom_{\mc{SC}_1}(R,B_{\ii})$, which is
a free $R$-module (see Claim \ref{increasinghom}). Therefore, the Hom space in $\mc{DC}_1$ must also be free.
\end{proof}

\begin{cor} The semilinear form on $\mc{H}$ induced by Hom spaces in $\mc{DC}_1$ agrees with the form defined
in Section \ref{subsec-hecke}, so it agrees with the form induced by $\mc{SC}_1$. Therefore $\mc{F}_1$ is
faithful. \end{cor}

\begin{proof} This is now immediate, from Remark \ref{unicity}. \end{proof}

In conclusion, $\mc{F}_1$ is an equivalence of categories, and Theorem~\ref{mainthm} is proven.

We also get for free the following corollary, which is difficult to prove purely diagrammatically.

\begin{cor} Hom spaces in $\mc{DC}_1$ are free as left or right $R$-modules. \end{cor}

\begin{remark} \label{whatisprovendiagrammatically} It is worth reiterating what is proven diagrammatically,
and what is proven using the functor $\mc{F}_1$ to Soergel's category. Diagrammatically, we can prove
Proposition \ref{prop-imoves} and Corollary \ref{cor-oneofeach}, which implies that $R$ surjects onto
$\Hom_{\mc{DC}_1}(\emptyset,\ii)$ for $\ii$ increasing. However, without the full functor $\mc{F}_1$ and the
non-diagrammatic knowledge of Hom spaces in Soergel's category, we do not know how to prove that the Hom space
is free. We do not have a fully diagrammatic proof that the Hom spaces in $\mc{DC}_1$ are what they are, we
only have a diagrammatic proof of an upper bound on the size of the Hom spaces. \end{remark}

%
\subsection{The $e_1$ quotient}
\label{subsec-quotient}
%

In the remainder of this chapter, we provide some sketched statements about generalizations and variations of
$\mc{DC}_1$. We do not use these results elsewhere in the paper, and while the proofs are only sketched, they
are fairly obvious and can be fleshed out without too much work.

As described in Remark \ref{dumbchange}, one usually constructs Soergel bimodules with respect to the
$n$-dimensional fundamental or geometric representation, instead of the $n+1$-dimensional standard
representation. This amounts to working over the ring $R_1 = \Bbbk \left[ x_1, x_2, \ldots, x_{n+1} \right] /
(x_1+x_2 + \ldots + x_{n+1}) $, which is a quotient of $R$ by the first elementary symmetric function $e_1$.
Since $e_1$ is symmetric, it is in the center of the category $\mc{DC}_1$ (it slides freely under all lines,
tensor-commuting with all morphisms), and one may easily take the quotient category, setting $e_1=0$, without
changing any of the diagrammatics. There is a functor from this quotient category to the appropriate category
of Soergel bimodules in $R_1\mathrm{-molf-}R_1$, which is an equivalence of categories, so this diagrammatical
quotient category also categorifies the Hecke algebra.

One advantage to passing to the quotient $e_1=0$ is that, after inverting a suitable integer, we may remove the
boxes from our list of generators. According to relation (\ref{doubleDot}), the double dot colored $i$ is equal
to $\xi - \xp$. Thus linear combinations of double dots of colors $i=1, \ldots , n$ will give us the
$\Bbbk$-span of $x_1 - x_2, x_2 - x_3, \ldots$ inside the space of linear polynomials (these are the simple
roots). This span will not include $\xi$ in $R$, which is why we require at least one box as an additional
generator. However, it is easy to check that, if $n+1$ is invertible in $\Bbbk$, then $\xi$ is in the
$\Bbbk$-span after passing to the quotient $R_1$. As an example, when $n=1$, the double dot is equal to
$x_1-x_2$, and passing to $x_1+x_2=0$, we get $x_1=\frac{1}{2}(x_1-x_2)$. Thus, if $n+1$ is invertible in
$\Bbbk$, one can eliminate boxes from the quotient calculus altogether, replacing them with linear combinations
of double dots.

{\bf Addendum} All the relations which involve boxes can be rewritten so that they only involve double dots.
This is not done here, but can be found in any of the papers \cite{E,Vaz,EKr}.

Replacing boxes with double dots is much more natural, since it emphasizes that the boxes themselves should
have a color, and that the set of polynomials depends in a natural way upon the coxeter group $S_{n+1}$.
Viewing them as boxes may help make the proofs in this paper more intuitive, however.

%
\subsection{Color elimination}
\label{subsec-colorelimination}
%

We have already shown that $\mc{DC}_1 \cong \mc{SC}_1$, without needing to investigate in depth any morphisms
except those from $\emptyset$ to $\ii$ an increasing sequence. Because this pins down the size of all Hom
spaces, we can deduce some additional facts about general morphisms. The following result is not used elsewhere
in this paper, but may come in handy when constructing the analogous category for arbitrary Coxeter systems.

\begin{prop} \label{prop-colorelimination} Let $\Gamma$ be a graph where the color $j$ does not appear in the
boundary. Then $\Gamma$ can be reduced to graphs not containing the color $j$ at all.\end{prop}

We can already show this when $j$ is extremal, since by Proposition \ref{prop-imoves} we may apply the weak
$j$-colored basic moves and reduce the color $j$ to the empty graph. Unfortunately, as in Remark
\ref{whatisprovendiagrammatically}, we do not have a diagrammatic result of this proposition in general, but
use the known size of Hom spaces to prove it.

For $X \subset \{1, \ldots ,n\}$, (i.e. for a Coxeter subgraph of $A_n$) we let $\mc{DC}_1[X]$ be the category
defined analogously but where edges can only be labelled by colors in $X$. If colors which don't appear on the
boundary are not needed in the graph after reduction, then the natural inclusion of $\mc{DC}_1[X]$ into
$\mc{DC}_1$ is fully faithful, which one would expect. The full faithfulness of this inclusion is equivalent to
the above proposition.

However, as discussed in the previous section, the boxes which are allowed to appear should also depend on $X$.
For the rest of this section, we assume we are working in the $e_1$ quotient. Let $\mc{DC}_1(X)$ be the
category defined analogously, where edges are labelled in $X$, and where the only polynomials appearing are in
the subring $R(X)$ generated by double dots colored in $X$. This is really the category we should look at.
Given a morphism in $\mc{DC}_1[X]$, we can use polynomial forcing rules to guarantee that the only polynomials
not in $R(X)$ appear in the lefthand region. We do not prove it here, $\mc{DC}_1[X]$ is simply $\mc{DC}_1(X)
\ot_{R(X)} R$; that is, the objects are unchanged, and morphism spaces undergo base change.

\begin{prop} Let $X$ be a Coxeter subgraph of $A_n$, and let $\mc{W}(X)$, $\mc{H}(X)$, $\mc{SC}_1(X)$ designate
the corresponding constructions for this Coxeter graph. Then there is a functor $\mc{F}_1(X)$ from
$\mc{DC}_1(X)$ to $\mc{SC}_1(X)$ which is an equivalence of categories. \end{prop}

\begin{proof} We will only sketch this result. The set $X$ will be a disjoint union of various subgraphs $X_l$,
each isomorphic to $A_{m_l}$ for some $m_l \le n$. For each $X_l$ we know that $\mc{DC}_1(X_l) \cong
\mc{SC}_1(X_l)$, and we have all the results aforementioned. Moreover, for any $l \ne l'$ and objects $M \in
\mc{DC}_1(X_l)$ and $N \in \mc{DC}_1(X_{l'})$, he have natural isomorphisms $M \TenR N \cong N \TenR M$
constructed with 4-valent crossings. Thanks to the distant sliding rules, these natural isomorphisms commute in
the proper way with all morphisms in $\mc{DC}_1(X_l)$ or $\mc{DC}_1(X_{l'})$. The category $\mc{DC}_1(X)$ can
be constructed via some universal ``symmetric monoidal product" construction, by taking the product over the
categories $\mc{DC}_1(X_l)$. The same holds true for $\mc{SC}_1(X)$, which will yield the result. \end{proof}

Given the sketchy result above, we can calculate the rank of Hom spaces in $\mc{DC}_1(X)$ by using the standard
trace map on $\mc{H}(X)$. This trace map commutes with the standard trace on $\mc{H}$ under the inclusion
$\mc{H}(X) \subset \mc{H}$. Thus we see that, for objects $\ii$ and $\jj$ in $\mc{DC}_1(X)$, the space
$\Hom_{\mc{DC}_1}(\ii,\jj)$ is a free $R$ module of some rank $r$, and the space $\Hom_{\mc{DC}_1(X)}(\ii,\jj)$
is a free $R(X)$ module of the same rank $r$. Therefore $\Hom_{\mc{DC}_1[X]}(\ii,\jj)$ is a free $R$ module of
rank $r$ as well, and the inclusion of $\mc{DC}_1[X]$ in $\mc{DC}_1$ is fully faithful.

We quickly sketch an alternative proof of Proposition \ref{prop-colorelimination}, which assumes that we know
that Hom spaces in $\mc{DC}_1$ conform to the standard trace. We can think of any graph with boundary as a
morphism in $\mc{DC}_1$ from $\emptyset$ to $\jj$. We have already ``inductively calculated'' the space of such
morphisms, along the lines of Remark \ref{unicity} and Proposition \ref{prop-full}. Namely, if $\jj$ has no
repeated colors, then every morphism is a polynomial with boundary dots. If $\jj$ has repeated colors, we use
idempotent decompositions to consider the Hom space as the sum of the Hom spaces of its summands, or apply
biadjunction to cycle the sequence $\jj$, until we have expressed the Hom space in terms of
$\Hom(\emptyset,\kk)$ for various $\kk$ non-repeating. Since neither biadjunctions nor idempotent
decompositions add any new colors to the graph, we have just constructed a generating set of morphisms in a way
which does not involve any colors which were not in $\jj$ to begin with.

\section{Proofs}
\label{sec-proofs}

It remains to prove Proposition \ref{f1isafunctor}, which we do in the first section, and Proposition
\ref{prop-imoves}, which we do in the second.

%
\subsection{$\mc{F}_1$ is a functor}
\label{subsec-f1}
%

We need now to check that each of our relations holds true for Soergel bimodules, after application of
$\mc{F}_1$. The relations are listed in full in Section \ref{subsec-fulldefn}. These checks are entirely
tedious and straightforward, and require little imagination. With a little imagination, however, there are some
tricks which make checking most relations trivial.

Again, we always use relation \ref{useoften} to identify elements of $B_{\ii}$ with $d(\ii)+1$ tensors of
polynomials.

For brevity, we will call any element which is a tensor product $1 \otimes 1 \otimes 1 \otimes \ldots$ a
\emph{1-tensor}. EndDot, Split, and the 4-valent and 6-valent vertices all send 1-tensors to 1-tensors. This
simple fact is already enough to prove relations (\ref{assoc}), (\ref{eqn-R2move}), (\ref{eqn-ijijDot}),
(\ref{eqn-pullFarThruTrivalent}), and (\ref{eqn-R3move}), since the bottom bimodule in each of these pictures
is generated by the 1-tensor; we call this the \emph{1-tensor trick}. Caps and Merges kill a 1-tensor, while Cups and StartDots send it to a sum of two
linear terms.

There are several choices to make when checking the various relations. Once the twisting relations are shown,
one is free to prove any twist of the other relations. Also, one may choose which set of generators of the
source bimodule to check equality on. Finally, whenever a Cup or a StartDot appears, there are two equivalent
ways to write the result, since $\xi \teni 1 - 1 \teni \xp = 1 \teni \xi - \xp \teni 1$ in $B_i$, and
one might be easier than the other to evaluate.

The first trick is to choose the set of generators which makes the calculation the easiest. Let us remind
ourselves of the arguments of Remark \ref{spanningintro}. An arbitrary module $B_{\ii}$ of length $d$ will have
$2^{d-m}$ generators as an $R$-bimodule, where $m$ is the number of different colors that appear. Let us use
the term $i$-\emph{pair} to denote two instances of the index $i$ in $\ii$, separated only by colors $\ne i$.
Let $X$ denote the set of all $i$-pairs for all $i$; the size of $X$ is $d-m$. As we force polynomials to the
right or left, a variable $\xi$ (or alternatively $\xp$) might get stuck between the two $i$-colored lines of
an $i$-pair, and this independently of each other $i$-pair or $j$-pair for $j \ne i$. The following claim is
easy to show from the forcing rules.

\begin{claim} Each $B_{\ii}$ will be generated as an $R$-bimodule by \emph{any} set $Y$ of $2^{d-m}$
\emph{linearly independent} tensors, for which we have a bijection between $Y$ and the power set of $X$,
satisfying the following property: If a tensor $y \in Y$ corresponds to a subset $X_y \subset X$, then each
$i$-pair in $X_y$ corresponds to a distinct linear factor of $y$ which is either $\xi$ or $\xp$ somewhere
inside the $i$-pair. \end{claim}

It's a mouthful, but an example clarifies it. The following is an example of a set of generators for
$B_{(i-1)(i+1)i(i+1)(i-1)i}$, where $d-m=3$ so we need $8$ generators:

\igc{.8}{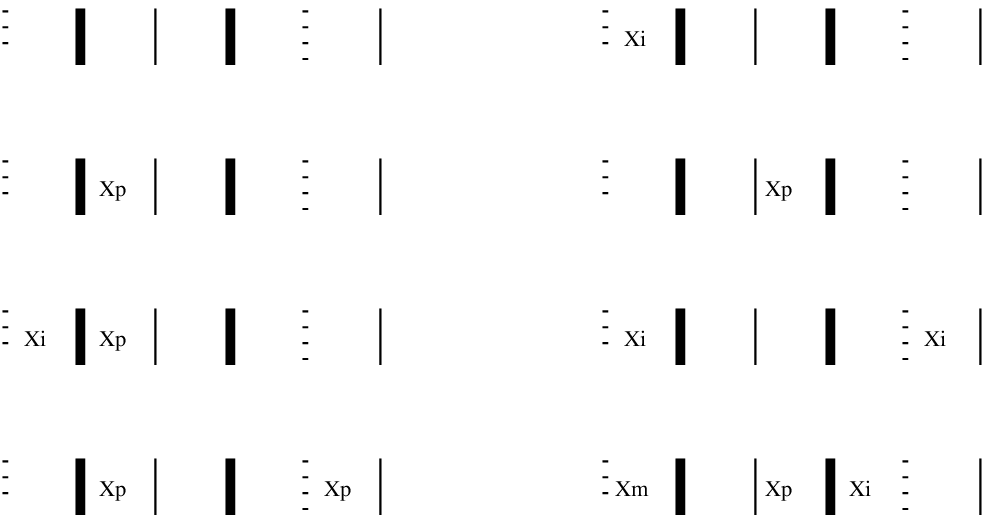}

Each picture represents a tensor, where by convention a blank area is filled with a 1-tensor. Reading across,
the first picture corresponds to the empty subset of $X$. The second picture corresponds to the $i-1$-pair, the
third to the $i+1$-pair, and the fourth to the $i$-pair. Since these three are linearly independent, they take
care of all the linear generators. Clearly the fourth picture could have also worked for the $i+1$-pair, but if
we had chosen it as such, we would have had to choose a new linearly independent vector for the $i$-pair. The
fifth generator corresponds to the $i-1$-pair and the $i+1$-pair, the sixth generator to the $i-1$-pair and the
$i$-pair, and the seventh generator to the $i$-pair and the $i+1$-pair. These three are also linearly
independent. The final generator corresponds to the entire set $X$.

A clever choice of which generators to use may greatly simplify a calculation by reducing the number
of terms in intermediate steps. There are two main instances when this occurs. Either 6-valent vertex
with strands $i$ and $i+1$ will send $1 \otimes \xp \otimes 1 \otimes 1$ or $1 \otimes 1 \otimes \xp
\otimes 1$ to a single tensor, while it may send $1 \teni \xi \tenp 1 \teni 1$ to the sum of two
tensors, thus doubling the work we need to do in the remainder of the calculation. Also, $\xi\xp$
entering a 6-valent vertex $i(i+1)i$ in the second slot may by moved across the $i$-strand to the
left for a simple calculation, while $\xi^2$ leads to a more complicated solution.

The second trick will be a useful diagrammatic way of evaluating homomorphisms in $\mc{SC}_1$, which really
only works well when applied to a well-chosen set of generators. In particular, the choice of generators above
makes the verification of triple overlap associativity (\ref{eqn-threeColorAss}) rather straightforward. We
demonstrate the graphical method in the most difficult case, the highest degree generator.

\begin{center} $\ig{.8}{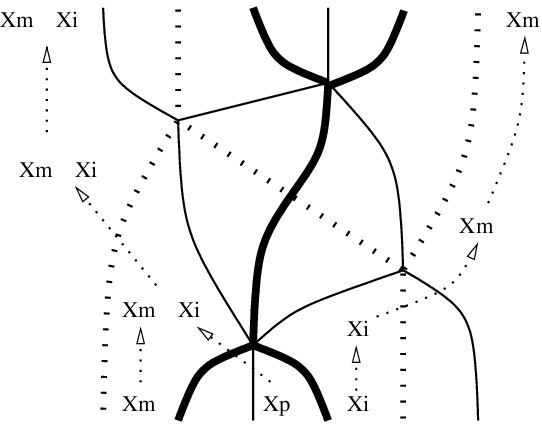}$ \end{center} 

\begin{center} $\ig{.8}{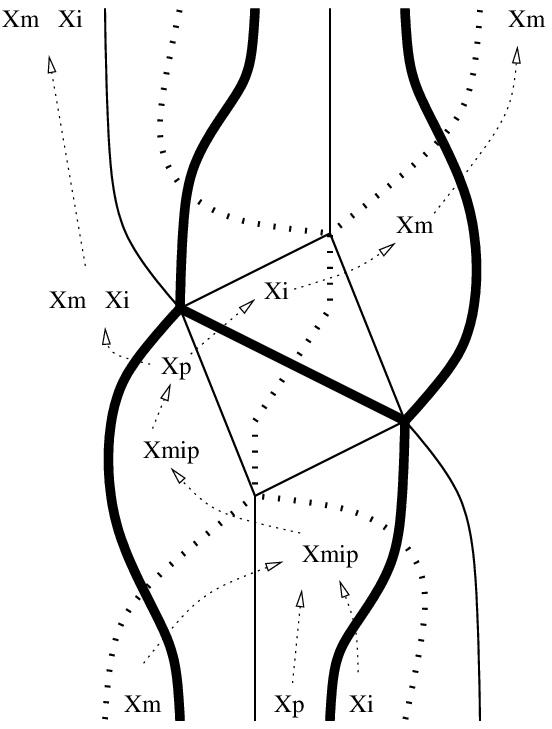}$ \end{center} 

Here $x'=x_{i-1}x_ix_{i+1}$. 

We keep track of the image at every stage in the calculation, which is one term as shown, with all blank spaces
being filled by 1-tensors. An arrow indicates either bringing a symmetric polynomial through a line, or
applying a 6-valent vertex to a tensor. The lack of an arrow indicates that a 1-tensor is sent to a 1-tensor.
This works well only because every intermediate term is a single tensor, not a sum of tensors.

We leave it as an exercise to the reader to verify, using this graphical method of calculation, that both sides
of the triple overlap associativity relation, as displayed above, send the other 7 generators to the same
elements. Precisely, both sides send the 1-tensor to a 1-tensor, and the other generators to 1-tensors with
polynomials in various slots: the 2nd and 4th generators to $\xqq$ in the last slot, the third generator to
$\xm$ in the last slot, the 5th to $\xm\xqq$ in the last slot, the 6th to $\xi\xm$ in the first slot, the 7th
to $\xp\xqq$ in the first slot, and the 8th as shown above to $\xi\xm$ in the first slot and $\xm$ in the last
slot. We promise that the graphical method will work for this set of generators.

One can extend this graphical method easily to handle other morphisms. As an example, we show the
Merge map below.

\begin{center} $\ig{.4}{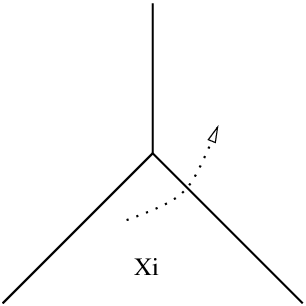} \quad \ig{.4}{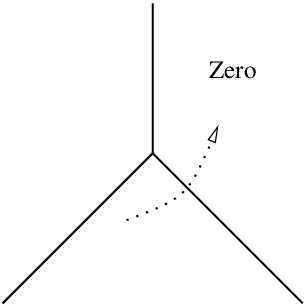}$ \end{center}

When $\xp$ enters a Merge, it is sent to $-1$. Almost the same pictures can be drawn for the Cap. Split and
EndDot send 1-tensors to 1-tensors, and no arrows are needed. 

When there is a cup or a StartDot, a 1-tensor is sent to a sum of two terms, each of which must be evaluated
separately. However, the sum can be written as either $\xi \teni 1 - 1 \teni \xp$ or $1 \teni \xi - \xp \teni
1$, so we may choose the one whichever is more convenient. Often the problem of having two terms is very
temporary. For example, consider what happens to the 1-tensor under the map

\igc{.3}{cupCap.eps}

The cup creates two terms, the first with a 1 under the cap and the second with $\xi$ under the cap. But
the first term is annihilated immediately by the cap, and the cap eats the $\xi$ from the other
term, to return back a 1-tensor. So long as a cup or a StartDot appears right next to a cap or a
merge, one of the two terms is always immediately annihilated. This is the case for
(\ref{biadjoint}), (\ref{twistMerge}), (\ref{twistDot1}), which all send 1-tensors to 1-tensors as a
result. After a quick polynomial slide, the same is true for (\ref{eqn-ijijRot}), and using the more
convenient choice so the $\xp$ ends up underneath the 6-valent vertex, a quick calculation shows
it for (\ref{eqn-ipipipRot}) as well. We call this the \emph{cupcap trick}.

We now list and prove the necessary relations, following the same order as in Section \ref{subsec-fulldefn}.
One could print out this list and hold it next to that section, where the relations are and the functor is
defined, to make this ordeal easier.

\begin{itemize}
\item Biadjointness (\ref{biadjoint}) follows from the cupcap trick.
\item Trivalent twisting (\ref{twistMerge}) follows from the cupcap trick.
\item Associativity (\ref{assoc}) follows from the 1-tensor trick.
\item Dot twisting (\ref{twistDot1}) follows from the cupcap trick.
\item The needle consists of a split, sending a 1-tensor to a 1-tensor, and a cap, killing a 1-tensor. Hence the needle relation (\ref{needle}) follows.
\item The polynomial slide relations clearly hold for Soergel bimodules, by definition.
\item For the broken line relation (\ref{dotSpaceDot}), both sides clearly send the 1-tensor to $\xi \teni 1 - 1 \teni \xp$.
\item For the double dot relation (\ref{doubleDot}), both sides clearly send the 1-tensor to $\xi - \xp$. 
\item The three-line decomposition relation (\ref{eqn-threeLines}) is a calculation. We need to show it for both color variants, whether thin is $i$ and thick $i+1$ or vice versa. Let $w$ be the 1-tensor and $x=1 \teni \xp \teni 1 \teni 1$, regardless of which variant we are in. Then it is easy to check that
$$
\begin{array}{ccc}
\ig{.2}{threeLines1.eps}=\ig{.2}{threeLines4.eps} - \ig{.2}{threeLines3.eps} & \auptob{w}{w} & 
\auptob{x}{x}
\end{array}
$$
Consider what happens to $x$, under the first color variant. The double 6-valent vertex sends $x$ to $1 \teni 1 \tenp 1 \teni \xqq = 1 \teni 1 \tenp \xqq \teni 1$, while the rightmost picture sends $x$ to $1 \teni \xp \tenp 1 \teni 1 - 1 \teni 1 \tenp \xqq \teni 1$. The sum of these two is just $x$ again. We leave similarly easy calculations to the reader in the future.
\item 6-valent twisting (\ref{eqn-ipipipRot}) follows from the cupcap trick.
\item Adding a dot to a 6-valent vertex (\ref{eqn-ipipipDot}) needs to be checked for both color variants. Define $w$ and $x$ as before, and it is easy to check that
$$
\begin{array}{ccc}
\igv{.2}{ipipipDot1.eps} = \igv{.2}{ipipipDot2.eps} + \igv{.2}{ipipipDot3.eps} & \auptob{w}{w} & 
\auptob{x}{1 \teni 1 \tenp 1 \teni y}
\end{array}
$$
Here $y=\xi$ for one color variant, $y=\xqq$ for the other.
\item Two-color associativity (\ref{eqn-ipipipAss}) is a calculation. We recommend using the following twist 
\begin{equation*} \ig{.3}{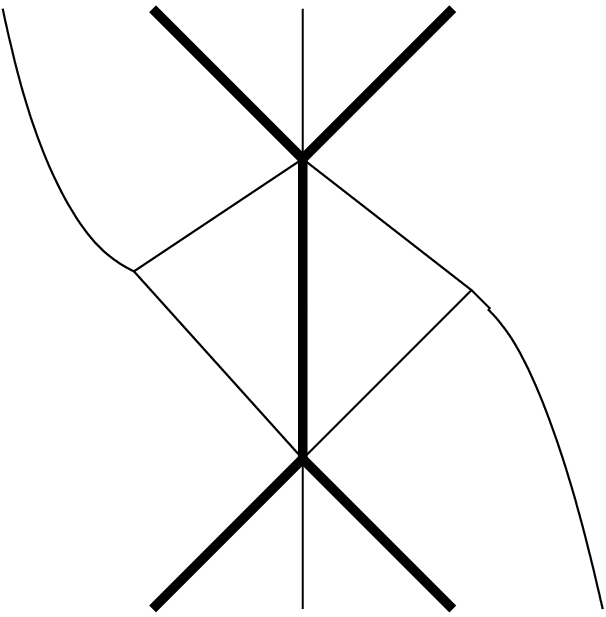} = \ig{.3}{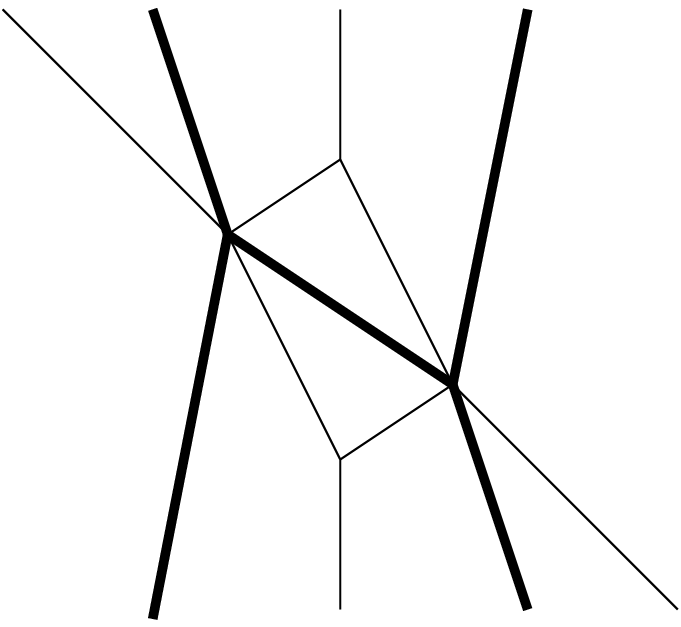}\end{equation*}
and using the four generators: a 1-tensor, $1 \tenp 1 \teni \xp \tenp 1 \teni 1$, $1 \tenp \xp \teni 1 \tenp 1 \teni 1$ and $1 \tenp \xp \teni 1 \tenp \xp \teni 1$. Evaluate using the graphical calculus. The first two generators are killed by both maps thanks to merge maps. The third generator is send to a 1-tensor by both maps. The final generator is also killed by both maps, as symmetric polynomials are slid out of the way and merge maps eat the remaining 1-tensor.
\item 4-valent twisting (\ref{eqn-ijijRot}) follows from the cupcap trick.
\item Sliding a dot or a trivalent vertex through a distant line both follow from the 1-tensor trick.
\item Sliding a 6-valent vertex through a distant line is easily checked on both generators.
\item Sliding a 4-valent vertex through a distant line follows from the 1-tensor trick.
\item Three-color associativity will be a calculation, using the generators and pictures described above, that we leave to the reader.
\end{itemize}

%
c\subsection{Graphical proofs}
\label{subsec-graphproof}
%

In this section, we provide graphical proofs for a series of propositions which collectively prove Proposition
\ref{prop-imoves}. Recall the definitions of the strict and weak $i$-colored moves from Section
\ref{subsec-coloredmoves}. Remember that an extremal color $i$ is one which appears in a graph, but for which
either $i-1$ or $i+1$ does not appear. Each of the results states what reduction moves the relations of
$\mc{DC}_1$ allow. Because more than three colors may be required for some proofs, we use extra line styles
that designate other arbitrary colors, and are very explicit about what colors they are allowed to be. A thin
line will still always represent $i$ and a thick line $i+1$, but we will be lax about other colors.

\begin{lemma} \label{lem-pullThru} One may \emph{strictly} pull a distant line under any other graph, like in
the relations (\ref{eqn-R2move})-(\ref{eqn-R3move}). That is, these relations can be viewed as graph reduction
moves. \end{lemma}

\begin{proof} The only significant part of this lemma is that we can still apply the relations mentioned, even
in the presence of arbitrary polynomials in each region. However, it is an immediate observation that any
polynomial inside a region bounded by at least two lines of distant colors may be slid entirely out of the
region, which does not affect the underlying graph. \end{proof}

The above lemma highlights the fact that we must deal with polynomials when they appear. However, the only
significant polynomials are those which appear in internal regions (regions not touching the boundary of the
planar strip) since all our computations will be local, so polynomials in external regions can be moved outside
the calculation. We will mention whenever a polynomial is relevant, and the reader can check that whenever we
do not mention it, there are only external regions.

\begin{prop} One may apply the (strict) $i$-colored double dot removal move to any graph $\Gamma$. \end{prop}

\begin{proof} Suppose the $i$-graph of $\Gamma$ contains two dots connected by an edge. Then in
  $\Gamma$, the ``edge'' connecting them can only meet series of 4-valent vertices with various
  $j$-strands for $j$ distant from $i$. Since dots can be slid across distant-colored strands by
  (\ref{eqn-ijijDot}), we may slide the dots until they are connected directly by an edge, and just
  use (\ref{doubleDot}) to reduce the graph.
\end{proof}

\begin{prop} One may apply the weak $i$-colored dot contraction move to any graph $\Gamma$. \end{prop}

\begin{proof} Suppose the $i$-graph of $\Gamma$ contains a trivalent vertex connected to a dot. In $\Gamma$,
the trivalent vertex is either a trivalent or 6-valent vertex $v$, and the edge connected the dot to the vertex
may have numerous 4-valent vertices. Again, in $\Gamma$ the dot may be slid across distant-colored strands
until it is connected directly to $v$. If $v$ is trivalent then (\ref{unit}) allows strict dot contraction,
while if $v$ is 6-valent, (\ref{eqn-ipipipDot}) sends $\Gamma$ to the sum of two graphs allowed by weak dot
contraction. \end{proof}

\begin{prop} One may apply the (strict) $i$-colored dot extension move to any graph $\Gamma$. \end{prop}

\begin{proof} This is trivial, by (\ref{unit}). \end{proof}

The following lemma is needed to prove the remainder of the basic $i$-colored moves.

\begin{lemma} Let $\Gamma$ be a sequence of lines $\jj$ (i.e. the identity map of $\jj$), and let $i$ be an
extremal color in $\jj$. Then $\Gamma$ is equal in $\mc{DC}_1$ to a sum of idempotents, such that each
idempotent factors through a sequence of lines $\kk$ where $i$ appears no more than once, and
such that the idempotents do not introduce any new colors not present in $\jj$. \end{lemma}

\begin{proof} This proof is akin to the combinatorial argument that the relations defining the Hecke
  algebra or the symmetric algebra are enough to take any complicated word in $b_j$ or $s_j$ to
  a reduced form.  We use induction on the length of $\jj$.  If $k(k+1)k$ appears in $\jj$
  for some $k$, then using (\ref{eqn-threeLines}), we may factor through $(k+1)k(k+1)$ instead,
  modulo morphisms that factor through shorter length sequences.  If any color appears twice
  consecutively in $\jj$ then we may apply (\ref{twoLines}) to replace the two adjacent lines with
  an ``H'', which factors through a sequence of lines of shorter length. If $jk$ appears for $j$
  distant from $k$, then applying the R2 move we can factor through $kj$.  None of these procedures
  added any new colors. So if we consider the sequence $\jj$ as a word in the symmetric group
  $s_{j_1}s_{j_2}\ldots s_{j_d}$, then any non-reduced words will reduce to smaller length
  sequences, and any two reduced words for the same element will factor through each other, modulo
  smaller length sequences.

  We may now use the easily observable fact that any element of the symmetric group can be represented by a
  reduced word which uses $s_1$ at most once, or $s_n$ at most once (although not both!). It follows from this
  that any element in a parabolic subgroup of $S_{n+1}$ can be represented by a reduced word which uses an
  extremal index at most once. Thus, by the process above, we can replace $\jj$ by idempotents factoring
  through reduced expressions $\kk$ which use $i$ at most once. \end{proof}

The usefulness of this lemma can be seen in the following proof.

\begin{prop} One may apply the (strict) $i$-colored connecting move to any graph $\Gamma$, so long as
  $i$ is extremal.
\end{prop}

\begin{proof} Assume $i-1$ does not appear, although the $i+1$ case is analogous. Induct on the
  number of colors present in the graph: the base case of one color is clear. Then we are attempting
  to apply the $i$-colored connecting move to some region of the graph which looks like this:

  \igc{.4}{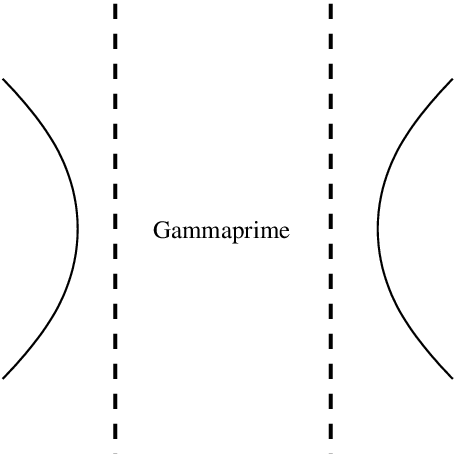}

  $\Gamma'$ is some graph composed entirely of the colors $1, \ldots, i-2$ and $i+1, \ldots ,n$, so that $i+1$
  is extremal or non-existent. In some neighborhood $\Gamma'$ is just a sequence of lines $\1_{\jj}$. By the
  previous lemma, we may replace this sequence of lines with a sum of idempotents factoring through sequences
  $\kk$ where $i+1$ appears at most once, and $i-1$ appears not at all. Localizing to a neighborhood around
  $\1_{\kk}$, we may as well assume that $\jj=\kk$. Now apply the R2 move to bring the $i$-strands past all the
  other colors so that they are separated by either nothing or a single $i+1$ strand. We have not yet altered
  the $i$-graph at all. Then equation (\ref{twoLines}) or equation (\ref{eqn-threeLines}) (respectively) will
  allow the strict $i$-colored connecting move. \end{proof}

When we have an edge in an $i$-graph, a neighborhood of that edge in the full graph will be nothing more than a
sequence of 4-valent vertices as the $i$-strand crosses distant strands. The next corollary allows us to place
restrictions on which distant strands the $i$-strand need cross.

\begin{cor} \label{cor-between} Suppose we are given a graph $\Gamma$ which is a neighborhood of an $i$-colored
strand consisting only of 4-valent vertices crossing the $i$-strand (see the picture below). Then $\Gamma$ is
equivalent in $\mc{DC}_1$ to a linear combination of graphs $\Gamma^\prime$ which have the same $i$-graph, and
for which the sequence of lines crossing the $i$-strand in $\Gamma^\prime$ contains $i+2$ and $i-2$ each at
most once. \end{cor}

  $$\Gamma \ \ = \ \ \ \  \ig{.4}{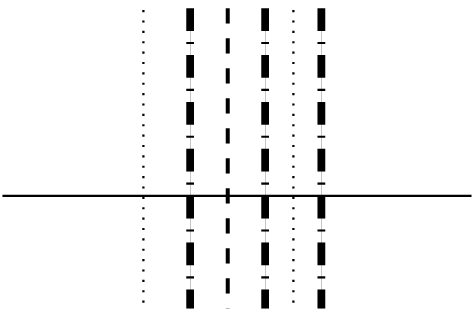}$$

\begin{proof}

 Let $\jj$ be the sequence of lines crossing the $i$-strand in $\Gamma$; clearly $i$, $i-1$, and $i+1$ do not
appear in $\jj$. According to the above lemma, we may replace $\1_{\jj}$ with a sum of idempotents which factor
through ``nice'' sequences $\kk$ where $i+2$ and $i-2$ each appear at most once. Moreover, for any idempotent
appearing in this sum, $i$ is distant from every color appearing in the idempotent (as well as the polynomials
which may appear in the idempotent). So replacing $\1_{\jj}$ with this sum immediately above the $i$-strand, we
may then slide the $i$-strand in each idempotent so that it passes through $\kk$. An example is given below,
ignoring polynomials which can also be slid.

$\ig{.4}{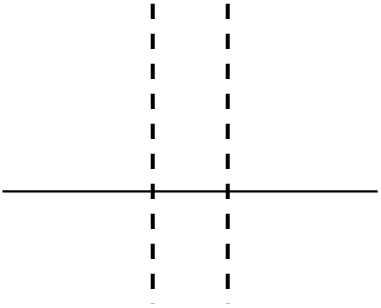} \lra \ig{.4}{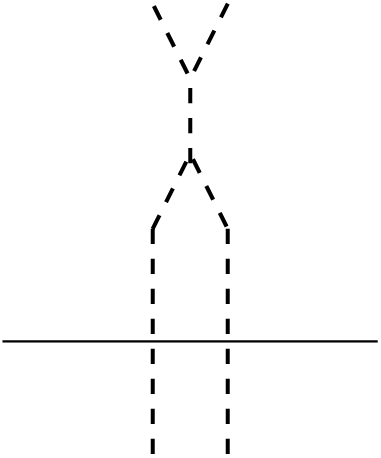} \lra \ig{.4}{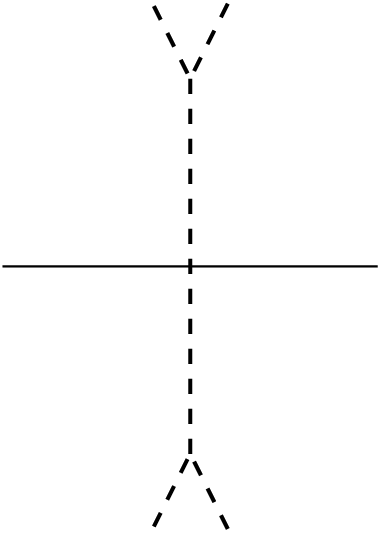}$

\end{proof}

The last two moves will take the rest of the paper to prove.

\begin{prop} One may apply the (strict) $i$-colored needle move to any graph $\Gamma$.\end{prop}

\begin{prop} One may apply the weak $i$-colored associativity move to any $i$-colored ``H'' in any graph
  $\Gamma$, so long as the ``H'' does not look like the following picture:

  \igc{.4}{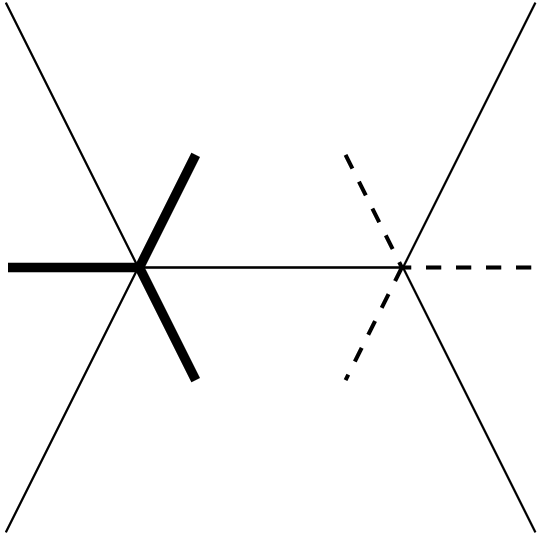}

  where the two vertices are 6-valent in $\Gamma$ with different colors $i+1$ and $i-1$.

  In particular, we may apply the $i$-colored associativity move to any graph $\Gamma$ when $i$ is extremal.
\end{prop}

\begin{proof}[Proof Setup] We apply induction on the set of colors present in the graph to prove these two
propositions simultaneously. Both propositions hold when $i$ is the only color in the graph, by
(\ref{eqn-associativity}), (\ref{needle}), (\ref{holepoly}), where the latter is needed because an arbitrary
polynomial may be within the eye of the needle. The inductive hypothesis then implies, with the previous
propositions, that for any graph with fewer colors, we may apply \emph{all} the weak $k$-colored basic moves
for an extremal color $k$. In this case Proposition \ref{prop-weakonecolor} says that we may reduce the
$k$-graph to a disjoint union of simple trees. Within each region of $\Gamma$ delimited by the $i$-graph the
color $i$ is absent, so there are fewer total colors and we may reduce both the $i-1$-graph and the $i+1$-graph
inside this region to a union of simple trees. This application of the inductive hypothesis was the only reason
to consider the two propositions together; we now treat them separately. \end{proof}

\begin{proof}[Needle move, modulo induction hypothesis] Suppose the $i$-colored graph contains a needle. The
trivalent vertex of the needle is either trivalent or 6-valent in $\Gamma$, and the edge of the needle may
contain a series of 4-valent vertices. So a neighborhood of the needle in $\Gamma$ looks like one of the
pictures below.

\begin{center} $\ig{.5}{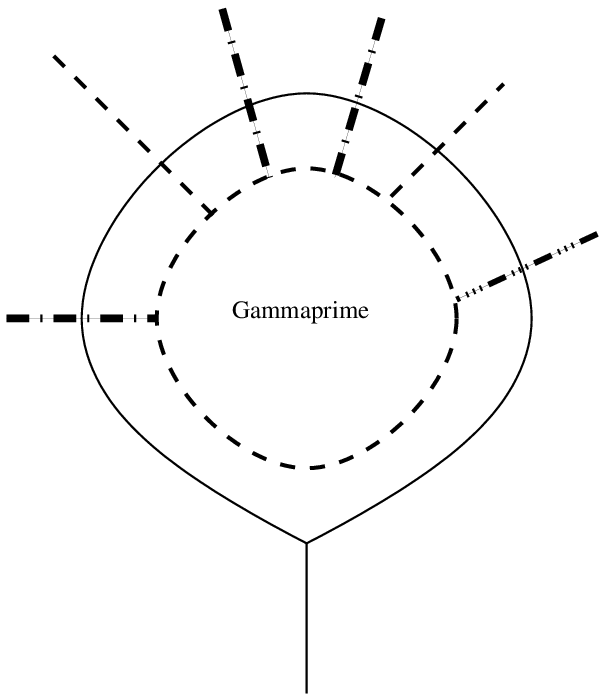}$ OR $\ig{.5}{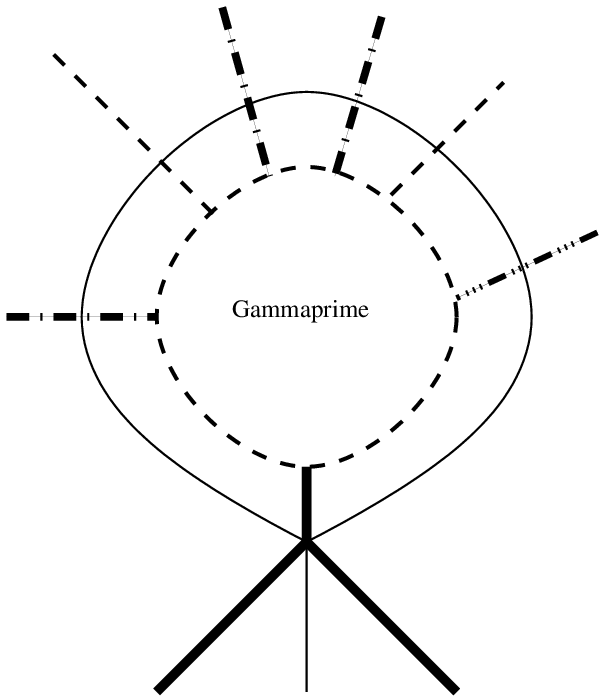}$\end{center}

The lines entering the needle around the top are colored with various $j$ all distant from $i$. We assume
without loss of generality that, in the second case, the 6-valent vertex has second color $i+1$.

We may now reduce the $i+1$ and $i-1$ graph of $\Gamma'$, since $\Gamma'$ does not contain the color $i$. In
the first case, the $i \pm 1$ graph has no boundary, so both reduce to the empty graph. In the second case,
$i+1$ has a single boundary line so it reduces to a boundary dot, and $i-1$ has none so it reduces to the empty
graph. Within the reduced $\Gamma'$, the possible $i+1$ dot may be slid under other lines until no 4-valent
vertices separate it from the 6-valent vertex (as in the proof of weak dot contraction). Thus we may assume
that, after reduction, our neighborhood looks like

\begin{center}$\ig{.5}{proof3.eps}$ OR $\ig{.5}{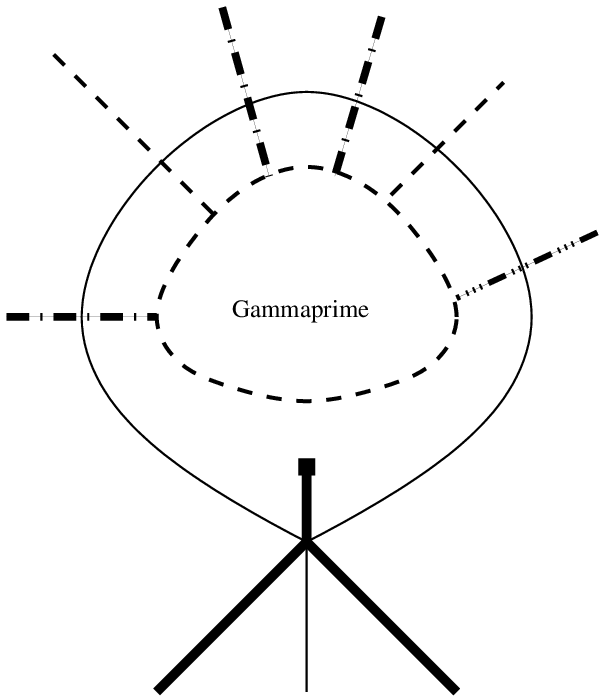}$\end{center}
where now $\Gamma'$ does not contain any of the colors $i-1,i,i+1$. Note that $\Gamma'$ may have
arbitrary polynomials in its various regions. But then by Lemma \ref{lem-pullThru} we can pull the line forming
the needle through all of $\Gamma'$, thus completely ignoring $\Gamma'$ from the picture! There still may be a
polynomial in the eye of the needle. We have effectively reduced to the 2-color case on the right, or the one
color case on the left. We know the one color case works. To check the 2-color case, we use
(\ref{eqn-ipipipDot}) followed by other reduction moves.

\begin{center} $ \ig{.3}{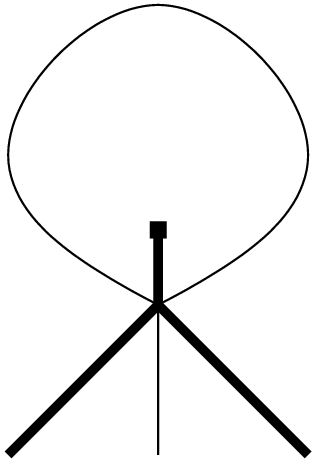}\ \ \to \ \ \ig{.3}{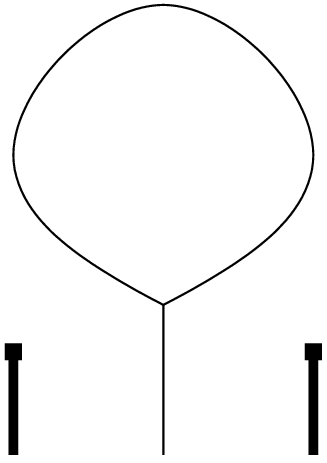} \ {\rm or} \ \ig{.3}{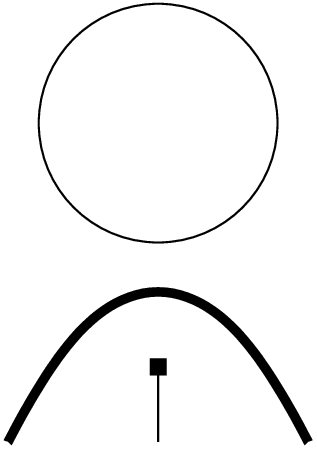} \ \ \to \
\ \ig{.3}{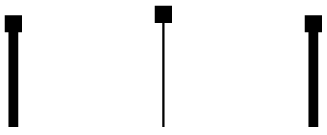} \ {\rm or} \ \ig{.3}{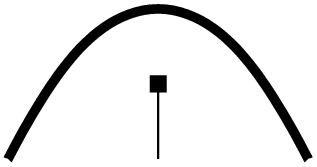}$ \end{center}

In these final graphs, the $i$-graph is a dot, as desired. So we may apply the strict $i$-colored needle move.
\end{proof}

\begin{proof}[Associativity move, modulo induction hypothesis] We would like to apply associativity to the
following subgraph of the $i$-graph.

\igc{.5}{Xdiagram.eps}

Each trivalent vertex of the $i$-graph may be either trivalent or 6-valent in $\Gamma$. We are forbidding the
case when one is 6-valent with $i+1$ and the other with $i-1$, so without loss of generality with assume that a
6-valent vertex has $i+1$ as the other color. The neighborhood of this subgraph in $\Gamma$ looks like one of
the following cases.

\begin{center} $\ig{.4}{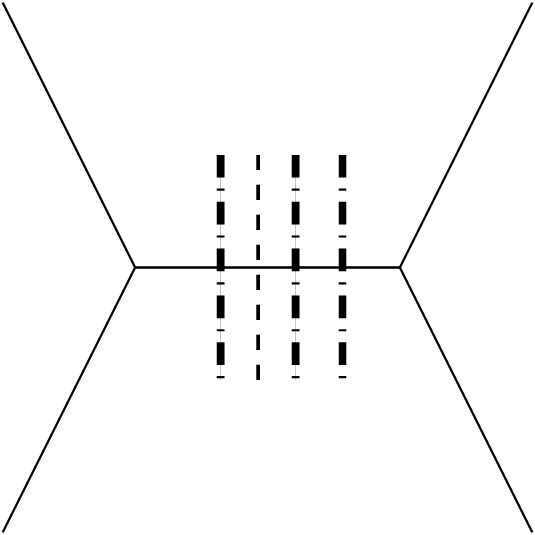}$ OR $\ig{.4}{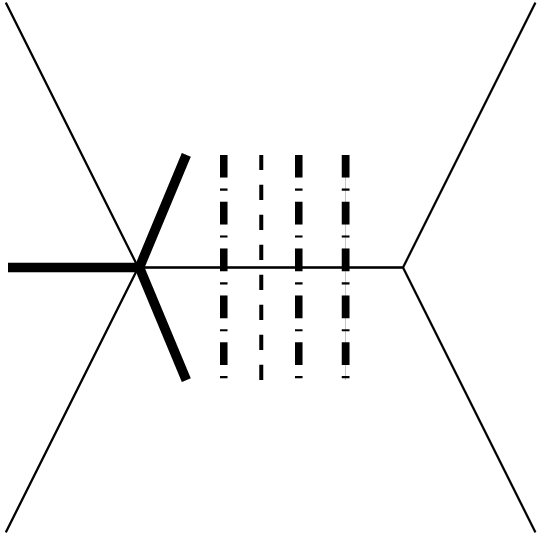}$ OR $\ig{.4}{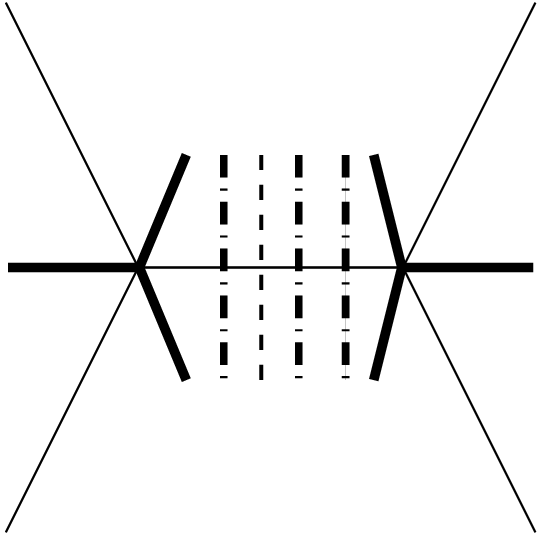}$ \end{center}

All polynomials may be assumed to lie outside the neighborhood. We may use Corollary \ref{cor-between} to
assume that the sequence of lines passing through the middle strand contains at most one instance of $i+2$ and
$i-2$. In the first two cases, all such lines may be slid one by one to the right over the trivalent vertex,
removing them from the neighborhood. In the third case, all lines not labelled $i+2$ can be slid to the right
or to the left, and since there is at most one line labelled $i+2$, then at most one line remains. (If there
had been multiple lines labelled $i+2$, then additional lines may have been stuck between them, but thankfully
Corollary \ref{cor-between} eliminates this possibility.) Four cases remain:

\begin{center} $\ig{.5}{Xdiagram.eps} \quad \ig{.3}{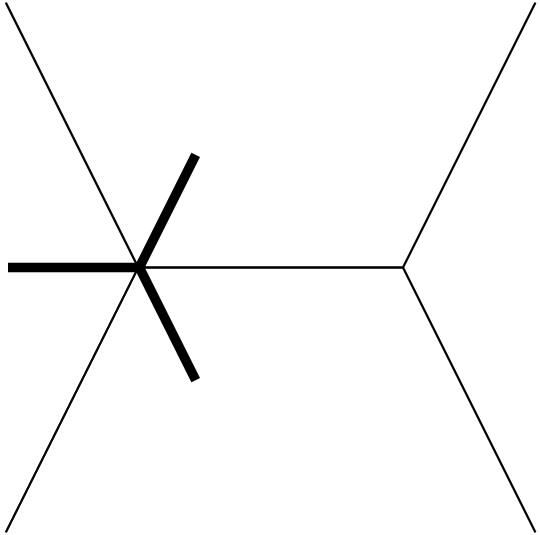} \quad \ig{.3}{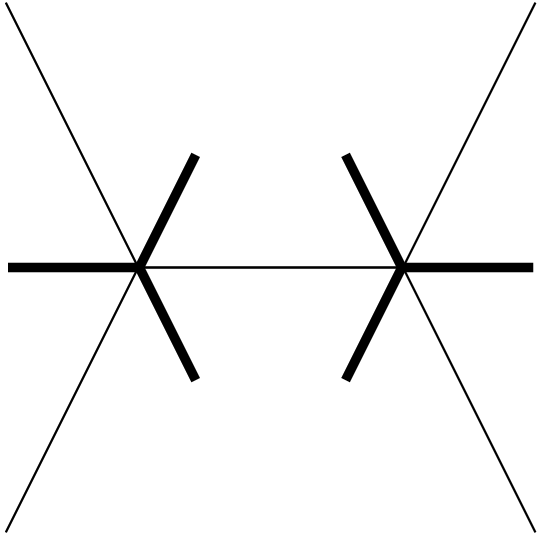} \quad
\ig{.3}{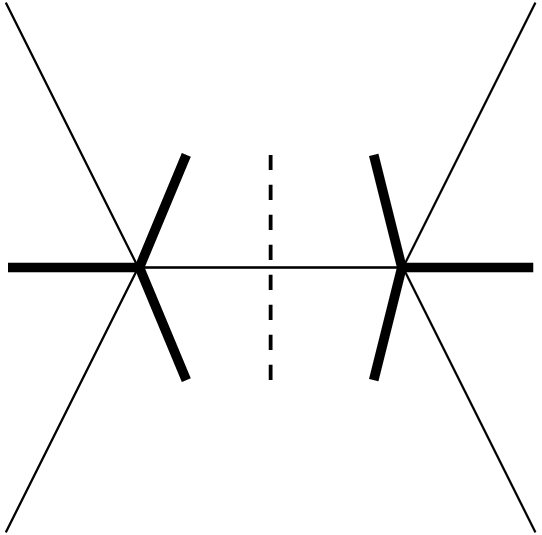}$ \end{center}

In the fourth case, the additional line is $i+2$. So far, the $i$-graph has been unchanged.

Case I: One color associativity allows the strict $i$-colored associativity move.

Case II: Double overlap associativity (\ref{eqn-ipipipAssWDot}) allows the strict $i$-colored associativity
move.

The remaining two cases use the same trick: they replace the interior lines on the top and bottom of the graph
with the corresponding sum of idempotents, and then resolve each one with double or triple overlap
associativity. The remaining two cases will not allow the strict associativity move, only the weak one.

Case III: We rewrite equation (\ref{twoLines}), using (\ref{dotSpaceDot}), so that there is no polynomial on
the bottom.

\begin{equation*} \ig{.4}{twoLines.eps} = \ig{.4}{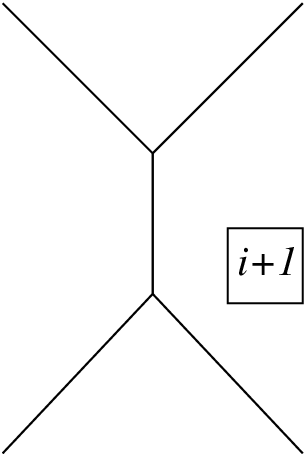} + \ig{.4}{twoLinesAlt2.eps} -
\ig{.4}{twoLinesFinal2.eps} \end{equation*}

Applying this to the thick edges on top of case III, and symmetrically to the bottom, we get a sum of graphs of
two kinds.

For terms which do not involve a dot, we get a graph which looks like the following, with
\emph{no} polynomial in any interior region:

\igc{.4}{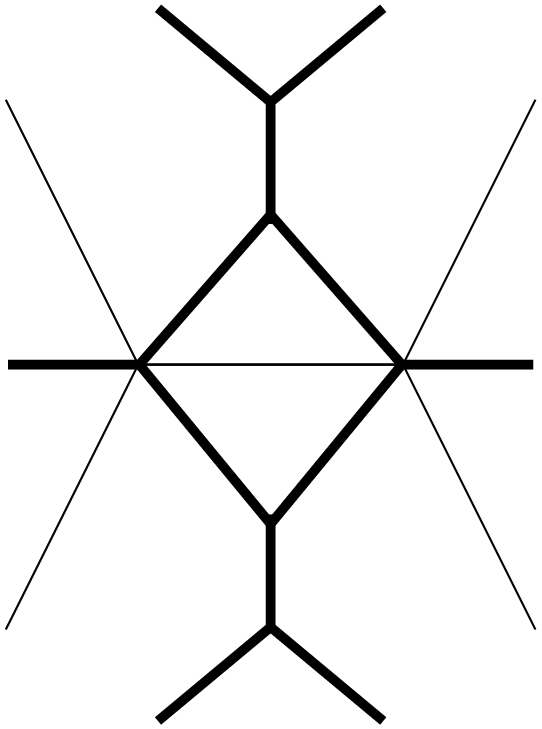}

We may apply double overlap associativity (\ref{eqn-ipipipAss}) to a subgraph of this diagram (ignoring the top
and bottom thick trivalent vertices), which will apply associativity to the $i$-graph.

Consider a term where a dot connects to a 6-valent vertex. Resolving the dot using (\ref{eqn-ipipipDot}), we
get a sum of two further terms: one which looks like Case II, and the other which is one of the alternate
graphs allowed by the weak associativity move.

Thus we may apply weak $i$-colored associativity to each term.

Case IV: Applying equation (\ref{eqn-threeLines}) to the $i+1$,$i+2$,$i+1$ sequence on top of the graph, we get
the sum

\begin{center} $\ig{.4}{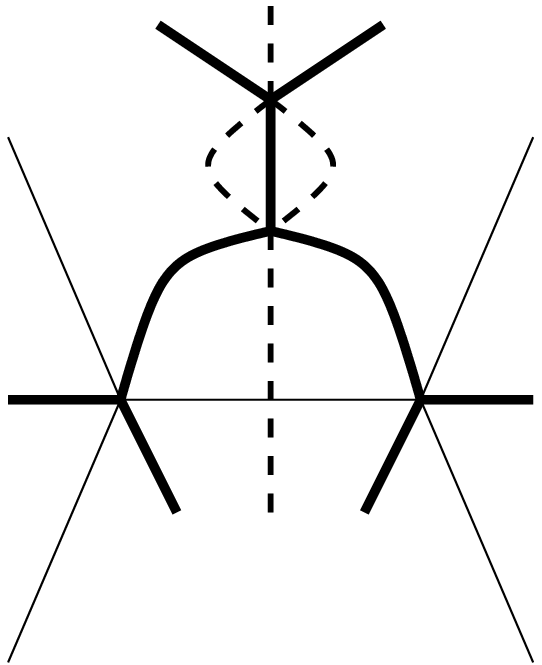} \ \  - \ \  \ig{.4}{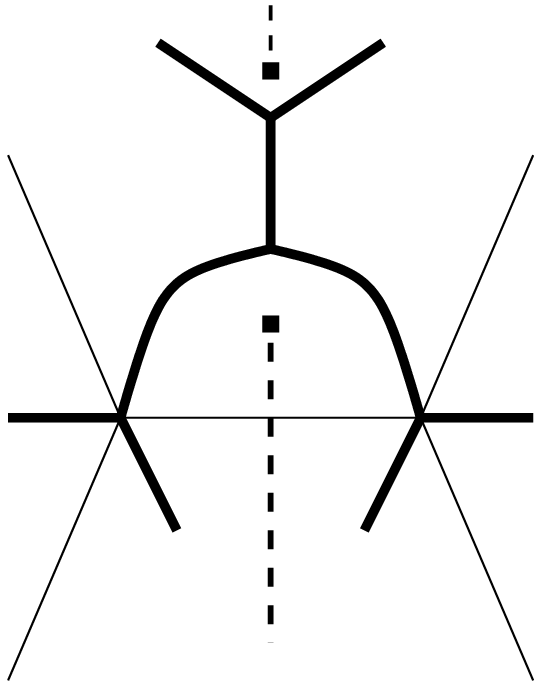}$ \end{center}.

Now we need to show that we can apply associativity to each of these.

For the second picture, we may drag the dot on $i+2$ through the $i$-strand, and then a smaller neighborhood of
the X looks like Case III.

For the first picture, we once again apply (\ref{eqn-threeLines}) to the bottom of the graph to get the sum

\begin{center} $\ig{.4}{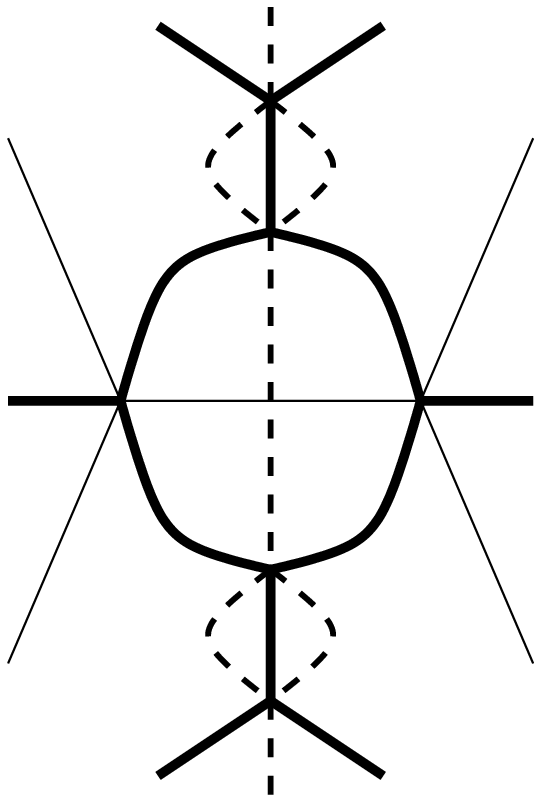} \ \  - \ \  \ig{.4}{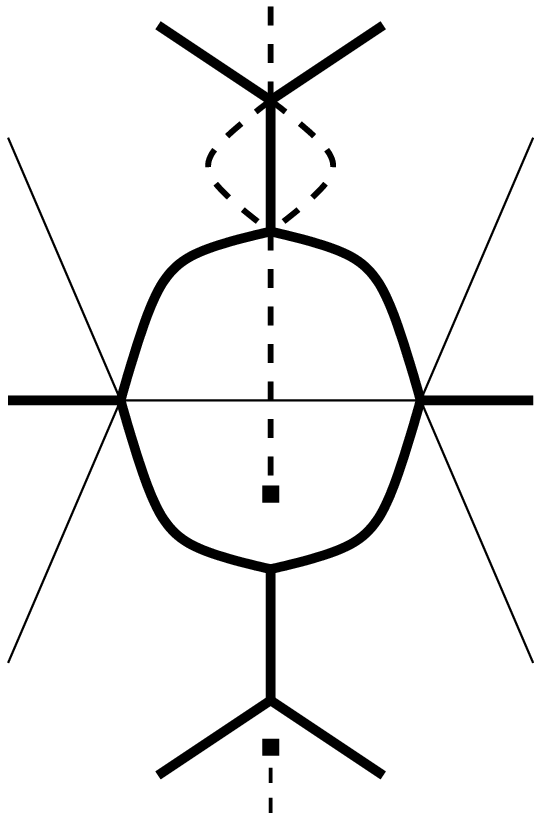}$ \end{center}.

Again, for the second picture we may drag the dot through and reduce using Case III.

For the first picture, there are \emph{no} polynomials in any internal region, because these internal regions
were just created by relations. We may now apply triple overlap associativity (\ref{eqn-threeColorAss}) to the
subgraph of the this picture which ignores the top and bottom 6-valent vertices. This has the effect of
applying weak $i$-colored associativity. \end{proof}

\begin{remark} It does not seem possible to apply $i$-colored associativity to the following graph.
	
\igc{.4}{noAss.eps}

Consider switching the $i$-graph to the other version of an ``H", and look at the triangular region delimited
by the $i$-graph on top. It has at most one 6-valent vertex on its perimeter, so that its boundary lines
consist of an $i-1$ line, an $i+1$ line, and at most one other $i \pm 1$ line. This means that one of the two
colors $i-1$ or $i+1$ must reduce to a boundary dot. It is not possible to obtain a degree 0 morphism if this
happens.

Because of the failure of $i$-colored associativity in this context, our algorithm towards reducing graphs has
always been to treat extremal colors and use induction. Despite this pessimism, Proposition
\ref{prop-colorelimination} (color elimination) implies that we can still do a lot. However, any purely
graphical proof of color elimination is complicated by the failure of $i$-colored associativity for
non-extremal colors. \end{remark}

\vspace{0.1in}
 
\noindent
{\sl \small Ben Elias, Department of Mathematics, Columbia University, New York, NY 10027}

\noindent 
{\tt \small email: belias@math.columbia.edu}

\vspace{0.1in} 

\noindent 
{ \sl \small Mikhail Khovanov, Department of Mathematics, Columbia University, New York, NY 10027}

\noindent
  {\tt \small email: khovanov@math.columbia.edu}

\end{document}